\documentclass[12pt,a4paper,reqno]{amsart}
\title[Hochschild homology of nc symmetric quotient stack]{The Hochschild homology of a noncommutative symmetric quotient stack}
\author{Rina Anno}
\email{ranno@math.ksu.edu}
\address{Department of Mathematics \\
Kansas State University \\
138 Cardwell Hall \\
Manhattan, KS 66506\\
USA}
\author{Vladimir Baranovsky}
\email{vbaranov@math.uci.edu}
\address{Department of Mathematics \\
UC Irvine \\
Irvine, CA 92697-3875 \\
USA}
\author{Timothy Logvinenko} 
\email{LogvinenkoT@cardiff.ac.uk} 
\address{School of Mathematics\\ 
Cardiff University\\
Senghennydd Road,\\
Cardiff, CF24 4AG\\
UK}
\usepackage{amsmath,amsfonts,amssymb,amsthm,epsfig,amscd,latexsym,comment,mathtools,bbold}
\usepackage{enumitem}
\usepackage{caption}
\usepackage{tikz-cd}
\usepackage{graphicx}
\usepackage{array}
\usepackage{subfigure}
\usepackage{leftidx}
\usepackage{xparse}

\usepackage[colorlinks=true, pdfpagemode=none, pdfmenubar=false, linkcolor=blue, citecolor=blue, urlcolor=blue]{hyperref}

\let\amsamp=&

\begingroup
\catcode`\&=13
\gdef\smallampmatrix{%
  \begingroup
  \let&=\amsamp
  \begin{smallmatrix}%
}
\gdef\endsmallampmatrix{\end{smallmatrix}\endgroup}
\endgroup

\DeclareMathOperator{\chr}{char}

\DeclareMathOperator{\obj}{Ob}

\DeclareMathOperator{\homm}{Hom}

\DeclareMathOperator{\eend}{End}

\DeclareMathOperator{\picr}{Pic}

\DeclareMathOperator{\cl}{Cl}

\DeclareMathOperator{\spec}{Spec}

\DeclareMathOperator{\action}{act}

\DeclareMathOperator{\modd}{\bf Mod}
\DeclareMathOperator{\lder}{\bf L}

\DeclareMathOperator{\ldertimes}{\overset{\lder}{\otimes}}

\DeclareMathOperator{\id}{Id}

\DeclareMathOperator{\opp}{{opp}}
\DeclareMathOperator{\fg}{{\it fg}}
\DeclareMathOperator{\qrep}{\it \mathcal{Q}r}
\DeclareMathOperator{\hproj}{\mathcal{P}}

\DeclareMathOperator{\semifree}{\mathcal{S}\mathcal{F}}
\DeclareMathOperator{\sffg}{\mathcal{S}\mathcal{F}_{\fg}}
\DeclareMathOperator{\perf}{{\it \mathcal{P}erf}}
\DeclareMathOperator{\hperf}{{\it h\mathcal{P}erf}}
\DeclareMathOperator{\hmtpy}{{Ho}}

\DeclareMathOperator{\tria}{{Tria}}
\DeclareMathOperator{\pretriag}{{Pre\text{-}Tr}}

\DeclareMathOperator{\TPair}{{\bf TPair}}
\DeclareMathOperator{\alg}{{\bf Alg}}

\DeclareMathOperator{\forget}{{Forg}}

\DeclareMathOperator{\eilmoor}{{\mathcal{E}}}
\DeclareMathOperator{\coeilmoor}{{co\mathcal{E}}}

\DeclareMathOperator{\hochhom}{{HH}}
\DeclareMathOperator{\hochcx}{{HC}}
 
\DeclareMathOperator{\shfl}{{\mathrm{Sh}}} 

\DeclareMathOperator\Sym{Sym}
\DeclareMathOperator\ezmap{{EZ}}
\DeclareMathOperator\awmap{{AW}}
\DeclareMathOperator\Ind{{Ind}}     
\DeclareMathOperator\Res{{Res}}

\begin{document}

\def\bv{\mathbf{v}}
\def\kgc_{K^*_G(\mathbb{C}^n)}
\def\kgchi_{K^*_\chi(\mathbb{C}^n)}
\def\kgcf_{K_G(\mathbb{C}^n)}
\def\kgchif_{K_\chi(\mathbb{C}^n)}
\def\gpic_{G\text{-}\picr}
\def\gcl_{G\text{-}\cl}
\def\trch_{{\chi_{0}}}
\def\regring{{R}}
\def\regrep{{V_{\text{reg}}}}
\def\givrep{{V_{\text{giv}}}}
\def\lbar{{(\mathbb{Z}^n)^\vee}}
\def\genpx_{{p_X}}
\def\genpy_{{p_Y}}
\def\genpcn_{p_{\mathbb{C}^n}}
\def\gnat{gnat}
\def\twalg{{\regring \rtimes G}}
\def\L{{\mathcal{L}}}
\def\gcd{\mbox{gcd}}
\def\lcm{\mbox{lcm}}
\def\tf{{\tilde{f}}}
\def\tD{{\tilde{D}}}
\def\A{{\mathcal{A}}}
\def\B{{\mathcal{B}}}
\def\C{{\mathcal{C}}}
\def\D{{\mathcal{D}}}
\def\F{{\mathcal{F}}}
\def\H{{\mathcal{H}}}
\def\L{{\mathcal{L}}}
\def\M{{\mathcal{M}}}
\def\N{{\mathcal{N}}}
\def\R{{\mathcal{R}}}
\def\T{{\mathcal{T}}}
\def\RF{{\mathcal{R}\mathcal{F}}}
\def\barA{{\bar{\mathcal{A}}}}
\def\barAi{{\bar{\mathcal{A}}_1}}
\def\barAj{{\bar{\mathcal{A}}_2}}
\def\barB{{\bar{\mathcal{B}}}}
\def\barC{{\bar{\mathcal{C}}}}
\def\barD{{\bar{\mathcal{D}}}}
\def\barT{{\bar{\mathcal{T}}}}
\def\barM{{\bar{\mathcal{M}}}}
\def\Aopp{{\A^{\opp}}}
\def\Bopp{{\B^{\opp}}}
\def\Copp{{\C^{\opp}}}
\def\aA{\leftidx{_{a}}{\A}}
\def\bA{\leftidx{_{b}}{\A}}
\def\Aa{{\A_a}}
\def\Ea{E_a}
\def\aE{\leftidx{_{a}}{E}{}}
\def\Eb{E_b}
\def\bE{\leftidx{_{b}}{E}{}}
\def\Fa{F_a}
\def\aF{\leftidx{_{a}}{F}{}}
\def\Fb{F_b}
\def\bF{\leftidx{_{b}}{F}{}}
\def\aM{\leftidx{_{a}}{M}{}}
\def\aN{\leftidx{_{a}}{N}{}}
\def\aMb{\leftidx{_{a}}{M}{_{b}}}
\def\aNb{\leftidx{_{a}}{N}{_{b}}}
\def\rfRFa{\leftidx{_{\RF}}{\RF}{_{\A}}}
\def\aRFrf{\leftidx{_{\A}}{\RF}{_{\RF}}}
\def\biAMA{\leftidx{_{\A}}{M}{_{\A}}}
\def\biAMC{\leftidx{_{\A}}{M}{_{\C}}}
\def\biCMA{\leftidx{_{\C}}{M}{_{\A}}}
\def\biCMC{\leftidx{_{\C}}{M}{_{\C}}}
\def\biALA{\leftidx{_{\A}}{L}{_{\A}}}
\def\biALC{\leftidx{_{\A}}{L}{_{\C}}}
\def\biCLA{\leftidx{_{\C}}{L}{_{\A}}}
\def\biCLC{\leftidx{_{\C}}{L}{_{\C}}}
\def\Na{{N_a}}
\def\modk{{\modd\text{-}\mathbb{k}}}
\def\Amod{{\A\text{-}\modd}}
\def\modA{{\modd\text{-}\A}}
\def\modbar{{\overline{\modd}}}
\def\modbarA{{\overline{\modd}\text{-}\A}}
\def\modbarAopp{{\overline{\modd}\text{-}\Aopp}}
\def\modB{{\modd\text{-}\B}}
\def\modC{{\modd\text{-}\C}}
\def\modD{{\modd\text{-}\D}}
\def\modbarB{{\overline{\modd}\text{-}\B}}
\def\modbarC{{\overline{\modd}\text{-}\C}}
\def\modbarD{{\overline{\modd}\text{-}\D}}
\def\modbarBopp{{\overline{\modd}\text{-}\Bopp}}
\def\AmodA{{\A\text{-}\modd\text{-}\A}}
\def\AmodM{{\A\text{-}\modd\text{-}\M}}
\def\AmodB{{\A\text{-}\modd\text{-}\B}}
\def\AmodT{{\A\text{-}\modd\text{-}\T}}
\def\BmodB{{\B\text{-}\modd\text{-}\B}}
\def\BmodA{{\B\text{-}\modd\text{-}\A}}
\def\DmodD{{\D\text{-}\modd\text{-}\D}}
\def\MmodA{{\M\text{-}\modd\text{-}\A}}
\def\MmodM{{\M\text{-}\modd\text{-}\M}}
\def\TmodA{{\T\text{-}\modd\text{-}\A}}
\def\TmodT{{\T\text{-}\modd\text{-}\T}}
\def\AmodbarA{\A\text{-}{\overline{\modd}\text{-}\A}}
\def\AmodbarB{\A\text{-}{\overline{\modd}\text{-}\B}}
\def\AmodbarC{\A\text{-}{\overline{\modd}\text{-}\C}}
\def\AmodbarD{\A\text{-}{\overline{\modd}\text{-}\D}}
\def\AmodbarM{\A\text{-}{\overline{\modd}\text{-}\M}}
\def\AmodbarT{\A\text{-}{\overline{\modd}\text{-}\T}}
\def\BmodbarA{\B\text{-}{\overline{\modd}\text{-}\A}}
\def\BmodbarB{\B\text{-}{\overline{\modd}\text{-}\B}}
\def\BmodbarC{\B\text{-}{\overline{\modd}\text{-}\C}}
\def\BmodbarD{\B\text{-}{\overline{\modd}\text{-}\D}}
\def\CmodbarA{\C\text{-}{\overline{\modd}\text{-}\A}}
\def\CmodbarB{\C\text{-}{\overline{\modd}\text{-}\B}}
\def\CmodbarC{\C\text{-}{\overline{\modd}\text{-}\C}}
\def\CmodbarD{\C\text{-}{\overline{\modd}\text{-}\D}}
\def\DmodbarA{\D\text{-}{\overline{\modd}\text{-}\A}}
\def\DmodbarB{\D\text{-}{\overline{\modd}\text{-}\B}}
\def\DmodbarC{\D\text{-}{\overline{\modd}\text{-}\C}}
\def\DmodbarD{\D\text{-}{\overline{\modd}\text{-}\D}}
\def\TmodbarA{\T\text{-}{\overline{\modd}\text{-}\A}}
\def\MmodbarA{\M\text{-}{\overline{\modd}\text{-}\A}}
\def\MmodbarM{\M\text{-}{\overline{\modd}\text{-}\M}}
\def\modbarT{{\overline{\modd}\text{-}T}}
\def\sfA{{\semifree(\A)}}
\def\sfB{{\semifree(\B)}}
\def\sffgA{{\sffg(\A)}}
\def\sffgB{{\sffg(\B)}}
\def\hprojA{{\hproj(\A)}}
\def\hprojB{{\hproj(\B)}}
\def\qrepA{{\qrep(\A)}}
\def\qrepB{{\qrep(\B)}}
\def\opp{{\text{opp}}}
\def\hperfA{{\hperf(\A)}}
\def\hperfB{{\hperf(\B)}}
\def\barperf{{\it\mathcal{P}\overline{er}f}}
\def\barperfA{{\barperf(\A)}}
\def\barperfB{{\barperf(\B)}}
\def\qrhpr{{\hproj^{qr}}}
\def\qrhprA{{\qrhpr(\A)}}
\def\qrhprB{{\qrhpr(\B)}}
\def\qrsf{{\semifree^{qr}}}
\def\qrsf{{\semifree^{qr}}}
\def\qrsfA{{\qrsf(\A)}}
\def\qrsfB{{\qrsf(\B)}}
\def\Aperfsf{{\semifree^{\A\text{-}\perf}(\AbimB)}}
\def\Bperfsf{{\semifree^{\B\text{-}\perf}(\AbimB)}}
\def\Aprfhpr{{\hproj^{\A\text{-}\perf}(\AbimB)}}
\def\Bprfhpr{{\hproj^{\B\text{-}\perf}(\AbimB)}}
\def\Aqrhpr{{\hproj^{\A\text{-}qr}(\AbimB)}}
\def\Bqrhpr{{\hproj^{\B\text{-}qr}(\AbimB)}}
\def\Aqrsf{{\semifree^{\A\text{-}qr}(\AbimB)}}
\def\Bqrsf{{\semifree^{\B\text{-}qr}(\AbimB)}}
\def\modAopp{{\modd\text{-}\Aopp}}
\def\modBopp{{\modd\text{-}\Bopp}}
\def\AmodA{{\A\text{-}\modd\text{-}\A}}
\def\AmodB{{\A\text{-}\modd\text{-}\B}}
\def\AmodC{{\A\text{-}\modd\text{-}\C}}
\def\BmodA{{\B\text{-}\modd\text{-}\A}}
\def\BmodB{{\B\text{-}\modd\text{-}\B}}
\def\BmodC{{\B\text{-}\modd\text{-}\C}}
\def\CmodA{{\C\text{-}\modd\text{-}\A}}
\def\CmodB{{\C\text{-}\modd\text{-}\B}}
\def\CmodC{{\C\text{-}\modd\text{-}\C}}
\def\AbimA{{\A\text{-}\A}}
\def\AbimC{{\A\text{-}\C}}
\def\AbimM{{\A\text{-}\M}}
\def\BbimA{{\B\text{-}\A}}
\def\BbimB{{\B\text{-}\B}}
\def\BbimC{{\B\text{-}\C}}
\def\BbimD{{\B\text{-}\D}}
\def\CbimA{{\C\text{-}\A}}
\def\CbimB{{\C\text{-}\B}}
\def\CbimC{{\C\text{-}\C}}
\def\DbimA{{\D\text{-}\A}}
\def\DbimB{{\D\text{-}\B}}
\def\DbimC{{\D\text{-}\C}}
\def\DbimD{{\D\text{-}\D}}
\def\MbimA{{\M\text{-}\A}}
\def\AhprA{{\hproj\left(\AbimA\right)}}
\def\BhprB{{\hproj\left(\BbimB\right)}}
\def\AhprB{{\hproj\left(\AbimB\right)}}
\def\BhprA{{\hproj\left(\BbimA\right)}}
\def\AbarA{{\overline{\A\text{-}\A}}}
\def\AbarB{{\overline{\A\text{-}\B}}}
\def\BbarA{{\overline{\B\text{-}\A}}}
\def\BbarB{{\overline{\B\text{-}\B}}}
\def\QAbimB{{Q\A\text{-}\B}}
\def\AbimB{{\A\text{-}\B}}
\def\AbimC{{\A\text{-}\C}}
\def\AonebimB{{\A_1\text{-}\B}}
\def\AtwobimB{{\A_2\text{-}\B}}
\def\BbimA{{\B\text{-}\A}}
\def\Aperf{{\A\text{-}\perf}}
\def\Bperf{{\B\text{-}\perf}}
\def\MddA{{M^{\tilde{\A}}}}
\def\MddB{{M^{\tilde{\B}}}}
\def\MhdA{{M^{h\A}}}
\def\MhdB{{M^{h\B}}}
\def\NhdB{{N^{h\B}}}
\def\Ainfty{{A_{\infty}}}
\def\Cat{{\it \mathcal{C}at}}
\def\twoCat{{\bicat{Cat}}}
\def\DGCat{{\bicat{DGCat}}}
\def\oneDGCat{{\bicat{DGCat}^{1}}}
\def\HoDGCat{{\hmtpy(\oneDGCat)}}
\def\HoDGCatV{{\hmtpy(\DGCat_\mathbb{V})}}
\def\MoDGCat{{Mo(\oneDGCat)}}
\def\EnhCat{{\bicat{EnhCat}}}                       
\def\EnhCatone{{\bicat{EnhCat}^1}}                  
\def\EnhCatCC{{\bicat{EnhCat}_{\mathrm{big}}}}                
\def\dgEnhCatCC{{\bicat{EnhCat}_{\mathrm{big}}^{\mathrm{dg}}}}                
\def\ainfEnhCatCC{{\bicat{EnhCat}_{\mathrm{big}}^{\Ainfty}}}                
\def\EnhCatKC{{\bicat{EnhCat}_{\mathrm{kc}}}}                
\def\dgEnhCatKC{{\bicat{EnhCat}_{\mathrm{kc}}^{\mathrm{dg}}}}         
\def\tr{{tr}}
\def\pretr{{pretr}}
\def\kctr{{kctr}}
\def\PreTrCat{{\DGCat^\pretr}}
\def\KcTrCat{{\DGCat^\kctr}}
\def\HoPretrCat{{\hmtpy(\PreTrCat)}}
\def\HoKcTrCat{{\hmtpy(\KcTrCat)}}
\def\Aquasirep{{\A\text{-}qr}}
\def\QAquasirep{{Q\A\text{-}qr}}
\def\Bquasirep{{\B\text{-}qr}} 
\def\lderA{{\tilde{\A}}} 
\def\lderB{{\tilde{\B}}} 
\def\adjunit{{\text{adj.unit}}}
\def\adjcounit{{\text{adj.counit}}}
\def\degzero{{\text{deg.0}}}
\def\degone{{\text{deg.1}}}
\def\degminusone{{\text{deg.-$1$}}}
\def\bareta{{\overline{\eta}}}
\def\barzeta{{\overline{\zeta}}}
\def\Ract{{R {\action}}}
\def\barRact{{\overline{\Ract}}}
\def\actL{{{\action} L}}
\def\baractL{{\overline{\actL}}}
\def\noddinf{{{\bf Nod}_{\infty}}}
\def\perfinf{{{\bf Perf}_{\infty}}}
\def\conodddgstrict{{{\bf coNod}_{dg}^{strict}}}
\def\conodddg{{{\bf coNod}_{dg}}}
\def\conodddghu{{{\bf coNod}_{dg}^{hu}}}
\def\conoddinf{{{\bf coNod}_{\infty}}}
\def\noddinfstr{{{\bf Nod}^{\text{strict}}_{\infty}}}
\def\noddinfA{{\noddinf\A}}
\def\noddinfB{{\noddinf\B}}
\def\perfinfA{{\perfinf\A}}
\def\perfinfB{{\perfinf\B}}
\def\noddinfAB{{\noddinf\AbimB}}
\def\noddinfBA{{\noddinf\BbimA}}
\def\noddinfu{{({\bf Nod}_{\infty})_u}}
\def\noddinfuA{{(\noddinfA)_u}}
\def\noddinfhu{{({\bf Nod}_{\infty})_{hu}}}
\def\noddinfhuA{{(\noddinfA)_{hu}}}
\def\noddinfhuB{{(\noddinfB)_{hu}}}
\def\noddinfdg{{({\bf Nod}_{\infty})_{dg}}}
\def\noddinfdgA{{(\noddinfA)_{dg}}}
\def\noddinfdgAA{{(\noddinf\AbimA)_{dg}}}
\def\noddinfdgAB{{(\noddinf\AbimB)_{dg}}}
\def\noddinfdgB{{(\noddinfB)_{dg}}}
\def\moddinf{{\modd_{\infty}}}
\def\moddinfA{{\modd_{\infty}\A}}
\def\naug{{\text{na}}}
\def\infbar{{B_\infty}}
\def\infbarnaug{{B^{\naug}_\infty}}
\def\infcobar{{CB_\infty}}
\def\infbarres{{\bar{B}_\infty}}
\def\infcobarres{{C\bar{B}_\infty}}
\def\infbarA{{B^\A_\infty}}
\def\infbarB{{B^\B_\infty}}
\def\infbarC{{B^\C_\infty}}
\def\inftimes{{\overset{\infty}{\otimes}}}
\def\infhom{{\overset{\infty}{\homm}}}
\def\barhom{{\rm H\overline{om}}}
\def\barend{{\overline{\eend}}}
\def\bartimes{{\;\overline{\otimes}}}
\def\bartimesA{{\;\overline{\otimes}_\A\;}}
\def\bartimesB{{\;\overline{\otimes}_\B\;}}
\def\bartimesC{{\;\overline{\otimes}_\C\;}}
\def\triaA{{\tria \A}}
\def\TPairdg{{\TPair^{dg}}}
\def\algA{{\alg(\A)}}
\def\Ainfty{{A_{\infty}}}
\def\gpmu{{\boldsymbol{\mu}}}
\def\odd{{\text{odd}}}
\def\even{{\text{even}}}
\def\pretriagmns{{\pretriag^{-}}}
\def\pretriagpls{{\pretriag^{+}}}
\def\eilmoordg{\eilmoor^{\text{dg}}}
\def\hprojemdg{\hproj^{\text{dg}}}
\def\hperfemdg{\hperf^{\text{dg}}}
\def\coeilmoordg{\coeilmoor^{\text{dg}}}
\def\eilmoorwk{\eilmoor^{\text{wk}}}
\def\eilmoorwkpf{\eilmoor^{\text{wk,perf}}}
\def\eilmoorwmor{\eilmoor^{\text{dg,wkmor}}}
\def\coeilmoorwk{\coeilmoor^{\text{wk}}}
\def\coeilmoorwmor{\coeilmoor^{\text{dg,wkmor}}}
\def\dnat{{d_{\text{nat}}}}
\def\bareta{{\tilde{\eta}}}
\def\bikmodk{{{\bf{biMod}}_k}}
\def\bidgbimod{{{\bf{biMod}}_{DG}}}
\def\kk{{\mathbb{k}}} 
\def\sym{{\mathcal{S}}}

\newcommand\bicat\mathbf            

\theoremstyle{definition}
\newtheorem{defn}{Definition}[section]
\newtheorem*{defn*}{Definition}
\newtheorem{exmpl}[defn]{Example}
\newtheorem*{exmpl*}{Example}
\newtheorem{exrc}[defn]{Exercise}
\newtheorem*{exrc*}{Exercise}
\newtheorem*{chk*}{Check}
\newtheorem*{remarks*}{Remarks}
\theoremstyle{plain}
\newtheorem{theorem}{Theorem}[section]
\newtheorem*{theorem*}{Theorem}
\newtheorem{conj}[defn]{Conjecture}
\newtheorem*{conj*}{Conjecture}
\newtheorem{question}[defn]{Question}
\newtheorem*{question*}{Question}
\newtheorem{prps}[defn]{Proposition}
\newtheorem*{prps*}{Proposition}
\newtheorem{cor}[defn]{Corollary}
\newtheorem*{cor*}{Corollary}
\newtheorem{lemma}[defn]{Lemma}
\newtheorem*{claim*}{Claim}
\newtheorem{Specialthm}{Theorem}
\renewcommand\theSpecialthm{\Alph{Specialthm}}
\numberwithin{equation}{section}
\renewcommand{\textfraction}{0.001}
\renewcommand{\topfraction}{0.999}
\renewcommand{\bottomfraction}{0.999}
\renewcommand{\floatpagefraction}{0.9}
\setlength{\textfloatsep}{5pt}
\setlength{\floatsep}{0pt}
\setlength{\abovecaptionskip}{2pt}
\setlength{\belowcaptionskip}{2pt}

\begin{abstract}
We prove an orbifold type decomposition theorem for the Hochschild homology
of the symmetric powers of a small DG category $\A$. In noncommutative
geometry, these can be viewed as the noncommutative symmetric quotient 
stacks of $\A$. We use this decomposition to show that the total 
Hochschild homology of the symmetric powers of $\A$ is isomorphic
to the symmetric algebra $S^*(\hochhom_\bullet(\A) \otimes t \kk[t])$. 
Our methods are explicit - we construct mutually inverse homotopy
equivalences of the standard Hochschild complexes involved.
These explicit maps are then used to induce from the symmetric algebra
onto the total Hochschild homology the structures of the Fock space for 
the Heisenberg algebra of $\A$, of a Hopf algebra, and 
of a free $\lambda$-ring generated by $\hochhom_\bullet(\A)$. 
\end{abstract}

\maketitle

\section{Introduction}
\label{section-introduction}

This paper is a step towards developing ``orbifold cohomology theory"
for DG categories. It derives its inspiration from the computation of
Hochschild and cyclic homology of twisted group rings 
(crossed product rings).  

Our starting point is the following well-known result. Let 
$X = \spec A$ be a smooth affine variety over an algebraically
closed field $\kk$ of characteristic $0$. Let $G$ be a finite
group acting on $X$ via a group homomorphism $\alpha: G \to Aut(X)$. 
Let $A_\alpha[G]$ be the corresponding twisted group ring. 
Its homology groups admit the following decomposition:
\begin{equation}
\label{eqn-orbifold-mixed-complex-decomposition-affine-case}
\hochcx_\bullet(A_\alpha[G], W) \simeq 
\big(\bigoplus_{g \in G}  \hochcx_\bullet (A(X^g), W)\big)_G
\end{equation}
Here $\hochcx_\bullet(\cdot, W)$ is the cyclic homology with coefficients
in a module $W$: for different $W$ this gives Hochschild, cyclic and periodic
cyclic homology 
\cite[\S4]{GetzlerJones-TheCyclicHomologyOfCrossedProductAlgebras}. 
Furthermore, $A(X^g)$ is the 
algebra of regular functions on the fixed point set $X^g$ and 
$\big(\ldots\big)_G$ denotes taking the coinvariants of the 
$G$ action $g\colon X^h \rightarrow X^{ghg^{-1}}$. 

The decomposition \eqref{eqn-orbifold-mixed-complex-decomposition-affine-case}
has two steps. The first step is
based on spectral sequence constructed by Feigin and Tsygan 
\cite{FeiginTsygan-CyclicHomologyOfAlgebrasWithQuadraticRelationsUniversalEnvelopingAlgebrasAndGroupAlgebras}
for any $G$ acting on a unital algebra $A$ over a commutative ring $R$.
For finite $G$ whose order is invertible in $R$, 
the spectral sequence reduces to a decomposition into a direct sum of
the twists $\leftidx{_g}A$ of the trivial $A\text{-}A$-bimodule by action 
of $g \in G$
\cite[\S4]{GetzlerJones-TheCyclicHomologyOfCrossedProductAlgebras}.
For a smooth commutative $A$, the second step is the interpretation of
these direct summands by Brylinski in terms of the fixed point sets
of the $G$-action on $X = \spec A$ \cite{Brylinski-CyclicHomologyAndEquivariantTheories}.

There are many geometrical generalisations of the decomposition 
\eqref{eqn-orbifold-mixed-complex-decomposition-affine-case}:
for smooth quasi-projective varieties over $\kk$ 
\cite{Baranovsky-OrbifoldCohomologyAsPeriodicCyclicHomology},  
as an orbifold version of the HKR isomorphism for a global quotient of a smooth variety \cite{ArinkinCaldararuHablicsek-FormalityOfDerivedIntersectionsAndTheOrbifoldHKRIsomorphism}
and a recent version for derived Deligne-Mumford stacks 
\cite{FuPortaSibillaScherotzke-HKRIsomorphismForDerivedDM}.

Our aim is a noncommutative generalisation of 
\eqref{eqn-orbifold-mixed-complex-decomposition-affine-case} 
replacing $A$ by a smooth and proper DG category $\A$. We are guided 
by a noncommutative geometry perspective 
which views small DG categories considered up to Morita equivalence 
as noncommutative schemes
\cite{KontsevichSoibelman-NotesOnAInftyAlgebrasAInftyCategoriesAndNoncommutativeGeometry,KatzarkovKontsevichPantev-HodgeTheoreticAspectsOfMirrorSymmetry, 
Kaledin-HomologicalMethodsInNoncommutativeGeometry, 
Orlov-SmoothAndProperNoncommutativeSchemesAndGluingofDGcategories,
Efimov-HomotopyFinitenessOfSomeDGCategoriesFromAlgebraicGeometry}. 
The first step of the noncommutative decomposition 
\eqref{eqn-orbifold-mixed-complex-decomposition-affine-case}
was established by Nordstr\"om for any small DG category $\A$ over  
a field $\kk$ with $\chr(\kk) = 0$
\cite{Nordstrom-FiniteGroupActionsOnDGCategoriesAndHochschildHomology}. 
The more delicate second step with its interpretation in 
terms of the fixed point loci, where we expect the smoothness
of $\A$ to become necessary, remains open. 

In this paper, we do it for the symmetric group $S_n$ 
acting on $\A^{\otimes n}$, the $n$-th tensor power of a DG category $\A$. 
The geometric fixed point loci are the diagonals $X^m
\subset X^n$ which are easily interpreted as tensor powers
$\A^{\otimes m}$. The conjugacy classes of $S_n$ are given by 
unordered partitions $\underline{\lambda} \vdash n$. 
Write $r_i(\underline{\lambda})$ for the number of parts of size $i$
in $\underline{\lambda}$ and
$r(\underline{\lambda})$ for the total number of parts.  
The fixed point locus of $\underline{\lambda}$ is
$\A^{\otimes r(\underline{\lambda})}$ and the centraliser acts on it 
as $S_{\underline{r}(\underline{\lambda})} := 
S_{r_1(\underline{\lambda})} \times \dots \times 
S_{r_n(\underline{\lambda})}$.
By the K{\"u}nneth formula
\begin{small}
$\hochhom_{\bullet}(\A^{r(\underline{\lambda})})_
{S_{\underline{r}(\underline{\lambda})}}
\simeq \Sym^{\underline{r}(\underline{\lambda})}
\hochhom_\bullet(\A)$\end{small}. Our first main result is:
\begin{theorem}[see Theorem
\ref{theorem-noncommutative-baranovsky-decomposition}]
\label{theorem-intro-noncommutative-orbifold-decomposition}
Let $\A$ be a small DG category and $n \geq 0$. The following
compositions are mutually inverse isomorphisms:
\begin{small}
\begin{equation}
\label{eqn-intro-noncommutative-baranovsky-decomposition-complexes}
\begin{tikzcd}
\hochhom_\bullet(\sym^n \A)
\ar[shift left = 1ex]{r}{\nu}
&
\ar[shift left = 1ex]{l}{\xi}
\bigoplus_{\underline{\lambda} \vdash n}
\hochhom_{\bullet}(\A^{\otimes n};\sigma_{\underline{\lambda}})_{C(\sigma_{\underline{\lambda}})}
\ar[shift left = 1ex]{r}{\sum f_{\underline{\lambda}}}
&
\ar[shift left = 1ex]{l}{\sum g_{\underline{\lambda}}}
\bigoplus_{\underline{\lambda} \vdash n}
\Sym^{\underline{r}(\underline{\lambda})} \hochhom_{\bullet}(\A),
\end{tikzcd}
\end{equation}
\end{small}
where $\nu$ and $\xi$ are the maps explicitly constructed in 
\S\ref{section-from-hh-of-quotient-stack-to-twisted-hhs}
and $f_{\underline{\lambda}}$ and $g_{\underline{\lambda}}$ -- in 
\S\ref{section-from-underline-n-twisted-hh-to-sym-underline-n-hh}. 
\end{theorem}

The action of $S_n$ on $\A^{\otimes n}$ is strong, i.e. it acts by
automorphisms, rather than autoequivalences. We can thus use 
the formalism of equivariant DG categories introduced in 
\cite[\S4.8]{GyengeKoppensteinerLogvinenko-TheHeisenbergCategoryOfACategory}. 
For any DG category $\A$ with a strong action of a finite group $G$, 
the equivariant category $\A \rtimes G$ is an analogue of
the twisted group ring, see Defn.~\ref{defn-equivariant-dg-category}. 
There is an isomorphism 
$\modd^G\text{-}\A \simeq \modd\text{-}(\A \rtimes G)$, i.e. the
$G$-equivariant $\A$-modules are the same 
as ordinary $(\A \rtimes G)$-modules. 
We can thus use $\A \rtimes S_n$ for  $\Sym^n \A$, as oppposed to
$\hperf^{S_n}\text{-}\A$ in an earlier approach of Kapranov and Ganter
\cite{GanterKapranov-SymmetricAndExteriorPowersOfCategories}.

In the Morita DG enhancement framework, both 
$\A \rtimes S_n$ and $\hperf^{S_n}\text{-}\A$ correspond
to the same enhanced triangulated category $\Sym^n \A$, 
the noncommutative analogue of the symmetric quotient stack $[X^n/S_n]$. 
The advantage is that $\A \rtimes G$ is a much smaller category, and
many computations become simpler. 
In particular, the explicit isomorphisms $\nu$ and $\xi$ in 
\eqref{eqn-intro-noncommutative-baranovsky-decomposition-complexes}
give in this simpler language the decomposition obtained by Nordstr\"om in 
\cite{Nordstrom-FiniteGroupActionsOnDGCategoriesAndHochschildHomology}
using the formalism of 
\cite{GanterKapranov-SymmetricAndExteriorPowersOfCategories}.

The decomposition of Theorem
\ref{eqn-intro-noncommutative-baranovsky-decomposition-complexes} 
implies that the total Hochschild homology $\bigoplus_{n \geq 0}
\hochhom_\bullet(\Sym^n \A)$ can be given the structure of a certain
symmetric algebra, the result
conjectured in \cite[Conj.~3.24]{BelmansFuKrug-HochschildCohomologyOfHilbertSchemesOfPointsOnSurfaces} and
\cite[Cor.~8.6]{GyengeKoppensteinerLogvinenko-TheHeisenbergCategoryOfACategory}
:
\begin{theorem}[see Theorem
\ref{theorem-total-hh-of-sym-stacks-iso-to-sym-algebra-of-hh-otimes-tkt}]
\label{theorem-intro-symmetric-algebra-structure-on-total-hh}
Let $\A$ be a small DG category. We have explicit mutually inverse 
isomorphisms  
\begin{equation}
\label{eqn-intro-symmetric-algebra-isomorphism-hh}
\begin{tikzcd}
\bigoplus_{n \geq 0} \hochhom_\bullet(\sym^n \A)
\ar[shift left = 1.24ex]{r}{\zeta}[']{\simeq}
&
S^*(\hochhom_\bullet(\A) \otimes t \kk[t]),
\ar[shift left = 1.25ex]{l}{\eta}
\end{tikzcd}
\end{equation}
where the RHS is the graded 
symmetric algebra of $\hochhom_\bullet(\A) \otimes t \kk[t]$ with 
$t \kk[t]= \kk t \oplus \kk t^2 \oplus \dots$ being of $\deg = 0$. 
Isomorphisms $\eta$ and $\zeta$ are the compositions of 
isomorphisms 
\eqref{eqn-intro-noncommutative-baranovsky-decomposition-complexes} 
of Theorem \ref{theorem-intro-noncommutative-orbifold-decomposition}
with the explicit isomorphism \eqref{eqn-symmetric-algebra-decomposition}. 
\end{theorem}
  
It is very important that we construct the isomorphisms 
of Theorems 
\ref{theorem-intro-noncommutative-orbifold-decomposition}
and \ref{theorem-intro-symmetric-algebra-structure-on-total-hh} 
explicitly, by constructing them as explicits maps of the standard
Hochschild complexes. For any graded vector space $E$, the symmetric
algebra $S^*(E \otimes t \kk[t])$ carries a number of natural 
structures: of a Hopf algebra, of the free $\lambda$-ring generated by $E$,
and of the Fock space for the Heisenberg algebra $H_{E}$ of $E$.  
Isomorphisms \eqref{eqn-intro-symmetric-algebra-isomorphism-hh}
allow us to transfer these structures to the total Hochschild homology 
$\bigoplus_{n \geq 0} \hochhom_\bullet(\sym^n \A)$ and compute
explicitly the result. 

In particular, one immediate application of our Theorem 
\ref{theorem-intro-symmetric-algebra-structure-on-total-hh}
is that $\bigoplus_{n \geq 0} \hochhom_\bullet(\sym^n \A)$  
is the Fock space for the Heisenberg algebra
$H_{\hochhom_\bullet(\A)}$. This generalises
the K-theoretic results of Segal
\cite{Segal-EquivariantKTheoryAndSymmetricProducts} and Wang
\cite{Wang-EquivariantKTheoryWreathProductsAndHeisenbergAlgebra}
inspired by the famous Heisenberg action of Nakajima 
\cite{Nakajima-HeisenbergAlgebraAndHilbertSchemesOfPointsOnProjectiveSurfaces}
and Grojnowski
\cite{Grojnowski-InstantonsAndAffineAlgebrasITheHilbertSchemeAndVertexOperators}.
In \cite{GyengeLogvinenko-TheHeisenbergAlgebraOfAVectorSpaceAndHochschildHomology}
Gyenge and the third author use the results of this paper
to compute explicitly the induced Heisenberg action on
$\bigoplus_{n \geq 0} \hochhom_\bullet(\sym^n \A)$
and prove it to be the same as the decategorification of the categorical
Heisenberg action of
\cite{GyengeKoppensteinerLogvinenko-TheHeisenbergCategoryOfACategory}.  

In this paper, we use the isomorphisms of Theorem 
\ref{theorem-intro-symmetric-algebra-structure-on-total-hh}
to compute the induced Hopf algebra and $\lambda$-ring structures on 
$\bigoplus_{n \geq 0} \hochhom_\bullet(\sym^n \A)$. 
Our third main result is:

\begin{theorem}[see Theorem
\ref{theorem-hopf-algebra-structure-on-the-total-hh-of-symn-a}]
\label{theorem-intro-hopf-algebra-structure-on-the-total-hh-of-symn-a}
Let $\A$ be a small DG category. The isomorphisms 
of Theorem \ref{theorem-intro-symmetric-algebra-structure-on-total-hh}
identify the standard Hopf algebra operations $\mu$, $\Delta$, $u$, $\epsilon$,
and $H$ on $S^*(\hochhom_\bullet(\A) \otimes t \kk[t])$ 
with the following operations on 
$\bigoplus_{n \geq 0} \hochhom_\bullet(\sym^n \A)$:
\begin{itemize}
\item The multiplication map $\mu$
given by the sum over all $n,m \geq 0$ of the maps
\begin{equation*}
\hochhom_\bullet(\sym^n \A) \otimes \hochcx_\bullet(\sym^m \A)
\xrightarrow{\simeq}
\hochhom_\bullet(\sym^n \A \otimes \sym^m \A)
\xrightarrow{\Ind_{S_n \times S_m}^{S_{n+m}}}
\hochhom_\bullet(\sym^{n+m} \A). 
\end{equation*}
\item The comultiplication map $\Delta$
given by the sum over all $n \geq 0$ and all subsets
$I \subseteq \left\{ 1,\dots, n \right\}$ of the maps
\begin{small}
\begin{equation*}
\begin{tikzcd}
\hochhom_\bullet(\sym^n \A) 
\ar{r}{\Res_{S_{I} \times S_{\bar{I}}}^{S_{n}}}
&
\hochhom_\bullet(\sym^{|I|} \A \otimes \sym^{n-|I|} \A)
\ar{r}{\simeq}
&
\hochhom_\bullet(\sym^{|I|} \A) \otimes \hochcx_\bullet(\sym^{n-|I|} \A), 
\end{tikzcd}
\end{equation*}
\end{small}
where $\bar{I} =  \left\{ 1, \dots, n \right\} \setminus I$, 
where $S_I, S_{\bar{I}} < S_n$ 
are the subgroups that only permute 
the elements of $I$ and of $\bar{I}$, and where $|I|$ is the size of $I$. 
\item The unit map $u'$ and the counit map $\epsilon'$
given by the inclusion of and the projection onto 
$\hochhom_\bullet(\sym^0 \A) \simeq \kk$, 
\item The antipode map $H'$
given on $\hochhom_\bullet(\sym^n \A)$ by $(-1)^n \id$. 
\end{itemize}
\end{theorem}
Finally, in Section \ref{section-k-theoretic-parallels-and-lambda-ring-structure} we compute the induced $\lambda$-ring structure on 
$\bigoplus_{n \geq 0} \hochhom_\bullet(\sym^n \A)$
and discuss the parallels with the $K$-theoretic results of Wang
\cite{Wang-EquivariantKTheoryWreathProductsAndHeisenbergAlgebra}.

When finishing our work on this paper, it was brought to our attention 
that the decompositions of Theorems 
\ref{theorem-intro-noncommutative-orbifold-decomposition} and
\ref{theorem-intro-symmetric-algebra-structure-on-total-hh}
were independently and simultaneously established by Nordstr\"om 
in \cite{Nordstrom-HochschildHomologyOfSymmetricPowersOfDGcategories}. 
He establishes them via general considerations which do 
not yield explicit maps on the level of Hochschild complexes. 
It is not possible, using his approach, to compute the induced 
Fock space, Hopf algebra, and $\lambda$-ring structures on
$\bigoplus_{n \geq 0} \hochhom_\bullet(\sym^n \A)$.

This paper is organized as follows. In Section \ref{section-preliminaries}
we give a quick overview of
the standard tools for DG categories: tensor products, modules, Hochschild 
complexes, semidirect products by a finite group and Eilenberg-Zilber theorem
for Hochschild complexes. In particular, for any functor $F: \A \to \A$ 
we have the twist $\leftidx{_F} \A$ of the $\AbimA$ bimodule $\A$ and
the Hochschild complex $\hochcx_\bullet(\A, F)$ computing Hochschild
homology of this bimodule. 
The emphasis is on explicit maps and constructions. 
In Section \ref{section-long-cycle-case} we consider the case of 
the cycle $t_n = (1, 2, \ldots, n)$ acting on $\mathcal{A}^{\otimes n}$ 
and construct an explicit homotopy equivalence between the 
complex of coinvariants $\hochcx_\bullet(\A^{\otimes n};t_n)_{t_n}$ and $\hochcx_\bullet(\A)$. In Section \ref{section-noncommutative-baranovsky-decomposition} we use this computation and Eilenberg-Zilber
theorem to write down the explicit isomorphisms of Theorems 
\ref{theorem-intro-noncommutative-orbifold-decomposition} and
\ref{theorem-intro-symmetric-algebra-structure-on-total-hh}. 
In Section 5 we use these explicit isomorphisms to compute
to compute the Hopf algebra and $\lambda$-ring structures on the 
direct sum $\bigoplus_{n \geq 0} \hochhom_\bullet(\sym^n \A)$ in terms of 
standard algebraic constructions involving $\hochhom_\bullet(\A)$.  

\bigskip

\noindent
\textbf{Acknowledgements.} 
We are grateful to Adam Gyenge and Andreas Krug for helpful comments
and suggestions. We are grateful to Hiraku Nakajima for introducing
the third author to the works of Segal
\cite{Segal-EquivariantKTheoryAndSymmetricProducts} and Wang
\cite{Wang-EquivariantKTheoryWreathProductsAndHeisenbergAlgebra}. 
The first author would also like to thank the participants of SUMaR at 
K-State 2025 --- Matthew Garwacke, Caleb Hornbuckle, Ben Joseph, and 
Emmerson Taylor. 
The first author was partially supported by the NSF grant DMS-2243854.

\section{Preliminaries}
\label{section-preliminaries}

\subsection{DG categories, modules and tensor products}

We briefly recall here the basics on DG categories (always assumed small in 
this paper) mostly to fix the notation, sending the reader to \cite{Keller-OnDifferentialGradedCategories}
for a more detailed treatment. 
A DG category over a field $\kk$ is a category $\A$ in which for any 
pair of objects $a, b$ the set $\leftidx{_a}{{\A}}{_b}$ of morphisms
$b \to a$ is a differential graded $\kk$-module, i.e. a complex of 
vector spaces over $\kk$. The inversion of order in the notation 
$\leftidx{_a}{{\A}}{_b}$ for morphisms makes it easier to identify composable morphisms
when working with Hochschild complexes
as compositions of morphisms are described by 
morphisms of complexes $\leftidx{_a}{{\A}}{_b} \otimes_\kk \;\leftidx{_b}{{\A}}{_c}
\to \leftidx{_a}{{\A}}{_c}$,  $\alpha \otimes \beta \mapsto \alpha \beta$. 

DG functors between DG categories are usual functors that are compatible with 
differentials on morphisms. 

A \textit{left DG module} $M$ over a DG category $\A$ is a DG functor 
$M: \A \to C_{dg}(\kk)$ taking values in the DG category of complexes over
$\kk$. Similarly, a \textit{right DG module} is a functor $\A^{op}$ from the opposite 
category. Hence for each object $a \in obj(\A)$ we have a complex 
$\leftidx{_a}M := M(a)$ with morphisms of complexes $\leftidx{_a}{{\A}}{_b}
\otimes_\kk \leftidx{_b}M \to \leftidx{_a}M$ for any pair of objects $a, b$, and 
similarly for right DG modules. 

A \textit{tensor product} $\A_1 \otimes \A_2$
of two DG categories $\A_1$, $\A_2$ has the set of objects
$\obj(\A_1) \times \obj(\A_2)$ with the morphisms
$$
\leftidx{_{a_1, a_2}}{{\left(\A_1 \otimes \A_2\right)}}{_{b_1, b_2}} := 
\leftidx{_{a_1}}{{\left(\A_1\right)}}{_{b_1}} \otimes_\kk\leftidx{_{a_2}}{{\left(\A_2\right)}}{_{b_2}}
$$
and natural compositions (involving the Koszul sign rule) and units. 

A \textit{DG $\A_1\text{-}\A_2$-bimodule} $M$ if a left DG module over $\A_1 \otimes \A_2^{op}$ which is
given by the complexes $\leftidx{_a}{{M}}{_b}: = M(a, b)$ similarly to the above. 
Central to this paper is the following special case of this notion. 
Let $F\colon \A \rightarrow \A$ be a DG functor. We define the bimodule
$$ \leftidx{_F}{\A} \in \AmodA $$
by setting 
$$ \leftidx{_a}{\left(\leftidx{_F}{\A}\right)}{_b} := 
\leftidx{_{Fa}}{{\A}}{_b}, $$
and letting $\A$ act naturally on the right and via $F$ on the left: 
$$ \alpha.(\beta).\gamma = F(\alpha) \beta \gamma \in \leftidx{_{Fc}}{{\A}}{_d} \quad \quad \quad 
\forall\ \alpha \in \leftidx{_c}{{\A}}{_a},\ 
\beta \in \leftidx{_{Fa}}{{\A}}{_b},\  \gamma
\in \leftidx{_b}{{\A}}{_d}. $$

\subsection{Equivariant DG categories}
\label{section-equivariant-dg-categories}

In this paper, we have to work with equivariant DG categories –– 
DG categories equipped with \textit{strong} group actions - and
 equivariant modules over them.  To that end, a new approach
 was introduced in
 \cite[\S4.8]{GyengeKoppensteinerLogvinenko-TheHeisenbergCategoryOfACategory}
 and further refined in
 \cite[\S5.2]{GyengeLogvinenko-TheHeisenbergAlgebraOfAVectorSpaceAndHochschildHomology}.
 It was 
 designed for working with DG categories up to Morita equivalence and it differs
 e.g. from the one   introduced by Kapranov and Ganter
in \cite{GanterKapranov-SymmetricAndExteriorPowersOfCategories}. 
We give below a small overview to set up notation. 

Let $\A$ be a small DG category. A \emph{strong} action 
of a finite group $G$ on $\A$ is an embedding of $G$ 
into the group of DG automorphisms of $\A$.  In particular, $G$ acts by 
\textit{automorphisms} rather than autoequivalences. It is a fairly rare situation 
but in this paper we work with the action of $G = S_n$ and its subgroups, on the tensor powers 
$\A^{\otimes n}$. In this case the action is strong. See \cite{Nordstrom-FiniteGroupActionsOnDGCategoriesAndHochschildHomology} for the discussion of
Hochschild complexes in the setting of weaker actions of groups on DG categories. 

\begin{defn}[\cite{GyengeKoppensteinerLogvinenko-TheHeisenbergCategoryOfACategory}, Defn.~4.45]
	\label{defn-equivariant-dg-category}
	The \emph{semi-direct product} $\A \rtimes G$ is the following
	DG category:
	\begin{itemize}
		\item $\obj \A \rtimes G = \obj \A$,
		\item For any two objects $a,b \in \obj(\A \rtimes G)$ their morphism complex is 
		\begin{equation*}
	\leftidx{_{a}}{{\A \rtimes G}}{_b} := 	\bigoplus_{g \in G}  \leftidx{_{ga}}{{\A}}{_b}.
		\end{equation*}
	We denote by $(\alpha, g)$ an element in the term corresponding to $g \in  G$ and $\alpha \in \leftidx{_{ga}}{{\A}}{_b}$. Then
		$\deg_{\A \rtimes G} (\alpha,g) = \deg_{\A} \alpha$ and 
		$d_{\A \rtimes G}(\alpha,g) = (d_\A \alpha,g)$,
		\item The composition in $\A \rtimes G$ is given, similarly to the case of a twisted group ring, by  the formula
		\begin{equation*}
			(\alpha_1, g_1) \circ (\alpha_2, g_2) =
			(\alpha_1 \circ g_1.\alpha_2,\, g_1 g_2). 
		\end{equation*}
		\item For any $a \in \obj (\A \rtimes G)$ the identity morphism 
		of $a$ is $(\id_a, 1_G)$. 
	\end{itemize}
\end{defn}
With this definition, modules over $\A \rtimes G$ 
are $G$-equivariant modules over $\A$ in the sense of 
\cite[\S2.1.3]{GanterKapranov-SymmetricAndExteriorPowersOfCategories}. 
\begin{lemma}[\cite{GyengeKoppensteinerLogvinenko-TheHeisenbergCategoryOfACategory}, Lemma 4.41]
	\label{lemma-A-rtimes-G-modules-are-G-equiv-A-modules}
	There are mutually inverse isomorphisms of categories
	\[ \modd\text{-}(\A \rtimes G) \leftrightarrows \modd^G\text{-}\A, \]
	\[ \hperf\text{-}(\A \rtimes G) \leftrightarrows \hperf^G\text{-}\A.\]
\end{lemma}
The 
second part of the following lemma uses the notion of the perfect hull
explained in Definition 4.22 or \textit{loc.cit.}.

\bigskip
\noindent
For any $g \in G$ the autoequivalence $g\colon \A \xrightarrow{\sim} \A$ 
extends to $\A \rtimes G$:
\begin{defn}
	\label{defn-group-element-as-autoequivalence-of-the-equivariant-category}
	Let $\A$ be a small DG category and $G$ be a group acting strongly on
	$\A$. For any $g \in G$ define the autoequivalence 
	$$ g\colon \A \rtimes G \xrightarrow{\sim} \A \rtimes G $$
	to have the same action on objects as
	$g\colon \A \xrightarrow{\sim} \A$ and to act on morphisms 
	by $g(\alpha, f) = (g(\alpha), gfg^{-1})$. 
\end{defn}

\subsection{Partitions}
\label{section-partitions}

\begin{defn}
Let $n \geq 0$. An \em (unordered) partition of $n$ \rm 
is an unordered collection 
$\underline{\lambda}\coloneqq \left\{ n_1, \dots, n_m \right\}$ of strictly
positive integers $n_i$ with $\sum_{i=1}^m n_i = n$. We denote this by 
$\underline{\lambda} \vdash n$. 
\end{defn}

The integers $n_i$ are the \em parts \rm of $\underline{\lambda} $. 
We write $r(\underline{\lambda} )$ for the \em length \rm of $\underline{\lambda} $. 
By this we mean the total number of parts in $\underline{\lambda}$, i.e. $m$. 
We write $r_k(\underline{\lambda} )$ for
the number of parts of size $k$ in $\underline{\lambda} $, i.e. the number of $i$ such 
that $n_i = k$.

For $n = 0$, we use the convention that there exists a unique
partition $\underline{0}$ of $0$ with $r(\underline{0}) =
r_k(\underline{0}) = 0$. 

\begin{defn}
Let $n \geq 0$. An \em ordered partition of $n$ \rm 
is $(n_1, \dots, n_m) \in \mathbb{Z}^m_{\geq 0}$ with 
$\sum_{i=1}^m n_i = n$. We denote this by $(n_1, \dots, n_m) \vdash n$. 
\end{defn}

We write $\forget$ for the forgetful map 
from all ordered partitions of $n$ to all unordered partitions of $n$ which 
forgets the ordering of the parts. 

\begin{defn}
Let $n \geq 0$ and let $\underline{\lambda}  \vdash n$ be
an unordered partition. Define $\underline{r}(\underline{\lambda} )$
to be the ordered partition $\left((r_1(\underline{\lambda} ), \dots,
r_n(\underline{\lambda} )\right)$ of $r(\underline{\lambda} )$.  
\end{defn}

\subsection{Symmetric powers of DG categories}
\label{section-symmetric-powers}

For any small DG category $\A$ and 
any $n \geq 0$ we write $\A^{\otimes n}$ for the $n$-fold tensor
product over $\kk$, viewed as an object in the category $\A$-$\A$-bimodules, cf.  
\cite[\S6.1]{Keller-DerivingDGCategories}. In other words
$$
\leftidx{_a}{\left(\A^{\otimes n}\right)}{_b} 
:= 
\bigoplus_{a_{1}, \dots, a_{n-1} \in \A} 
\leftidx{_a}{\A}{_{a_1}} \otimes_\kk
\leftidx{_{a_1}}{\A}{_{a_2}} \otimes_\kk \dots
\otimes \leftidx{_{a_{n-1}}}{\A}{_{b}}. 
$$
For $n = 0$, we use the convention that $\A^{\otimes 0} = \kk$.  Our main object of study in this paper is introduced in the following definition. 

\begin{defn}
Let $\A$ be a small DG category.  
Let $n \geq 0$ and let $(n_1, \dots, n_m) \vdash n$ be an ordered
partition. Define 
$$ S_{n_1, \dots, n_m} : = S_{n_1} \times \dots \times S_{n_m} < S_n,  $$
and define \em $(n_1, \dots, n_m)$-th symmetric power of $\A$ \rm to be 
$$ \sym^{n_1, \dots, n_m} \A : = 
\A^{\otimes n} \rtimes \left( S_{n_1} \times \dots \times S_{n_m} \right). $$
\end{defn}
\begin{defn}
Let $\A$ be a small DG category.  
Let $n \geq 0$ and let $\underline{\lambda}  \vdash n$ be a partition. 
Define $S_{\underline{\lambda} }$ and $\sym^{\underline{\lambda} } \A$ to be the
isomorphism classes of $S_{n_1, \dots, n_m}$ and 
$\sym^{n_1, \dots, n_m} \A$ for any $(n_1, \dots, n_m)$
such that $\forget(n_1, \dots, n_m) = \underline{\lambda}$. 
\end{defn}

We make similar definitions for graded vector spaces, complexes of
vector spaces, etc. All of these can be viewed as a special case 
of the following definition, which we need to work with Hochschild
complexes: 

\begin{defn}
Let $E \in \pretriagmns(\modk)$ 
be a bounded above twisted complex over $\modk$,
cf.~\cite[\S3]{AnnoLogvinenko-UnboundedTwistedComplexes}. 
Let $n \geq 0$ and let $(n_1, \dots, n_m) \vdash n$ be an ordered
partition. Define 
\em $(n_1, \dots, n_m)$-th symmetric power of $E$ \rm to be 
$$ \sym^{n_1, \dots, n_m} E : = 
(E^{\otimes n})_{S_{n_1, \dots, n_m}} \in \pretriagmns(\modk),$$
where the index denotes taking the coinvariants under the action 
of $S_{n_1, \dots, n_m}$ which permutes the corresponding factors. 

For any unordered partition $\underline{\lambda} \vdash n$
define $\sym^{\underline{\lambda}} E$ to be the isomorphism class of 
$\sym^{n_1, \dots, n_m} E$ for any $(n_1, \dots, n_m)$
such that $\forget(n_1, \dots, n_m) = \underline{\lambda}$.
\end{defn}

\subsection{Hochschild homology of a bimodule}

Once we have the notion of a bimodule $M$ over a category $\A$, the usual 
formalism of resolutions and derived functors can be extended to the categorical
level in a reasonably straghtforward way, see 
\cite[\S2 and \S3]{Keller-DerivingDGCategories},
 and the usual definition makes sense, cf.
 \cite[\S6.8]{Keller-HochschildCohomologyAndDerivedCategories}
 for the case 
$M = \A$:

\begin{defn}
	Let $\A$ be a small DG-category and let $M \in \AmodA$. 
	The \em Hochschild homology of $\A$ with coefficients in $M$ \rm is 
	\begin{equation}
		\hochhom_\bullet(\A;M) := H^\bullet(M \ldertimes_{\A\text{-}\A} \A),
	\end{equation}
\end{defn}
See also \cite[\S1.1.3]{Loday-CyclicHomology}
\cite[\S3, Step
4]{Baranovsky-OrbifoldCohomologyAsPeriodicCyclicHomology}\cite[\S3.2]{Nordstrom-FiniteGroupActionsOnDGCategoriesAndHochschildHomology}. 
Here $\ldertimes_{\A\text{-}\A}$ is the derived functor of the DG functor
\begin{equation}
	\label{eqn-bimodule-left-right-tensor}
	\otimes_{\A\text{-}\A}\colon \AmodA \otimes_k \AmodA \rightarrow \modk 
\end{equation}
where we tensor the left $\A$-action with the right $\A$-action and vice
versa. More precisely, $\A$-$\A$-bimodules are, equivalently, right
$\Aopp \otimes_{\kk} \A$-modules or left $\A \otimes_\kk \Aopp$-modules . 
We can view them as both left and right $\Aopp \otimes_{\kk} A$-modules 
via the canonical isomorphism on objects
\begin{align*}
	\A \otimes_\kk \Aopp & \xrightarrow{\sim} \Aopp \otimes_\kk \A,
	\\
	a \otimes b \quad & \mapsto \quad b \otimes a
\end{align*}
while for morphisms on the right hand side there is a 
factor $(\pm 1)$ determined by the Koszul rule. Then
\eqref{eqn-bimodule-left-right-tensor} is the functor of tensoring 
over this module structure. Explicitly, it sends any pair $E, F \in
\AmodA$ to the complex of $\kk$-modules
\begin{small}
	\begin{equation}
		E \otimes_{\AbimA} F \coloneqq 
		\big(\bigoplus_{a, b \in obj(\A)} \leftidx{_a}E{_b} \otimes_\kk \leftidx_{b}F{_a}\big) /
		\left\{ e \otimes \alpha.f.\beta - \beta.e.\alpha \otimes f \right\}
	\end{equation}
\end{small}
for all $e \in \leftidx{_a} E {_b}$, $\alpha \in \leftidx{_b}\A{_c}$, 
$f \in \leftidx{_c}F{_d}$, $\beta \in \leftidx{_d}\A{_a}$ and quadriples
of objects $a, b, c, d$ in $\A$. Note that the term  
$e \otimes \alpha.f.\beta$ is in $\leftidx{_a}E{_b} \otimes_\kk \leftidx_{b}F{_a}$ while $\beta.e.\alpha \otimes f$ is 
in $\leftidx{_d}E{_c} \otimes_\kk \leftidx_{c}F{_d}$.

\medskip
\noindent
We can compute $\ldertimes_{\A\text{-}\A}$ by taking an h-projective
resolution in either variable. Using the standard bar-complex
resolution $\barA$ of the diagonal bimodule $\A$, see e.g. 
\cite[Section~2.11]{AnnoLogvinenko-BarCategoryOfModulesAndHomotopyAdjunctionForTensorFunctors},
we see that  $\hochhom_\bullet(\A;M)$ are isomorphic to the cohomology
of the convolution of the \em Hochschild complex 
$\hochcx_\bullet(\A;M)$ with coefficients in $M$\rm:
\begin{small}
	\begin{equation}
		\label{eqn-hochschild-complex-of-a-category-with-coefficients-in-bimod}
		\dots
		\rightarrow 
		\bigoplus_{a,b,c \in \A} 
		\leftidx{_a}{M}{_c}
		\otimes_\kk
		\leftidx{_c}{\A}{_b}
		\otimes_\kk
		\leftidx{_b}{\A}{_a}
		\rightarrow 
		\bigoplus_{a,b \in \A} 
		\leftidx{_a}{M}{_b}
		\otimes_\kk
		\leftidx{_b}{\A}{_a}
		\rightarrow 
		\bigoplus_{a \in \A} 
		\leftidx{_a}{M}{_a}
	\end{equation}
\end{small}
with the differentials defined by 
\begin{align*}
	m_0 \otimes \alpha_1 \otimes \dots \otimes \alpha_n 
	& \quad \mapsto \quad 
	\quad \quad \quad  m_0.\alpha_1 \otimes \dots
	\otimes \alpha_{n-1}\; +
	\\
	& \quad \quad \quad \quad \quad + \sum_{i = 1}^{n-1} (-1)^i \; m_0 \otimes \dots \otimes \alpha_{i} \alpha_{i+1}
	\otimes \dots \otimes \alpha_n\; + \\ 
	& \quad \quad \quad \quad \quad + (-1)^{n + |\alpha_n|(|m_0| + \dots +
		|\alpha_{n-1}|)} \alpha_n.m_0 \otimes \alpha_1 \otimes \dots
	\otimes \alpha_{n-1}
\end{align*}

 \bigskip
 \noindent
 \textbf{Main Example.} For $M = \A$ we get Hochschild homology 
  $\hochhom_\bullet(\A; \A)$. If $F: \A \to \A$ is a functor then 
we define the \em $F$-twisted Hochschild homology $\hochhom_\bullet(\A;F)$ \rm 
to be the Hochschild homology $\hochhom_\bullet(\A;\leftidx{_F}{\A})$
of $\A$ with coefficients in the bimodule $\leftidx{_F}{\A}$ induced by $F$.   

The main computation of this paper is carried out in the case when 
$F$ is the permutation functor $\mathcal{A}^{\otimes n} \to \A^{\otimes n}$
given by the cycle $t_n = (1, 2, \ldots n)$.

\subsection{Shuffles}
\label{section-shuffles}
The usual K{\"u}nneth isomorphism works for DG categories with very little change. 
As in the case of $\kk$-algebras, one starts by defining shuffles: 
\begin{defn}
	Let $p,q \in \mathbb{Z}_{\leq 0}$. Define the subset of $(p,q)$-shuffles
	\begin{equation}
		S_{p,q} := \left\{ \sigma \in S_{p+q} \;\middle|\;
		\sigma(i) < \sigma(j) \text{ if } 
		\begin{matrix}
			1 \leq i < j \leq p, \\
			\text{ or }\\
			p+1 \leq i < j \leq p+q. 
		\end{matrix}
		\right\}
	\end{equation}
	The sign $(-1)^{\sigma}$ of a shuffle $\sigma \in S_{p,q}$ is its sign
	as a permutation in $S_n$. 
\end{defn}

In other words, $(p,q)$-shuffles are the elements of the permutation
group $S_{p+q}$ which preserve relative order of the first $p$ and 
the last $q$ elements.  Later in the paper we will need a computation 
of the number defined as follows: 
\begin{defn}
	Let $p,q \in \mathbb{Z}_{\leq 0}$. Define
	$$ Y_{p,q} : = 
	\begin{cases}
		\binom{\lfloor \frac{p+q}{2} \rfloor}{\lfloor \frac{p}{2} \rfloor}
		\quad \quad & \text{ either } p \text{ or } q \text{ is even},
		\\
		0, & \text{ otherwise. }
	\end{cases}
	$$
\end{defn}

\begin{lemma}
	Let $p,q \in \mathbb{Z}_{\leq 0}$. Then 
	\begin{equation}
		\label{eqn-formula-for-Ypq}
		Y_{p,q} = \sum_{\sigma \in S_{p,q}} (-1)^{\sigma}.
	\end{equation}
\end{lemma}
\begin{proof}
	We prove by induction on $p+q$.
	All $(p,q)$-shuffles $\sigma$ can be divided into those where the
	$\sigma^{-1}(1) = 1$ and $\sigma^{-1}(1) = p+1$. The former 
	correspond bijectively to $(p-1,q)$-shuffles of the same parity, 
	the latter correspond bijectively to $(p,q-1)$-shuffles whose
	parity differs by $(-1)^{p}$. Thus 
	$$ \sum_{\sigma \in S_{p,q}} (-1)^{\sigma}
	= \sum_{\sigma \in S_{p-1,q}} (-1)^{\sigma} + (-1)^p
	\sum_{\sigma \in S_{p,q-1}} (-1)^{\sigma}. $$
	By direct computation, we also have 
	\begin{equation}
		\label{eqn-recursive-formula-for-Ypq}
		Y_{p,q} = Y_{p-1,q} + (-1)^p Y_{p,q-1}.
	\end{equation}
	It remains to establish \eqref{eqn-formula-for-Ypq}
	in the cases where either $p=0$ or $q=0$, which is trivial. 
\end{proof}

We will also need below a slight generalisation of shuffles:
\begin{defn}
	Let $P,Q \subset \left\{ 1,\dots,n \right\}$ be two disjoint subsets. 
	Define the subset of $(P,Q)$-shuffles
	\begin{equation}
		S_{P,Q} := \left\{ \sigma \in S_{n} \;\middle|\;
		\begin{matrix}
			\sigma(i) < \sigma(j) \quad \text{ if } \quad 
			i < j \text{ with } i,j \in P \text{ or } i,j \in Q, \quad \quad \\
			\sigma(i) = i \quad \quad \;\; \text{ if } \quad i \neq P \cup Q
			\quad \quad\quad \quad\quad \quad\quad \quad\quad  \quad \quad
		\end{matrix}
		\right\}
	\end{equation}
	The sign $(-1)^{\sigma}$ of a shuffle $\sigma \in S_{P,Q}$ is its sign
	as a permutation in $S_n$. 
\end{defn}

\subsection{K{\"u}nneth isomorphism}
\label{subsection-kunneth-isomorphism}

Let $\A$ and $\B$ be two DG categories. The well-known K{\"u}nneth isomorphism 
\begin{equation} 
\hochhom_\bullet(\A; M) \otimes_{\kk}  \hochhom_\bullet(\B; N)
\simeq 
\hochhom_\bullet(\A \otimes_{\kk} \B; M \otimes_\kk N)
\end{equation}
can be written in terms of the Hochschild complexes
\eqref{eqn-hochschild-complex-of-a-category-with-coefficients-in-bimod} as follows. First we
define it on the level of bar complexes:
\begin{defn}
Let $\A$ and $\B$ be two DG categories. Let $\barA$ and $\barB$ be
their bar-complexes viewed as twisted complexes of $\A$-$\A$- 
and $\B$-$\B$-bimodules respectively
\cite[Defn.
2.24]{AnnoLogvinenko-BarCategoryOfModulesAndHomotopyAdjunctionForTensorFunctors}.
Similarly, let $\overline{\A \otimes_\kk \B}$ be the bar complex of
the tensor product category $\A \otimes_\kk \B$. 

Define the Eilenberg-Zilber map 
$$ \ezmap\colon
\barA \otimes_\kk \barB \rightarrow \overline{\A \otimes_\kk \B} $$
of twisted complexes of $\A \otimes_\kk \B$-$\A \otimes_\kk \B$-bimodules
to comprise the components
\begin{equation}
\label{eqn-kunneth-morphism-for-bar-complexes}
\A^{\otimes n} \otimes_\kk \B^{\otimes m}
\xrightarrow{\sum_{\sigma \in S_{n-2,m-2}} (-1)^{\sigma}
\iota_\sigma}
\A^{\otimes n+m-2} \otimes_\kk \B^{\otimes n+m-2}
\xrightarrow{\sim}
(\A\otimes_\kk \B)^{\otimes n+m-2}
\end{equation}
where $n,m \geq 2$ and  
for any $(n-2, m-2)$-shuffle $\sigma \in S_{n-2,m-2}$ the map 
$$ \iota_\sigma\colon 
\A^{\otimes n} \otimes_\kk \B^{\otimes m}
\rightarrow 
\A^{\otimes n+m-2} \otimes_\kk \B^{\otimes n+m-2}
$$
is given first by the repeated application of the identity insertion map 
\begin{align*}
\A^{\otimes 2} & \rightarrow \A^{\otimes 3} \\
\alpha_1 \otimes \alpha_2 & \mapsto \alpha_1 \otimes \id_{a_1} \otimes \alpha_2
\quad \quad \forall\; \alpha_i \in \homm_{\A}(a_i, a_{i-1}),
\end{align*}
to insert $m-2$ identity maps into $\A^{\otimes n}$ onto the positions 
$\sigma((n-2)+1)+1$, $\sigma((n-2)+2)+1$, \dots, $\sigma((n-1)+m-2)+1$ in 
$\A^{\otimes(n+m-2)}$, 
and then inserting $n-2$ identity maps 
onto the positions $\sigma(1)+1, \dots, \sigma(n-2)+1$
in $\B^{\otimes (n+m-2)}$.
\end{defn}

Since for any DG category $\C$ and a bimodule $L$ we have  
$\hochcx_\bullet(\C; M) \simeq \barC \otimes_{\CbimC} L$, we
can now define the Eilenberg-Zilber morphism for the Hochschild complexes:
\begin{defn}
 Let $\A$ and $\B$ be two DG categories. The Eilenberg-Zilber
morphism 
\begin{equation}
\label{eqn-kunneth-morphism-for-hochschild complexes}
\ezmap\colon \hochcx_\bullet(\A; M) \otimes_\kk \hochcx_\bullet(\B; N)
\rightarrow \hochcx_\bullet(\A \otimes \B; M \otimes_\kk N), 
\end{equation}
is the composition
\begin{equation}
\begin{tikzcd}[row sep = 0.5cm]
\left(\barA \otimes_{\AbimA} M\right)
\otimes_\kk
\left(\barB \otimes_{\BbimB} N\right)
\ar{d}{\sim}
\\
\left(\barA \otimes_{\kk} \barB \right) 
\otimes_{\A\otimes_\kk\B\text{-} \A\otimes_\kk\B}
\left(M \otimes_{\kk} N \right) 
\ar{d}{\eqref{eqn-kunneth-morphism-for-bar-complexes}}
\\
\overline{\A \otimes_{\kk} \B}
\otimes_{\A\otimes_\kk\B\text{-}\A\otimes_\kk\B}
\left(M \otimes_{\kk} N \right). 
\end{tikzcd}
\end{equation}
\end{defn}

Explicitly, for any Hochschild chains
$$\underline{\alpha} = m \otimes \dots \otimes \alpha_n \in \hochcx_n(\A; M)$$  
$$\underline{\beta} = n \otimes \dots \otimes \beta_m \in \hochcx_m(\B; N)$$
we have
\begin{small}
\begin{equation}
\label{eqn-explicit-formula-for-kunneth-morphism}
\ezmap\left(\underline{\alpha} \otimes \underline{\beta}\right)
= 
\sum_{\sigma \in S_{n,m}} (-1)^{\sigma}
(-1)^{\deg_\sigma(\underline{\alpha},\underline{\beta})}
(m \otimes n) \otimes \dots \otimes (\alpha_i \otimes \id)
\otimes \dots (\id \otimes \beta_j) \otimes \dots
\end{equation}
\end{small}
where in the summand indexed by $\sigma \in S_{n,m}$ the
factors $(\alpha_i \otimes \id)$  and $(\id \otimes \beta_j)$ occur
in the positions $\sigma(i)$ and $\sigma(n+j)$, respectively. The 
second sign is computed by setting 
$d_\sigma(\underline{\alpha},\underline{\beta})$ to be the sum of 
$\deg(\alpha_i) \deg(\beta_j)$ for all $0 \leq i \leq n$ and 
$0 \leq j \leq m$ such that the factor containing $\alpha_i$ occurs to 
the right of the factor containing $\beta_j$. 

As in \cite[\S4.2]{Loday-CyclicHomology} 
the morphism $\ezmap$ induces isomorphism of the Hochschild homology groups. 
The inverse isomorphism is induced by the  Alexander-Whitney map
\cite[Corollary X.7.2]{MacLane-Homology}

$$
\awmap ((a_0\otimes b_0) \otimes (a_1 \otimes b_1) \otimes \ldots \otimes
(a_n \otimes b_n))= 
$$
$$
\sum_{p=0}^n (-1)^\eqref{eqn-sign-for-AW-map}
\big(a_0  a_1 a_2 \ldots a_p \otimes a_{p+1} \otimes
\ldots \otimes a_n\big) \otimes \big( b_{p+1} \ldots b_n b_0 \otimes b_1 \otimes \ldots \otimes b_p \big)
$$
where the sign is given by
\begin{equation}
\label{eqn-sign-for-AW-map}
p(n-p)+\sum\limits_{i=0}^{n-1} \deg b_i\left( \sum\limits_{j=i+1}^n \deg a_j  \right)
+\left(\sum\limits_{i=p+1}^n \deg b_i \right) \left(\sum\limits_{j=0}^p \deg b_j \right).
\end{equation}

\subsection{Shuffle product with the constants}

We need to introduce the following notion:
\begin{defn}
Let $\A$ be a DG category. Let 
$$ \underline{\alpha} = \alpha_0 \otimes \alpha_1 \otimes \dots \otimes \alpha_n \in
\hochcx_{n}(\A) $$
be a Hochschild $(n+1)$-chain in $\A$ and let 
$$ \underline{\kappa} = \kappa_0 \otimes \kappa_1 \otimes \dots
\otimes \kappa_m \in \hochcx_{m}(\kk) $$
be a Hochschild $(m+1)$-chain in $\kk$. 
Define the shuffle 
\begin{equation}
\shfl\left(\underline{\alpha}, \underline{\kappa}\right) \in 
\hochcx_{n+m}(\A)
\end{equation}
to be the image of $\underline{\alpha} \otimes \underline{\kappa}$
under the composition 
$$ \hochcx_\bullet(\A) \otimes \hochcx_\bullet(\kk) \xrightarrow{K}
\hochcx_\bullet(\A \otimes_\kk \kk) \xrightarrow{\sim}
\hochcx_\bullet(\A) $$
where the first map is the K{\"u}nneth morphism 
\eqref{eqn-kunneth-morphism-for-hochschild complexes}
and the second map is induced by the canonical DG category isomorphism 
$\A \otimes_\kk \kk \xrightarrow{\sim} \A$. 
\end{defn}

Our principal application is the case $\underline{\kappa} =
1^{\otimes (m+1)}$. Explicitly, 
$\shfl\left(\underline{\alpha}, 1^{\otimes (m+1)} \right)$ 
is the sum over all $\sigma \in S_{n,m}$ of 
the $n+m$-chains in $\A$ obtained by using $\sigma$ to shuffle $m$ 
appropriate identity morphisms into $\underline{\alpha}$. The sign of
each summand is $(-1)^\sigma$. 

\section{Long cycle case}
\label{section-long-cycle-case}

Let $\A$ be a small DG category and let $n \geq 1$. The permutation group 
$S_n$ acts on the tensor power $\A^{\otimes n}$ by permuting the factors. 
For example, the long cycle $t_n:= (1 \dots n) \in S_n$ acts by the
isomorphism of categories
$$ t_n \colon \A^{\otimes n} \rightarrow \A^{\otimes n} $$
which maps
$$ a_1 \otimes a_2 \otimes \dots \otimes a_n \rightarrow a_n \otimes a_1 \otimes
\dots a_{n-1} $$
on objects and correspondingly on the morphisms. We use the same notation 
$t_n$ for the above permutation functor and write just $t$ for $t_n$ 
when no confusion is possible. 
Since the cycle $t_n$ commutes with itself, the isomorphism functor induced
by it induces an isomorphism  
$$ t_n \colon \hochcx_\bullet(\A^{\otimes n};t) \xrightarrow{\sim} 
\hochcx_\bullet(\A^{\otimes n};t).$$
Let $\hochcx_\bullet(\A^{\otimes n};t_n)_{t_n}$ denote the coinvariants of
the action of $t$, that is, the vector space quotient by the subspace
generated by $\underline{\alpha} - t_n.\underline{\alpha}$ for all 
$\underline{\alpha} \in \hochcx_\bullet(\A^{\otimes n};t_n)$. 

The main result of this section is
the construction of two mutually inverse homotopy equivalences 
\begin{equation}
\begin{tikzcd}
\hochcx_\bullet(\A^{\otimes n};t_n)_{t_n}
\ar[shift left = 1ex]{r}{f}
&
\hochcx_\bullet(\A).
\ar[shift left = 1ex]{l}{g}
\end{tikzcd}
\end{equation}

\subsection{Notation}

For a chain
$$ \underline{\alpha}_0 \otimes \dots \otimes
\underline{\alpha}_{m-1} \in \hochcx_{m-1}(\A^n;t)$$ 
we will shift indexing and decompose it 
 as a product of basic chains as follows
\begin{equation}
\label{eqn-basic-chain-in-hochcx-A^n}
\underline{\alpha}_i = \alpha_{(i+1)1} \otimes \dots \otimes
\alpha_{(i+1)n} 
\end{equation}
for some
$$ a_1, \dotsm, a_{nm+1} \in \A \text{ with } a_1 = a_{nm+1},$$
$$ \alpha_{ij} \in
\leftidx{_{a_{(j-1)m+i+1}}}\A{_{a_{(j-1)m+i}}}
$$ 
The index shift in \eqref{eqn-basic-chain-in-hochcx-A^n} is
due to us wanting to use the matrix notation:
the chain above can be visualised as 
\begin{equation}
\label{eqn-m-times-n-hochschild-chain-with-the-composable-order-indicated}
\begin{tikzcd}[row sep = 0.15cm, column sep = 0.15cm]
(\alpha_{11} 
\ar[dashed]{dd}
\ar[dashed, rounded corners, to path={(-4.25,2.50)
|- ([yshift=2.25ex]\tikztostart.north) -- (\tikztostart)}]
& \otimes &
\alpha_{12}
\ar[dashed]{dd}
& \otimes &
\dots
\ar[dashed]{dd}
& \otimes &
\alpha_{1n})
\ar[dashed]{dd}
\\
& & & \bigotimes & & & 
\\
(\alpha_{21} 
\ar[dashed]{dd}
& \otimes &
\alpha_{22}
\ar[dashed]{dd}
& \otimes &
\dots
\ar[dashed]{dd}
& \otimes &
\alpha_{2n})
\ar[dashed]{dd}
\\
& & & \bigotimes & & & 
\\
\dots 
\ar[dashed]{dd}
& \dots & 
\dots 
\ar[dashed]{dd}
& \dots & 
\dots 
\ar[dashed]{dd}
& \dots & 
\dots
\ar[dashed]{dd}
\\
& & & \bigotimes & & & 
\\
(\alpha_{m1} 
\ar[dashed, rounded corners, to
path={(\tikztostart.south) |- (-2.2,-3.25) -- 
(-2.2,3.25) -| (\tikztotarget.north)}]{rruuuuuu}
& \otimes &
\alpha_{m2}
\ar[dashed, rounded corners, to
path={(\tikztostart.south) |- (0,-3.25) -- 
(0,3.25) -| (\tikztotarget.north)}]{rruuuuuu}
& \otimes &
\dots 
\ar[dashed, rounded corners, to
path={(\tikztostart.south) |- (2.025,-3.25) -- 
(2.025,3.25) -| (\tikztotarget.north)}]{rruuuuuu}
& \otimes &
\alpha_{mn}). 
\ar[dashed, rounded corners, to path={
-- ([yshift=-2.25ex]\tikztostart.south) -|
(4.25,-2.0)}]
\ar[dashed, rounded corners, to path={
(4.25,-2.0) -- (-4.25,2.50)}]
\end{tikzcd}
\end{equation}
where the dotted line denotes the order in which 
the morphisms $\alpha_{ij}$ can be composed in. 
We denote such chain simply by the $m \times n$ matrix
\begin{equation}
\begin{pmatrix}
\alpha_{11} & \alpha_{12} & \dots & \alpha_{1n}  
\\
\alpha_{21} & \alpha_{22} & \dots & \alpha_{2n}  
\\
& & \dots & 
\\
\alpha_{m1} & \alpha_{m2} & \dots & \alpha_{mn}. 
\end{pmatrix}
\end{equation}

Note that for $n = 1$ we simply have 
$$\hochcx_{\bullet}(\A^n;t) = \hochcx_{\bullet}(\A),$$
so the chain
$$ \alpha_1 \otimes \dots \alpha_{m} \in \hochcx_{m-1}(\A) $$
is represented in the above notation by the column vector
\begin{equation}
\begin{pmatrix}
\alpha_{1} \\
\alpha_{2} \\
\dots
\\
\alpha_{m} 
\end{pmatrix}.
\end{equation}
Again, to facilitate the matrix notation we switched our notation 
for a chain in   $\hochcx_{m-1}(\A)$ from 
$\alpha_0 \otimes \dots \otimes \alpha_{m-1}$ to $\alpha_1 \otimes \dots \otimes \alpha_{m}$.

For example, we write a basic $4$-chain in $\hochcx_3(\A)$ as 
\begin{equation*}
\left(\begin{matrix}
\alpha_1 
\\
\alpha_2 
\\
\alpha_3 
\\
\alpha_4 
\end{matrix}
\right).
\end{equation*}
Each $\alpha_i$ is a morphism $a_{i+1} \rightarrow a_{i}$ for some
$a_1, \dots, a_5 \in \A$ with $a_5 = a_1$:
\begin{equation*}
a_1 \xrightarrow{\alpha_4} a_4 
\xrightarrow{\alpha_3} a_3 \xrightarrow{\alpha_2} a_2
\xrightarrow{\alpha_1} a_1.
\end{equation*}
Similarly, we write a basic $3$-chain in $\hochcx_{2}(\A^2;t)$ as 
\begin{equation*}
\left(\begin{matrix}
\alpha_{11} 
&
\alpha_{12} 
\\
\alpha_{21} 
&
\alpha_{22} 
\\
\alpha_{31} 
&
\alpha_{32}
\end{matrix}
\right).
\end{equation*}
Each $\alpha_{ij}$ is a morphism $a_{(i+1)j} \rightarrow a_{ij}$ for
some $a_{11}, a_{12}, \dots, a_{41}, a_{42} \in \A$ with $a_{41} =
a_{12}$ and $a_{42} = a_{11}$:
$$
a_{11} 
\xrightarrow{\alpha_{32}} a_{32}
\xrightarrow{\alpha_{22}} a_{22}
\xrightarrow{\alpha_{12}} a_{12}
\xrightarrow{\alpha_{31}} a_{31}
\xrightarrow{\alpha_{21}} a_{21}
\xrightarrow{\alpha_{11}} a_{11}.
$$

\subsection{Map $f$}
\label{section-map-f}

\begin{defn}
\label{defn-the-map-f}
Let $\A$ be a small DG category. 
For each $n \geq 1$ define the map 
$$ f \colon 
\hochcx_\bullet(\A^{\otimes n};t)_t
\rightarrow 
\hochcx_\bullet(\A)
$$
by setting for each $m \geq 1$ the map 
$f \colon 
\hochcx_{m-1}(\A^{\otimes n};t)_t
\rightarrow 
\hochcx_{m-1}(\A)$ to be
\begin{equation}
\label{eqn-quasi-iso-f-long-cycle}
\begin{pmatrix}
\alpha_{11} & \alpha_{12} & \dots & \alpha_{1n}  
\\
\alpha_{21} & \alpha_{22} & \dots & \alpha_{2n}  
\\
& & \dots & 
\\
\alpha_{m1} & \alpha_{m2} & \dots & \alpha_{mn}. 
\end{pmatrix} 
\quad \mapsto \quad 
\frac{1}{n}
\sum_{s = 1}^n
\begin{pmatrix}
\alpha_{1(s+1)} \dots \alpha_{m{s-1}}\alpha_{1s} \\
\alpha_{2s} \\
\dots
\\
\alpha_{ms} 
\end{pmatrix}
\end{equation}
 
\end{defn}

\subsection{Map $g$}
\label{section-map-g}

\begin{defn}
\label{defn-the-map-g}
Let $\A$ be a small DG category. For each $n \geq 1$ define the map 
$$ g \colon 
\hochcx_\bullet(\A)
\rightarrow 
\hochcx_\bullet(\A^{\otimes n};t)_t
$$
by setting for each $m \geq 1$ the map 
$g \colon 
\hochcx_{m-1}(\A)
\rightarrow 
\hochcx_{m-1}(\A^{\otimes n};t)_t$
to be
\begin{equation}
\label{eqn-quasi-iso-g-long-cycle}
\begin{pmatrix}
\alpha_{1} \\
\alpha_{2} \\
\dots
\\
\alpha_{m} 
\end{pmatrix}
\quad \mapsto \quad 
\sum_{c \in \left\{1,\dots,n\right\}^{m}
\\
\text{ with } c_1 = 1
}
(-1)^{\sigma_{c}}
\begin{pmatrix}
\beta_{11} & \beta_{12} & \dots & \beta_{1n}  
\\
\beta_{21} & \beta_{22} & \dots & \beta_{2n}  
\\
& & \dots & 
\\
\beta_{m1} & \beta_{m2} & \dots & \beta_{mn}.
\end{pmatrix}
\end{equation}
Here: 
\begin{itemize}
\item The summation is taken over all $c = (c_1, \dots, c_m)$ with 
$c_i \in \left\{1, \dots, n\right\}$ and $c_1 = 1$.
We think of each $c_i$ as a choice of a column position in the $i$-th row 
of the matrix $(\beta_{ij})$ and we always choose 
the first position in the first row.
\item We define $\sigma_c \in S_m$ to be 
the permutation where $\sigma_c(i)$ is 
the number of $c_k$ with $c_k < c_i$ or with
$c_k = c_i$ and $k \leq i$. 
\item We define
\begin{equation*}
\beta_{ij} := 
\begin{cases}
\alpha_{\sigma_c(i)} \quad \quad j = c_i, \\ 
\id \quad \quad \; j \neq c_i, \\ 
\end{cases}
\end{equation*}
\end{itemize}
\end{defn}

This definition can be understood as follows: take the matrix
$(\beta_{ij})$, start at $\beta_{11}$, and follow the composable
order of its entries (down each column and then to the top of 
the column to its right), placing $\alpha_1, \dots, \alpha_n$, 
in that order, in the chosen column position in each row. The remaining
entries are filled with the identity maps. We obtain a matrix 
each whose row contains exactly one $\alpha_i$. 
The permutation $\sigma_{c}$ corresponds to the 
new ordering on $\alpha_i$'s given by counting from the top to the bottom row.

For example, for $n = 2$ and $m = 4$ the map $g$ is
the map 
\begin{align*}
\begin{pmatrix}
\alpha_1
\\
\alpha_2
\\
\alpha_3 
\\
\alpha_4
\end{pmatrix}
\mapsto 
\quad
&
\begin{pmatrix}
\alpha_1 & \id
\\
\alpha_2 & \id
\\
\alpha_3 & \id
\\
\alpha_4 & \id
\end{pmatrix}
+
\begin{pmatrix}
\alpha_1 & \id
\\
\alpha_2 & \id 
\\
\alpha_3 & \id   
\\
\id & \alpha_4 
\end{pmatrix}
-
\begin{pmatrix}
\alpha_1 & \id
\\
\alpha_2 & \id 
\\
\id & \alpha_4 
\\
\alpha_3 & \id  
\end{pmatrix}
+
\begin{pmatrix}
\alpha_1 & \id
\\
\alpha_2 & \id 
\\
\id & \alpha_3 
\\
\id & \alpha_4 
\end{pmatrix}
+
\\
+ 
&\begin{pmatrix}
\alpha_1 & \id
\\
\id & \alpha_4 
\\
\alpha_2 & \id
\\
\alpha_3 & \id
\end{pmatrix}
-
\begin{pmatrix}
\alpha_1 & \id
\\
\id & \alpha_3 
\\
\alpha_2 & \id
\\
\id & \alpha_4
\end{pmatrix}
+
\begin{pmatrix}
\alpha_1 & \id
\\
\id & \alpha_3 
\\
\id & \alpha_4
\\
\alpha_2 & \id
\end{pmatrix}
+
\begin{pmatrix}
\alpha_1 & \id
\\
\id & \alpha_2 
\\
\id & \alpha_3
\\
\id & \alpha_4 
\end{pmatrix}
\end{align*}

\subsection{Homotopy $\Phi$}

We first need the following auxilliary construction:

\begin{defn}
\label{defn-cut-and-slice-matrix-operation-bpqdrc}
Let $m,n \geq 1$. We define the map 
\begin{equation}
B_{p,q,d,r,\underline{c}}\colon 
\hochcx_{m-1}(\A^{\otimes n})  \rightarrow 
\hochcx_{m}(\A^{\otimes n})
\end{equation}
parametrised by
\begin{itemize}
\item $p \in \left\{1, \dots, m\right\}$,
\item $q \in \left\{2, \dots, n\right\}$,
\item $d \in \left\{1, \dots, p\right\}$, 
\item $r \in \left\{0, \dots, d-1\right\}$ if $q < n$ and $r = 0$ if $q = n$, 
\item $\underline{c} = (c_1, \dots, c_r) \in \left\{ 1, \dots, n-q \right\}^r$,
\end{itemize}
to be the map 
which sends any chain
\begin{equation*}
\underline{\alpha}
:= 
\begin{pmatrix}
\alpha_{11} & \alpha_{12} & \dots & \alpha_{1n}  
\\
\alpha_{21} & \alpha_{22} & \dots & \alpha_{2n}  
\\
& & \dots & 
\\
\alpha_{m1} & \alpha_{m2} & \dots & \alpha_{mn}. 
\end{pmatrix} 
\end{equation*}
to the chain $\underline{\beta} = (\beta_{ij})$ defined by 
\\
\begin{tabular}{|l|l|l|l|}
\hline
Row group:
&
$i$
&
$j$
&
$\beta_{ij}$
\\
\hline
$1$
&
$1$ 
&
$1$
&
$\alpha_{(p+1)q}\dots\alpha_{mn}\alpha_{11}$
\\
&
&
otherwise
&
$\id$
\\
\hline
$2$
&
$2,\dots,r+1$
&
$c_{i-1}$
&
$\alpha_{\left(\sigma_{\underline{c}}(i-1)+1\right)1}$
\\
&
&
otherwise
&
$\id$
\\
\hline
$3$
&
$r+2, \dots, d$
& 
$n-q+1$
&
$\alpha_{i1}$
\\
& 
&
otherwise
&
$\id$ 
\\
\hline
$4$
&
$d+1, \dots, m+d-p$
&
$n-q+2,\dots, n$
&
$\alpha_{(i+p-d)(j-n+q-1)}$
\\
&
&
otherwise
& 
$\id$  
\\
\hline
$5$
&
$(m+d-p+1)$
& 
$n-q+2,\dots,n$
&
$\alpha_{1(j-n-q)}\dots\alpha_{d(j-n-q)}$  
\\
&
&
otherwise
&
$\id$  
\\
\hline
$6$
&
$m+d-p+2, \dots, m+1$
& 
$n-q+1,\dots,n$  
&
$\alpha_{(i-m+p-1)(j-n+q)}$ 
\\
&
& 
otherwise
&
$\id$ 
\\
\hline
\end{tabular}
\\
where $\sigma_{\underline{c}} \in S_r$, similar to 
Defn.~\ref{defn-the-map-g}, is the permutation of the rows
which arises when we distribute  
$\alpha_{21}, \dots, \alpha_{(r+1)1}$ in composable order
onto the positions in the rows $2,\dots,r+1$ chosen by
$\sigma_{\underline{c}}$. 
\end{defn}

\begin{defn}
\label{defn-map-phi-n}
For each $n \geq 1$ define a degree $-1$ map 
$$ \Phi_n\colon \hochcx_\bullet(\A^{\otimes n}; t) \rightarrow \hochcx_\bullet(\A^{\otimes n}; t) $$
to be the map 
\begin{equation}
\label{eqn-defn-of-psi-n}
\frac{1}{n}\sum_{s=1}^n \Phi'_n \circ t^{s-1}, 
\end{equation}
where $\Phi'_n$ is the map
which sends any chain 
\begin{equation*}
\underline{\alpha}
:= 
\begin{pmatrix}
\alpha_{11} & \alpha_{12} & \dots & \alpha_{1n}  
\\
\alpha_{21} & \alpha_{22} & \dots & \alpha_{2n}  
\\
& & \dots & 
\\
\alpha_{m1} & \alpha_{m2} & \dots & \alpha_{mn}. 
\end{pmatrix} 
\end{equation*}
to
\begin{equation}
\label{eqn-defn-of-psi'-n}
\sum_{p, q, d, r, \underline{c}}
\sum_{\tau }
\sum_{\upsilon}
(-1)^{\sigma_{\underline{c}}}
(-1)^{\tau}
(-1)^{\upsilon} 
(-1)^{m+(m+d)(p+d)}
\tau \upsilon B_{p,q,d,r,\underline{c}} \left( \underline{\alpha} \right). 
\end{equation}
where the summation ranges for $p,q,d,r, \underline{c}$ are as in
Definition \ref{defn-cut-and-slice-matrix-operation-bpqdrc}, 
$\tau$ ranges over  
$S_{\left\{r+2, \dots, m-p+d\right\}, \left\{d+1, \dots,
p\right\}, m+1}$, $\upsilon$ ranges over  
$S_{\left\{2, \dots, r+1\right\}, \left\{r+2, \dots, m+1\right\},
m+1}$.
The shuffles $\tau$ and $\upsilon$ act by permuting
the rows of the matrix form of the chain. Finally, 
the permutation $\sigma_{\underline{c}} \in S_r$ is as in  
Definition \ref{defn-cut-and-slice-matrix-operation-bpqdrc}. 
\end{defn}

\begin{theorem}
\label{theorem-phi-is-the-homotopy-of-gf-minus-id}
We have an equality of maps 
$\hochcx_\bullet(\A^{\otimes n}; t)_t 
\rightarrow \hochcx_\bullet(\A^{\otimes n}; t)_t$:
$$ d\Phi_n = g_n f_n - \id. $$
\end{theorem}
\begin{proof}
Let $f'_n$ be the $s = 1$ summand of 
$f_n$, so that $f_n = \frac{1}{n} \sum_{s=1}^n f'_n \circ t^{s-1}$. 
It suffices to show that we have an equality of maps 
$\hochcx_\bullet(\A^{\otimes n}; t) 
\rightarrow \hochcx_\bullet(\A^{\otimes n}; t)$:
$$ d\Phi'_n = g_n f'_n - \id. $$

For any $m \geq 1$ take any basic Hochschild chain 
$\alpha \in \hochcx_{m-1}(\A^{\otimes n})$ 
\begin{equation*}
\underline{\alpha}
:= 
\begin{pmatrix}
\alpha_{11} & \alpha_{12} & \dots & \alpha_{1n}  
\\
\alpha_{21} & \alpha_{22} & \dots & \alpha_{2n}  
\\
& & \dots & 
\\
\alpha_{m1} & \alpha_{m2} & \dots & \alpha_{mn}. 
\end{pmatrix} 
\end{equation*}

We need to show that 
\begin{equation}
\label{eqn-dpsi-equals-minus-psid+gf-id}
d\left(\Phi'_n\left(\underline{\alpha}\right)\right) = 
- \Phi'_n\left(d\underline{\alpha}\right) + g_nf'_n(\underline{\alpha}) - 
\underline{\alpha}. 
\end{equation}

We verify this by going through each summand in 
the sum \eqref{eqn-defn-of-psi'-n} for $\Phi'_n(\underline{\alpha})$ on the LHS
of \eqref{eqn-dpsi-equals-minus-psid+gf-id}
and differentiating it. Each summand of the result 
is a basic chain of $\hochcx_{m-1}(\A^{\otimes n})$ given 
by an $(m \times n)$-matrix involving $\alpha_{ij}$ 
with a coefficient $\pm 1$ coming from the the sign twists in 
\eqref{eqn-defn-of-psi'-n} and in the Hochschild differential. 
For each such matrix, we point
out that it occurs exactly one more time in the
sum \eqref{eqn-defn-of-psi'-n} for $d(\Phi'_n(\underline{\alpha}))$ 
on the LHS, in the sum \eqref{eqn-defn-of-psi'-n} for $\Phi'_n(d
\underline{\alpha})$ on the RHS, 
or in the term $g_nf'_n(\underline{\alpha}) - \underline{\alpha}$. 
We point out where it occurs and verify that the signs are such 
that the cancellation occurs. 

At the end of the proof, we will list the special summands of 
$d(\Phi'_n(\underline{\alpha}))$ on the LHS which cancel out 
those in the the term  $g_nf'_n(\underline{\alpha}) -
\underline{\alpha}$. For most of the proof, we will take 
the summands of $d(\Phi'_n(\underline{\alpha}))$ on the LHS, and 
identify the corresponding summands of 
$d(\Phi'_n(\underline{\alpha}))$ on the LHS or of 
$\Phi'_n(d \underline{\alpha})$ on the RHS which cancel them out.  
Our analysis makes it clear that all of the summands 
of $d(\Phi'_n(\underline{\alpha}))$ on the RHS 
are cancelled out. 

We first assume that $\tau = \upsilon = \id_{S_{m+1}}$. This is the
main substance of our analysis. We now consider the following cases: 

\fbox{$p, q, d, r$, and $\underline{c}$ generic:}

We start with generic values of $p, q, d, r$, and $\underline{c}$
in \eqref{eqn-defn-of-psi'-n} for which all six row
ranges listed in the table definining $B_{p, q, d, r
\underline{c}}(\underline{\alpha})$ in
Defn.~\ref{defn-cut-and-slice-matrix-operation-bpqdrc} are present and
non-empty. Explicitly, this means that $r \geq 2$, $d \geq r+1$, 
$d < p < m$, and $q < n$. 

We produce the cancelling summand for each summand of
$d B_{p, q, d, r \underline{c}}(\underline{\alpha})$. 
Denote by $d_i B_{p, q, d, r \underline{c}}(\underline{\alpha})$
the summand where $i$th and $(i+1)$st rows of
$B_{p, q, d, r \underline{c}}(\underline{\alpha})$ are merged.  
The way we produce the cancelling summand for 
$d_i B_{p, q, d, r \underline{c}}(\underline{\alpha})$
depends only on the row group $i$ and the row group of $i+1$ 
in the table of the row groups of 
$B_{p, q, d, r \underline{c}}(\underline{\alpha})$
given in Definition
\ref{defn-cut-and-slice-matrix-operation-bpqdrc}. 
It thus suffices to check the following cases, where by
case $(pq)$ we mean the case where $i$ lies in group $p$
and $i+1$ lies in group $q$:
\begin{itemize}
\item \underline{(12)}: This is the case $i = 1$. It has  
two subcases. If $c_1 = 1$, the 
the cancelling summand is  
$$B_{p-1, q, d-1, r-1,(c_2, \dots, c_r)}(d_1 \underline{\alpha}).$$
If $c_1 > 1$, then it is 
$$
d_{m+1} \tau B_{p, q, d, r, (c_2, \dots, c_r, c_1 - 1)}(\underline{\alpha})
\quad \text{ with } \quad \tau = \left((m+1)m\dots(r+1)\right).$$

\item \underline{(22)}: There are two subcases. If $c_i =
c_{i+1}$, the cancelling summand is 
$$B_{p-1, q, d-1, r-1,(c_1, \dots, c_{i}, c_{i+2}, \dots,
c_r)}(d_{\sigma_{\underline{c}}(i)} \underline{\alpha}).$$ 
If $c_i \neq c_{i+1}$, then it is 
$$d_{i} B_{p, q, d, r,(c_1, \dots, c_{i-1},
c_{i+1}, c_{i}, c_{i+2} \dots, c_r)} (\underline{\alpha}).$$ 

\item \underline{(23)}: This is the case $i = r+1$. 
The cancelling summand is 
$$ d_{r+1} \tau B_{p, q, d, r,\underline{c}} (\underline{\alpha})
\quad \text{ with } \quad \tau = \left((r+1)(r+2)\right). $$

\item \underline{(33)}: The cancelling summand is 
$$ B_{p-1, q, d-1, r, \underline{c}} (d_i \underline{\alpha}). $$

\item \underline{(34)}: This is the case $i=d$ The cancelling summand is 
$$ d_d \upsilon B_{p, q, d, r, \underline{c}} (\underline{\alpha})
\quad \text{ with } \quad \upsilon = \left(d(d+1)\right). $$

\item \underline{(44)}: The cancelling summand is 
$$ B_{p, q, d, r, \underline{c}} (d_{i + p - d} \underline{\alpha}).$$

\item \underline{(45)}: This is the case $i = m-p+d$. 
The cancelling summand is 
$$ B_{p, q, d, r, \underline{c}} (d_{m} \underline{\alpha}).$$

\item \underline{(56)}: This is the case $i = m-p+d+1$. 
The cancelling summand is 
$$ d_{m-p+d+1} \upsilon B_{p, q, d+1, r, \underline{c}} (\underline{\alpha})
\quad \text{ with } \quad \upsilon = \left((m-p+d+1)\dots(d+1)\right). $$

\item \underline{(66)}: The cancelling summand is 
$$ B_{p-1, q, d, r, \underline{c}} (d_{i-m+p-1} \underline{\alpha}).$$

\item \underline{(61)}: This is the case $i = m+1$. 
The cancelling summand is 
$$ d_1 \tau \upsilon B_{p-1, q, d, r, \underline{c}} (\underline{\alpha})
\quad \text{ with } \quad 
\tau = \left(2\dots(r+2)\right),
\upsilon = \left((r+2)\dots(d+1)\right). $$
\end{itemize}

\fbox{$p, q, d, r$, and $\underline{c}$ non generic:}

These are the cases when one or more of the six row
ranges listed in the table definining $B_{p, q, d, r
\underline{c}}(\underline{\alpha})$ in
Defn.~\ref{defn-cut-and-slice-matrix-operation-bpqdrc} are not
present. 

This adds to the above analysis the following additional cases 
where we need to describe the cancelling summand: 

\begin{itemize}
\item \underline{(13)}: This is the case $r = 0$ but $d \geq 2$. 
There are two subcases. If $q = n$, the the cancelling summand is  
$$B_{p-1, n, d-1, 0,\emptyset}(d_1 \underline{\alpha}),$$
where $\emptyset$ is the value of $\underline{c}$ when $r=0$. 
If $q < n$, the cancelling summand is
$$
d_{m+1} \tau B_{p, q, d, 1, (n-q)}(\underline{\alpha})
\quad \text{ with } \quad \tau = \left((m+1)m\dots2\right).$$

\item \underline{(14)}: This is the case $r = 0$, $d = 1$, $p < m$. The
cancelling summand is 
$$ B_{p+1, q, 1, 0, \emptyset}(d_{m+1}\underline{\alpha}). $$

\item \underline{(15)}: This is the case $r = 0$, $d = 1$, $p = m$ 
The are two cases. If $q = n$, then the cancelling summand is 
$\underline{\alpha}$ itself on the RHS of \eqref{eqn-defn-of-psi'-n}. 
If $q \neq n$, then the cancelling summand is
$$ B_{1, q+1, 1, 0, \emptyset}(d_{m+1}\underline{\alpha}). $$

\item \underline{(24)}: This is the case $d > 1$, $r = d-1$, $p < m$.
The cancelling summand is 
$$ d_d \tau B_{p, q, d, d-1, \underline{c}} (\underline{\alpha})
\quad \text{ with } \quad 
\tau = \left(d(d+1)\right).$$

\item \underline{(25)}: This is the case $d > 1$, $r = d-1$, $p = m$. 
The cancelling summand is 
$$ d_{d} \tau B_{m, q, d, d-1, \underline{c}} (\underline{\alpha})
\quad \text{ with } \quad 
\tau = \left(d(d+1)\right).$$

\item \underline{(35)}: This is the case $d > 1$, $r < d-1$, $p = m$.
The cancelling summand is:
$$ d_{d} B_{m,q,d-1,r,\underline{c}}(\underline{\alpha}). $$

\item \underline{(51)}: This is the case $d=p$. There are two
subcases. If $q = 2$, the cancelling summand is 
the summand of $g_n f'_n(\underline{\alpha})$ corresponding
to the summand of $g_n$ indexed by 
\begin{equation}
\label{eqn-indexing-c-for-the-summand-of-gf'}
\underline{c}' = (1,c_1, \dots, c_r,n-1, \dots, n-1, n, \dots, n) \in 
\left\{ 1, \dots, n\right\}^m,
\end{equation}
where there are $p-1-r$ and $m-p$ instances of $n-1$ and $n$,
respectively. If $q \neq 2$, the cancelling summand is
$$ d_{1} \tau B_{m,q-1,d,r,(c_1, \dots, c_r, n-q+1, \dots, n-q+1)}(\alpha)
\quad \text{ with } \quad \tau = (23\dots(d+1)).$$
\end{itemize}

Finally, let us assume that $\upsilon$ is non-trivial. 
Then 
$d \left(\upsilon B_{p,q,d,r,\underline{c}}(\underline{\alpha})\right)$ 
has three kinds of summands. First are those where the differential 
merges two rows belonging one each to the two row subsets which
$\upsilon$ shuffles. Then let $\upsilon'$ be the composition of
$\upsilon$ with the transposition interchanging these two rows. Then 
$d \left(\upsilon' B_{p,q,d,r,\underline{c}}(\underline{\alpha})\right)$ 
has an identical summand coming from merging the same two rows. All 
the sign twists in \eqref{eqn-defn-of-psi'-n} stay the same, 
except for $(-1)^{\upsilon'} = - (-1)^{\upsilon}$. We conclude that
the two summands cancel each other out. 

The second type of summands in 
$d \left(\upsilon B_{p,q,d,r,\underline{c}}(\underline{\alpha})\right)$
are those which do not belong to the first group, but
where the two rows being merged were also adjacent in 
$B_{p,q,d,r,\underline{c}}(\underline{\alpha})$. 
Thus there is a summand of 
$d(B_{p,q,d,r,\underline{c}}(\underline{\alpha}))$ 
where these two rows are merged.  
The analysis above give the cancelling counterpart $X$ for this summand. 
Let $\upsilon' \in S_m$ be the permutation obtained from $\upsilon$
where we treat the two merged rows as a single row. 
Then $\upsilon' X$ would cancel the desired summand of 
$d \left(\upsilon B_{p,q,d,r,\underline{c}}(\underline{\alpha})\right)$. 
It remains to show that $\upsilon'X$ exists as a summand in  
\eqref{eqn-dpsi-equals-minus-psid+gf-id}. We verify this
by analysing every applicable cancelling summand in the analysis
above for the case $\tau = \upsilon = \id$. Firstly, when 
$X = B_\bullet(d_\bullet \underline{\alpha})$, it is clear 
that $\upsilon' X$ is also a summand of $\Phi'_n(d \alpha)$. 
The remaining cases are:
\begin{itemize}
\item \underline{(12)}, the subcase $c_1 > 1$. Then 
$$ \upsilon' X = \upsilon' 
d_{m+1} \tau B_{p, q, d, r, (c_2, \dots, c_r, c_1 - 1)}(\underline{\alpha})
= 
d_{m+1} \tau \upsilon 
B_{p, q, d, r, (c_2, \dots, c_r, c_1 - 1)}(\underline{\alpha})
$$
where $\tau = \left((m+1)m\dots(r+1)\right)$.
\item \underline{(22)}, the subcase $c_i \neq c_{i+1}$. Then 
$$
\upsilon' X = \upsilon' 
d_{i} B_{p, q, d, r,(c_1, \dots, c_{i-1},
c_{i+1}, c_{i}, c_{i+2} \dots, c_r)} (\underline{\alpha})
= 
d_{i} \upsilon B_{p, q, d, r,(c_1, \dots, c_{i-1},
c_{i+1}, c_{i}, c_{i+2} \dots, c_r)} (\underline{\alpha}).$$ 
\item \underline{(23)}. Then 
$$ \upsilon' X = 
\upsilon' d_{r+1} \tau B_{p, q, d, r,\underline{c}}(\underline{\alpha})
= 
d_{r+1} \tau \upsilon B_{p, q, d,
r,\underline{c}}(\underline{\alpha}), $$
where $\tau = \left((r+1)(r+2)\right)$. 
\item \underline{(56)}. Then 
$$ \upsilon' X = 
\upsilon' d_{m-p+d+1} \upsilon'' B_{p, q, d+1, r, \underline{c}}
(\underline{\alpha})
= d_{m-p+d+1} \upsilon \upsilon'' B_{p, q, d+1, r, \underline{c}}
(\underline{\alpha}) $$
where $\upsilon'' = \left((m-p+d+1)\dots(d+1)\right).$
\item \underline{(61)}. Then
$$ \upsilon' X = 
\upsilon' d_1 \tau \upsilon'' B_{p-1, q, d, r, \underline{c}} (\underline{\alpha})
=
d_1 \tau (t_{m+1} \upsilon t_{m+1}^{-1} \upsilon'')  B_{p-1, q, d, r, \underline{c}} (\underline{\alpha})
$$
where $\tau = \left(2\dots(r+2)\right), 
\upsilon'' = \left((r+2)\dots(d+1)\right), t_{m+1} = (1\dots{m+1}).$
\item \underline{(13)}, the subcase $q < n$.
Then 
$$ \upsilon' X =
\upsilon'
d_{m+1} \tau B_{p, q, d, 1, (n-q)}(\underline{\alpha})
= d_{m+1} \upsilon \tau B_{p, q, d, 1, (n-q)}(\underline{\alpha}), $$
where $\tau = \left((m+1)m\dots2\right)$.

\item \underline{(51)}. In the subcase $q = 2$, recall that
$X$ is the summand of $g_n f'_n(\underline{\alpha})$ 
corresponding to the summand of $g_n$ indexed by $\underline{c}'$
as in \eqref{eqn-indexing-c-for-the-summand-of-gf'}. 
Then $\upsilon' X$ is the summand of 
$g_n f'_n(\underline{\alpha})$ corresponding to $\upsilon' \underline{c}'$. 

In the subcase $q \neq 2$, 
\begin{align*}
\upsilon' X & = 
\upsilon' d_{1} \tau B_{m,q-1,d,r,(c_1, \dots, c_r, n-q+1, \dots,
n-q+1)}(\underline{\alpha}) = \\
& = 
d_{1} \tau' B_{m,q-1,d,r,(c_1, \dots, c_r, n-q+1, \dots,
n-q+1)}(\underline{\alpha}),
\end{align*}
where $\tau = (23\dots(d+1))$ and $\tau' = t_{m+1} \upsilon
t_{m+1}^{-1} \tau$. 
\end{itemize}
 
The third and the last type of summands are those which do not belong
to the first group and where the two rows being merged were not adjacent in 
$B_{p,q,d,r,\underline{c}}(\underline{\alpha})$. This necessarily
means that one of the rows being merged is involved in the shuffle $\upsilon$, 
and the other isn't. There are three such cases, where by the case
$(ij)$ we mean the case where the rows being merged, before
they were permuted by $\upsilon$, belonged to the row
groups $i$ and $j$ of the table in 
Definition \ref{defn-cut-and-slice-matrix-operation-bpqdrc}:
\begin{itemize}
\item \underline{(24)}: This occurs when the shuffle $\upsilon$
permutes $d+1$-st row (the top row in group 4) just below the $r+1$-st
row (the bottom row of the group 2). The cancelling summand to
$d_{r+1} \left(\upsilon B_{p,q,d,r,\underline{c}}(\underline{\alpha})\right)$
is then 
$$ d_{r+1} \left(\tau \upsilon
B_{p,q,d,r,\underline{c}}(\underline{\alpha})\right) \quad \quad
\text{ with } \tau = ((r+1)(r+2)).$$
\item \underline{(35)}: This occurs when $\upsilon$ shuffles the
bottom row of the 3rd row group all the way down to the row $m-p+d+1$, 
just above the $5$th row group.  The cancelling summand is 
$$ \upsilon d_{d + m - p} B_{p,q,d-1,r,\underline{c}}(\underline{\alpha}). $$
where $\upsilon' \in S_m$ is obtained from $\upsilon \in S_{m+1}$ by
skipping $d$ in $\{1, \dots, m+1\}$ on the source side, and 
$\upsilon(d)$ on the target side.
\item \underline{(14)}: This occurs when the 2nd row group is empty, 
3rd row group is present, but $\upsilon$ shuffles the top row of the
$4$th group all the way up to the second row. 
The cancelling summand is 
$$ \upsilon' B_{p+1, q, d, r, \emptyset}(d_{m+1}\underline{\alpha}), $$
where $\upsilon' \in S_m$ is obtained from $\upsilon \in S_{m+1}$ by
skipping $d+1$ in $\{1, \dots, m+1\}$ on the source side, and 
$\upsilon(d+1)$ on the target side. 
\end{itemize}

Similar analysis reduces the case of arbitrary $\tau$ to the case 
$\tau = \id$.  
\end{proof}

\subsection{Homotopy $\Psi$}

We first need to introduce some integer-valued functions:
\begin{defn}
For any integers $k, l \geq 1$ define:
\begin{equation}
N_{k,l} := \sum_{c \in \left\{1,\dots, l\right\}^k} (-1)^{\sigma_c}.
\end{equation}
where the permutation $\sigma_c \in S_k$ is as in Definition 
\ref{defn-the-map-g}.
\end{defn}
In other words, $N_{k,l}$ is the number of the ways to choose a single 
entry in each row of a $k \times l$ matrix when counted with the
parity determined by the sign of the induced permutation of the rows.
The permutation is obtained by travelling down the columns 
from left to right, taking the order in which we encounter 
the chosen entries, and ordering their rows accordingly. 

\begin{lemma}
$$N_{k,l} = \sum_{i=1}^l \sum_{j=1}^{k} N_{k-j,i-1} Y_{j,k-j}.$$
\end{lemma}
\begin{proof}
The summands of the RHS each count the (signed) number of the choices of 
a single entry in each row of a $k \times l$ matrix where there are exactly $j$
chosen entries in the $i$-th column and no entries in columns $i+1,\ldots, l$.
\end{proof}

\begin{lemma}
\label{lemma-N(k+1)l-equals-Nkl-for-odd-k}
For odd $k$ we have $N_{k+1,l} = N_{k,l}$.  
\end{lemma}
\begin{proof}
Let us prove this by double induction on $k,l$. 
For all $l$, we have $N_{1,l}=l$. To compute $N_{2,l}$ let us sum over the position $m$ of the entry in the first row:
$$
N_{2,l}=\sum\limits_{m=1}^l (l-m+1)-(m-1)=\sum\limits_{m=1}^l l+2-2m = l(l+2)-2\frac{l(l+1)}{2}=l.
$$
Suppose we have the statement for all $N_{i,j}$ for $i\leq k$, $j\leq l$, and $i+j<k+l$. 
Let $k=2s+1$. Then 
\begin{align*}
N_{2s+2, l}  = & \sum\limits_{i=1}^l \sum\limits_{j=1}^{2s+2} N_{2s+2-j,i-1}Y_{j,2s+2-j}\\
= & \sum\limits_{i=1}^l \sum\limits_{t=1}^{s} N_{2s+2-2t,i-1}Y_{2t,2s+2-2t}+1
\qquad\qquad (\text{since }Y_{a,b}=0 \text{ for odd } a, b)\\
= & \sum\limits_{i=1}^l \sum\limits_{t=1}^{s} N_{2s+2-2t, i-1} (Y_{2t-1,2s+2-2t} + Y_{2t, 2s+2-2t-1}) +1\\
= & \sum\limits_{i=1}^l \sum\limits_{t=1}^{s} N_{2s+1-(2t-1), i-1} Y_{2t-1,2s+2-2t}  + \sum\limits_{i=1}^l \sum\limits_{t=1}^s N_{2s+1-2t, i-1} Y_{2t, 2s+2-2t-1}+1\\
= &N_{2s+1, l}.
\end{align*}
\end{proof}

\begin{defn}
Let $n,m,p,q \in \mathbb{Z}$ with $n, m, p \geq 1$, $q \geq 0$, $p+q \leq m$. 

We define
$$ T^{p,q}_{m,n} := \frac{1}{n} (-1)^{(p-1)(m-p-q)} Y_{p-1,q}
\sum_{s=1}^n N_{p-1,s-1} N_{q,n-s}. $$
\end{defn}

\begin{lemma}
\label{lemma-technical-results-on-Tpqmn}
Let $n,m,p,q \in \mathbb{Z}$ with $n, m, p \geq 1$, $q \geq 0$, $p+q \leq m$. 
Then:
\begin{enumerate}
\item 
\label{item-sign-alternating-formula-for-Tpq(m-1)n}
$T^{p,q}_{m,n} = (-1)^{p-1} T^{p,q}_{m-1,n}$.  
\item 
\label{item-Tpqmn-vanishes-for-even-p-and-odd-q}
For even $p$ and odd $q$, we have $T^{p,q}_{m,n} = 0$.  
\item 
\label{item-Tpqmn-via-T(p-1)qmn-and-Tp(q-1)mn-for-odd-p-and-even-q}
For odd $p > 1$ and even $q > 0$, we have 
\begin{equation}
T^{p,q}_{m,n} = (-1)^{m} T^{p-1,q}_{m,n} + T^{p,q-1}_{m,n}
\end{equation}
\end{enumerate}
\end{lemma}
\begin{proof}
For \eqref{item-sign-alternating-formula-for-Tpq(m-1)n}, note that 
the only thing that depends on $m$ in the definition of
$T^{p,q}_{m,n}$ is the sign twist $(-1)^{(p-1)(m-p-q)}$. For 
\eqref{item-Tpqmn-vanishes-for-even-p-and-odd-q}, note that by
definition $Y_{p-1,q} = 0$ for even $p$ and odd $q$. Finally, 
\eqref{item-Tpqmn-via-T(p-1)qmn-and-Tp(q-1)mn-for-odd-p-and-even-q}
follows from Lemma \ref{lemma-N(k+1)l-equals-Nkl-for-odd-k}
and the formula \eqref{eqn-recursive-formula-for-Ypq}. 
\end{proof}

\begin{prps}
Let $m,n \in \mathbb{Z}$ with $m, n \geq 1$. 
Let 
$$ 
\underline{\alpha} : = 
\left(
\begin{matrix}
\alpha_1 \\
\alpha_2 \\
\dots  \\
\alpha_m 
\end{matrix}
\right)
\in \hochcx_{m-1}(\A)
$$
be any Hochschild chain of length $m$. 

Let 
\begin{equation}
\begin{tikzcd}
\hochcx_\bullet(\A^{\otimes n};t)_t
\ar[shift left = 1ex]{r}{f_n}
&
\hochcx_\bullet(\A).
\ar[shift left = 1ex]{l}{g_n}
\end{tikzcd}
\end{equation}
be the maps defined in Definitions
\ref{defn-the-map-g} and \ref{defn-the-map-f}, respectively. 
Then 
\begin{small}
\begin{equation}
\label{eqn-the-summands-of-f_ng_n}
 f_n g_n (\underline{\alpha}) 
=
\sum_{
\begin{smallmatrix}
1 \leq p \leq m
\\
0 \leq q \leq m-p
\end{smallmatrix}
}
T^{p,q}_{m,n}
\shfl\left(
\left(
\begin{matrix}
\alpha_{m-q+1} \dots \alpha_m \alpha_1 \dots \alpha_p \\
\alpha_{p+1} \\
\dots \\
\dots \\
\dots \\
\alpha_{m-q} 
\end{matrix}
\right), 
\left(
\begin{matrix}
1 \\
\dots \\
p+q \\ 
\text{ in total } \\
\dots \\
1
\end{matrix}
\right)
\right).
\end{equation}
\end{small}
\end{prps}

\begin{proof}
By definition, the summands of $g_n(\alpha)$ are enumerated by $c \in
\left\{1,\dots,n\right\}^m$ with $c_1 = 1$. We think of $c$ as a
choice of one entry in each row of an $m \times n$ matrix. The
corresponding summand is this matrix with $\alpha_1, \dots, \alpha_n$ 
placed in the chosen entries in the cyclic order 
\eqref{eqn-m-times-n-hochschild-chain-with-the-composable-order-indicated}. 
All the other entries are filled with identity maps. 
The sign of this summand is the sign of 
the permutation $\sigma_c$ defined by setting $\sigma_c(i)$ to be 
the index of $\alpha_\bullet$ which was placed in row $i$,
cf.~\S\ref{section-map-g}. 

Applying $f_n$ to such matrix yields a sum indexed by its columns. 
For each column, we take 
the other columns of the matrix, compose them in the cyclic order 
\eqref{eqn-m-times-n-hochschild-chain-with-the-composable-order-indicated}
and pre-compose them with the top entry of the column. There is no
sign change, thus all $m$ summands have the sign $(-1)^{\sigma_c}$,
but gain a coefficient $\frac{1}{n}$. 

It follows that the summands of $f_n g_n(\alpha)$ on the LHS 
are precisely those in the sum on the RHS. It remains to show 
that the coefficients match. 

Fix any $1 \leq p \leq m$ and $0 \leq q \leq m-p$. Let $X$ be
any summand in the corresponding shuffle product on the RHS.  
Take any choice of $c \in \left\{1,\dots,n\right\}^m$ 
and of a column $1 \leq s \leq m$ which contribute to $X$ 
on the LHS. Permuting the elements $c_2, \dots, c_m$ by any permutation 
$\tau \in S_{m-1}$ which preserves the relative order of $c_i = s$ 
and the relative order of $c_i \neq s$ produces a new 
choice $c'$ of entries in the rows with $\sigma_c' = \tau \sigma_c$. 
On the LHS the summand $c'$ and $s$ contribute to is $\tau(X)$
and the sign of the contribution differs by $(-1)^\tau$ 
from that of $c$ and $s$. Hence on the LHS the coefficitients 
of $X$ and $\sigma(X)$ differ by $(-1)^\tau$. On the other hand, 
the same holds on the RHS by the definition of the shuffle product. 

It is therefore enough to show that the coefficient of any one summand 
in the shuffle product on the RHS matches its coefficient on the LHS.  

Let $X$ be the summand in the shuffle product which has $p-1$
identity maps above $\alpha_{p+1}, \dots, \alpha_{m-q}$ and has the
remaining $q$ identity maps below them. Take any 
$c \in \left\{1,\dots,n\right\}^m$ and $1 \leq s \leq m$ 
which contribute to $X$ on the LHS. This means
that $c_{p+1} = \dots = c_{m-q} = s$ and there are $p$ elements $c_i$ 
with $c_i < s$ and $q$ with $c_i > s$. Now let $S_{p-1+q} < S_{m-1}$ be
the subgroup which permutes first $p-1$ and the last $q$ entries and
let $\tau \in S_{p-1,q} \subset S_{p-1+q}$ be any shuffle. Then, 
similar to above, applying $\tau$ to $c_2, \dots, c_m$ produces 
$c'$ such that $\sigma' = \tau \sigma_c$. Since the swapped rows 
of the matrix all contain $\id$ in the $s$-th column, 
on the LHS $c'$ and $s$ contribute to the same summand $X$ 
but with the sign differing by $(-1)^\tau$. 

The action of $S_{p-1,q}$ on $c$ has no stabiliser and
its orbit contains exactly one element $c'$ where 
$c'_1, \dots, c'_p < s$ and 
$c'_{m-q+1}, \dots, c'_{q} > s$. It follows that to count 
the contribution to $X$ on the LHS from all 
$c \in \left\{1,\dots,n\right\}^m$ and fixed $s$, we can count 
contributions only from $c$ with $c_1, \dots, c_p < s$ 
and $c_{m-q+1}, \dots, c_{q} > s$ and multiply the result by 
$Y_{p-1, q} = \sum_{\tau \in S_{p-1,q}} (-1)^\tau$. 

It remains to count the contributions to $X$ from $c$ and $s$
with 
$$
\begin{array}{l l}
\bullet\; c_1 = 1, \quad
& \bullet\; 1 \leq c_2, \dots, c_p \leq s-1, 
\\
\bullet\; c_{p+1} = \dots = c_{m-q} = s, \quad
& \bullet\; s+1 \leq c_{m-q+1}, \dots, c_{q} \leq n. 
\end{array}
$$
In terms of choosing the entries in the rows of a $m \times n$ matrix, 
this means that our chosen entries must lie in four submatrices of 
size $1 \times 1$, $(p-1) \times (s-1)$, $(m-p-q) \times 1$ 
and $q \times (n-s)$. We are free to choose any placement of entries within
each submatrix. The cyclic order 
\eqref{eqn-m-times-n-hochschild-chain-with-the-composable-order-indicated}
traverses the four submatrices sequentially: first all the entries of
the first submatrix, then all the entries of the second, etc. 
This implies that $\sigma_c$ is the product 
of the four permutations corresponding to each submatrix. 
Thus all possible $c$ and $s$, counted with the sign $(-1)^{\sigma_c}$
contribute 
$$ \frac{1}{n} \sum_{s = 1}^{n} N_{1,1} N_{p-1,s-1} N_{s,1} N_{q,n-s} = 
\frac{1}{n} \sum_{s = 1}^{n} N_{p-1,s-1} N_{q,n-s}. $$

Multiplying by $Y_{p-1,q}$ we obtain the total coefficient of $X$ 
on the LHS. By inspection, this matches the coefficient $T^{p,q}_{m,n}
(-1)^{(p-1)(m-p-q)}$ of $X$ on the RHS. Here $(-1)^{(p-1)(m-p-q)}$
is the sign with which $X$ appears in the shuffle product. 
\end{proof}

This leads us to define the desired homotopy $\Psi$ as follows:

\begin{defn}
For each $n \geq 1$ define a degree $-1$ map 
$$ \Psi_n\colon \hochcx_\bullet(\A) \rightarrow \hochcx_\bullet(\A) $$
to be the map which sends any chain $$ 
\underline{\alpha} : = 
\left(
\begin{matrix}
\alpha_1 \\
\alpha_2 \\
\dots  \\
\alpha_m 
\end{matrix}
\right)
\in \hochcx_{m-1}(\A)
$$
to 
\begin{small}
\begin{equation}
\label{eqn-the-definition-of-the-map-psi}
(-1)^m
\sum_{
\begin{smallmatrix}
1 \leq p \leq m 
\\
0 \leq q \leq m-p 
\\
p + q \text{ even }
\end{smallmatrix}
}
T^{p,q}_{m,n}
\shfl\left(
\left(
\begin{matrix}
\alpha_{m-q+1} \dots \alpha_m \alpha_1 \dots \alpha_p \\
\alpha_{p+1} \\
\dots \\
\dots \\
\dots \\
\alpha_{m-q} 
\end{matrix}
\right), 
\left(
\begin{matrix}
1 \\
\dots \\
p+q+1 \\
\text{ in total } \\
\dots \\
1
\end{matrix}
\right)
\right).
\end{equation}
\end{small}
\end{defn}

\begin{theorem}
\label{theorem-psi-is-the-homotopy-of-fg-minus-id}
We have  
$$ d\Psi_n = f_n g_n - \id. $$
\end{theorem}
\begin{proof}
For any $m \geq 1$ take any Hochschild chain of length $m$.
$$ 
\underline{\alpha} : = 
\left(
\begin{matrix}
\alpha_1 \\
\alpha_2 \\
\dots  \\
\alpha_m 
\end{matrix}
\right)
\in \hochcx_{m-1}(\A)
$$
The shuffle product satisfies Leibnitz rule. Moreover, 
for any $1 \leq p \leq m$ and $0 \leq q \leq m-p$ with $p+q$ even, we have
\begin{small}
\begin{align*}
d 
&\left(
\begin{matrix}
\alpha_{m-q+1} \dots \alpha_m \alpha_1 \dots \alpha_p \\
\alpha_{p+1} \\
\dots \\
\dots \\
\dots \\
\alpha_{m-q} 
\end{matrix}
\right)
= 
\left(
\begin{matrix}
\alpha_{m-q+1} \dots \alpha_m \alpha_1 \dots \alpha_{p+1} \\
\alpha_{p+2}\\
\dots \\
\dots \\
\dots \\
\alpha_{m-q} 
\end{matrix}
\right)
+ 
\\
+
\sum_{j = 1}^{m-p-q-1}
(-1)^{j}
&\left(
\begin{matrix}
\alpha_{m-q+1} \dots \alpha_m \alpha_1 \dots \alpha_p \\
\alpha_{p+1} \\
\dots \\
\alpha_{p+j}\alpha_{p+j+1} \\
\dots \\
\alpha_{m-q} 
\end{matrix}
\right)
+
(-1)^{m}
\left(
\begin{matrix}
\alpha_{m-q} \dots \alpha_m \alpha_1 \dots \alpha_p \\
\alpha_{p+1} \\
\dots \\
\dots \\
\dots \\
\alpha_{m-q-1} 
\end{matrix}
\right)
\end{align*}
\end{small}
and
\begin{small}
\begin{equation*}
d\left(
\begin{matrix}
1 \\
\dots \\
p+q+1 \\
\text{ in total } \\
\dots \\
1
\end{matrix}
\right)
= 
\left(
\begin{matrix}
1 \\
\dots \\
p+q \\
\text{ in total } \\
\dots \\
1
\end{matrix}
\right)
\end{equation*}
\end{small}

It follows that that $d(\Psi_n(\underline{\alpha}))$ is the sum of
the following summands:
\begin{itemize}
\item For all $1 \leq p \leq m$ and $0 \leq q \leq m-p$ 
$$ \shfl\left(
\left(
\begin{matrix}
\alpha_{m-q+1} \dots \alpha_m \alpha_1 \dots \alpha_p \\
\alpha_{p+1} \\
\dots \\
\dots \\
\dots \\
\alpha_{m-q} 
\end{matrix}
\right), 
\left(
\begin{matrix}
1 \\
\dots \\
p+q \\
\text{ in total } \\
\dots \\
1
\end{matrix}
\right)
\right) $$
with the coefficient
$$
\begin{cases}
T^{p,q}_{m,n} & p + q \text{ even } \\
(-1)^m T^{p-1,q}_{m,n} + T^{p,q-1}_{m,n} & p + q \text{ odd }. 
\end{cases}
$$
\item For all $1 \leq p \leq m$, $0 \leq q \leq m-p$ and $1 \leq j
\leq m-p-q-1$
$$ \shfl\left(
\left(
\begin{matrix}
\alpha_{m-q+1} \dots \alpha_m \alpha_1 \dots \alpha_p \\
\alpha_{p+1} \\
\dots \\
\alpha_{p+j}\alpha_{p+j+1} \\
\dots \\
\alpha_{m-q} 
\end{matrix}
\right), 
\left(
\begin{matrix}
1 \\
\dots \\
p+q+1 \\
\text{ in total } \\
\dots \\
1
\end{matrix}
\right)
\right) $$
with the coefficient
$$
\begin{cases}
(-1)^{m+j}\; T^{p,q}_{m,n}  & p + q \text{ even } \\
0             & p + q \text{ odd }. 
\end{cases}
$$
\end{itemize}
On the other hand, 
\begin{equation}
d\underline{\alpha} 
= d \left(
\begin{matrix}
\alpha_1 \\ 
\alpha_2 \\
\dots \\ 
\dots \\
\dots \\ 
\alpha_m 
\end{matrix}
\right)
 = \sum_{j=1}^{m-1} 
(-1)^{j-1}
\left(
\begin{matrix}
\alpha_1 \\
\alpha_2 \\
\dots  \\
\alpha_j \alpha_{j+1} \\ 
\dots \\ 
\alpha_m \\
\end{matrix}
\right)
+ 
(-1)^{m-1}
\left(
\begin{matrix}
\alpha_m \alpha_1 \\
\alpha_2 \\
\dots \\
\dots \\
\dots \\
\alpha_m 
\end{matrix}
\right),
\end{equation}
and so $\Psi_n(d\underline\alpha)$ is the sum of the following summands:
\begin{itemize}
\item For all $1 \leq p \leq m$ and $0 \leq q \leq m-p$ 
$$ \shfl\left(
\left(
\begin{matrix}
\alpha_{m-q+1} \dots \alpha_m \alpha_1 \dots \alpha_p \\
\alpha_{p+1} \\
\dots \\
\dots \\
\dots \\
\alpha_{m-q} 
\end{matrix}
\right), 
\left(
\begin{matrix}
1 \\
\dots \\
p+q \\
\text{ in total } \\
\dots \\
1
\end{matrix}
\right)
\right) $$
with the coefficient
$$
\begin{cases}
(-1)^{m-1}\; T^{p-1,q}_{m-1,n} + T^{p,q-1}_{m-1,n} & 
p \text{ even}, q \text{ odd }, \\
0 & \text{ otherwise }. 
\end{cases}
$$
\item For all $1 \leq p \leq m$, $0 \leq q \leq m-p$ and $1 \leq j
\leq m-p-q-1$
$$ \shfl\left(
\left(
\begin{matrix}
\alpha_{m-q+1} \dots \alpha_m \alpha_1 \dots \alpha_p \\
\alpha_{p+1} \\
\dots \\
\alpha_{p+j}\alpha_{p+j+1} \\
\dots \\
\alpha_{m-q} 
\end{matrix}
\right), 
\left(
\begin{matrix}
1 \\
\dots \\
p + q + 1 \\
\text{ in total } \\
\dots \\
1
\end{matrix}
\right)
\right) $$
with the coefficient
$$
\begin{cases}
(-1)^{m+p+j}\; T^{p,q}_{m-1,n}  & p + q \text{ even } \\
0             & p + q \text{ odd }. 
\end{cases}
$$
\end{itemize}

The desired assertion now follows by adding the coefficients 
of the summands in $d(\Psi_n(\underline{\alpha}))$ and 
$\Psi_n(d\underline\alpha)$ to obtain the coefficients
in $(d\Psi_n)(\underline{\alpha})$ and then applying 
Lemma \ref{lemma-technical-results-on-Tpqmn}. 
\end{proof}

\section{Noncommutative orbifold decomposition}
\label{section-noncommutative-baranovsky-decomposition}

Let $\A$ be a small DG category. In this section, we use the mutually
inverse homotopy equivalences 
\begin{equation*}
\begin{tikzcd}
\hochcx_\bullet(\A^{\otimes n};t_n)_{t_n}
\ar[shift left = 1ex]{r}{f}
&
\hochcx_\bullet(\A).
\ar[shift left = 1ex]{l}{g}
\end{tikzcd}
\end{equation*}
constructed explicitly in \S\ref{section-long-cycle-case} to define
(also explicitly) the noncommutative orbifold decomposition:
\begin{equation}
\label{eqn-noncommutative-baranovsky-decomposition-initial}
\hochhom_\bullet(\sym^n \A) \simeq \bigoplus_{\underline{\lambda} \vdash n}
\Sym^{\underline{r}(\underline{\lambda})} \hochhom_\bullet(\A),
\end{equation}
where our notation $\Sym^{\underline{r}(\underline{\lambda})}$ for
symmetric powers of a graded vector space is explained in 
\S\ref{section-symmetric-powers}.

We construct the decomposition 
\eqref{eqn-noncommutative-baranovsky-decomposition-initial}
in two steps. 

\subsection{From the Hochschild homology of a quotient stack 
to the group-twisted Hochschild homologies}
\label{section-from-hh-of-quotient-stack-to-twisted-hhs}

The first step is very general and works for any quotient stack
(under a strong group action), not
just the symmetric one. 

\begin{defn}
\label{defn-maps-nu-xi_g}
Let $\A$ be a small DG category with a strong action of a finite group
$G$, cf.~\S\ref{section-equivariant-dg-categories}. Define maps 
\begin{equation}
\label{eqn-decomposition-of-orbifold-hh-into-the-sum-of-twisted-hhs}
\begin{tikzcd}
\hochcx_\bullet(\A \rtimes G)
\ar[shift left = 1ex]{r}{\nu}
&
\Bigl(
\bigoplus_{ g\in G}
\hochcx_{\bullet}(\A;g)
\Bigr)_{G} 
\ar[shift left = 1ex]{l}{\xi}
\end{tikzcd}
\end{equation}
as follows. The map $\nu$ sends 
\begin{equation*}
\begin{tikzcd}
(\alpha_0, \sigma_0) \otimes \dots \otimes (\alpha_m, \sigma_m) 
\in \hochcx_\bullet(\A \rtimes G) 
\ar[|->]{d}
\\
(\sigma_0\dots\sigma_m)^{-1}(\alpha_0) \otimes (\sigma_1 \dots
\sigma_m)^{-1}(\alpha_1) \otimes \dots \otimes 
(\sigma_{m})^{-1}(\alpha_m)
\in \hochcx_{m}(\A^n;(\sigma_0 \dots \sigma_m)^{-1}). 
\end{tikzcd}
\end{equation*}
It is a generalisation of \cite[Proposition
3.5]{Baranovsky-OrbifoldCohomologyAsPeriodicCyclicHomology}. 
The map $\xi$ is the composition 
\begin{equation}
\label{eqn-definition-of-the-map-xi}
\Bigl(
\bigoplus_{ g\in G}
\hochcx_{\bullet}(\A;g)
\Bigr)_{G} 
\rightarrow 
\bigoplus_{ g\in G}
\hochcx_{\bullet}(\A;g)
\xrightarrow{\sum_{g \in G} \xi_g}
\hochcx_\bullet(\A \rtimes G)
\end{equation}
where the first map is 
\begin{equation}
\label{eqn-coinvariant-embedding-map-symmetriser}
[\underline{\alpha}] \mapsto \frac{1}{|G|} \sum_{h \in G} 
h.\underline{\alpha}
\end{equation}
and each 
$$
\xi_g\colon \hochcx_{\bullet}(\A;g) \rightarrow \hochcx_\bullet(\A \rtimes G)
$$ 
is the map 
\begin{equation}
\alpha_0 \otimes \dots \otimes \alpha_m
\mapsto 
\left((\id,g^{-1}) \circ (\alpha_0, \id)\right) \otimes (\alpha_1, \id) \otimes \dots \otimes
(\alpha_{m-1},\id) \otimes  (\alpha_m,\id). 
\end{equation}
\end{defn}

\begin{lemma}
The maps $\nu$ and $\xi_g$ in Defn.~\ref{defn-maps-nu-xi_g} are 
well defined: they are compatible with the Hochschild differential 
in the twisted complexes $\hochcx_\bullet(\A \rtimes G)$, 
$\Bigl( \bigoplus_{ g\in G} \hochcx_{\bullet}(\A;g) \Bigr)_{G}$, 
and $\hochcx_\bullet(\A;g)$.
\end{lemma}
\begin{proof}
It is immediate that the maps 
$\nu$ and $\xi_g$ commute with the regular summands of 
the Hochschild differential. 
It remains to verify that this also holds 
for the cyclic permutation summand $d_{cyc}$. 

For map $\nu$, we have 
\begin{small}
\begin{align*}
& \nu\bigl(
d_{cyc} \bigl((\alpha_0, \sigma_0) \otimes \dots \otimes 
(\alpha_{m-1}, \sigma_{m-1}) \otimes (\alpha_m, \sigma_m) \bigr)\bigr)
=
\\
=\;
& \nu\bigl(
(-1)^{m} (-1)^{\alpha_m | \alpha_0, \dotsm \alpha_{m-1}} 
(\alpha_m, \sigma_m) (\alpha_0, \sigma_0) \otimes \dots \otimes 
(\alpha_{m-1}, \sigma_{m-1})\bigr)
= 
\\
=\; 
& 
(-1)^{m} (-1)^{\alpha_m | \alpha_0, \dotsm \alpha_{m-1}} 
(\sigma_m \sigma_0 \dots \sigma_{m-1})^{-1} (\alpha_m)
(\sigma_0 \dots \sigma_{m-1})^{-1}(\alpha_0) \otimes 
\dots \otimes \sigma_{m-1}^{-1}(\alpha_{m-1}), 
\end{align*}
\end{small}
while 
\begin{small}
\begin{align*}
& d_{cyc} \nu\bigl(
(\alpha_0, \sigma_0) \otimes \dots \otimes 
(\alpha_{m-1}, \sigma_{m-1}) \otimes (\alpha_m, \sigma_m) \bigr)
=
\\
=\;
&
d_{cyc}
\bigl(
(\sigma_0\dots\sigma_m)^{-1}(\alpha_0) \otimes \dots \otimes
(\sigma_{m-1} \sigma_m)^{-1}(\alpha_{m-1}) \otimes
(\sigma_{m})^{-1}(\alpha_m)
\bigr)
=
\\
=\;
&
(-1)^{m} (-1)^{\alpha_m | \alpha_0, \dotsm \alpha_{m-1}} 
(\sigma_{m}\sigma_0\dots\sigma_m)^{-1}(\alpha_m)
(\sigma_0\dots\sigma_m)^{-1}(\alpha_0) \otimes 
\dots \otimes 
(\sigma_{m-1}\sigma_m)^{-1}(\alpha_{m-1}) 
= 
\\
=\;
& \sigma_m^{-1} \Bigl(
(-1)^{m} (-1)^{\alpha_m | \alpha_0, \dotsm \alpha_{m-1}} 
(\sigma_{m}\sigma_0\dots\sigma_{m-1})^{-1}(\alpha_m)
(\sigma_0\dots\sigma_{m-1})^{-1}(\alpha_0) \otimes 
\dots \otimes 
(\sigma_{m-1})^{-1}(\alpha_{m-1}) 
\Bigr).
\end{align*}
\end{small}
Thus $d_{cyc} \nu$ and $\nu d_{cyc}$ become equal after 
taking the $G$-coinvariants.  

For each $g \in G$, for map $\xi_g$ we have
\begin{small}
\begin{align*}
& \xi_g \left(d_{cyc} \left( \alpha_0 \otimes \dots \otimes \alpha_m
\right) \right) = 
\\
=
& 
\xi_g \left( 
(-1)^{m} (-1)^{\alpha_m | \alpha_0, \dotsm \alpha_{m-1}} 
g(\alpha_m) \alpha_0 \otimes \dots \otimes \alpha_{m-1}  \right) = 
\\
=
&(-1)^{m} (-1)^{\alpha_m | \alpha_0, \dotsm \alpha_{m-1}} 
\left(\left(\id,g^{-1}\right) \circ \left(g(\alpha_m) 
\alpha_0, \id\right)\right)
\otimes \dots \otimes \left(\alpha_{m-1}, \id\right),
\end{align*}
\end{small}
while
\begin{small}
\begin{align*}
&  d_{cyc}\left( \xi_g\left( \alpha_0 \otimes \dots \otimes \alpha_m
\right) \right) = 
\\
=
& 
d_{cyc} \left( 
\left(\left(\id,g^{-1}\right) \circ \left(\alpha_0, \id\right)\right)
\otimes \dots \otimes \left(\alpha_{m}, \id \right)
\right)
=
\\
=
&
\left(
\left(\alpha_m,\id\right) \circ \left(\id,g^{-1}\right) 
\circ \left(\alpha_0, \id \right)
\right)
\otimes \dots \otimes \left(\alpha_{m-1}, \id \right)
=
\\
=
&
(-1)^{m} (-1)^{\alpha_m | \alpha_0, \dotsm \alpha_{m-1}} 
\left(\left(\id,g^{-1}\right) \circ \left(g(\alpha_m) \alpha_0, \id\right)
\right) \otimes \dots \otimes \left(\alpha_{m-1}, \id \right). 
\end{align*}
\end{small}
Thus $d_{cyc} \xi_g$ and $\xi_g d_{cyc}$ are equal.  
\end{proof}

\begin{prps}
\label{prps-nu-and-xi-are-mutually-inverse-homotopy-equivalences}
The maps $\nu$ and $\xi$ in Defn.~\ref{defn-maps-nu-xi_g} are mutually 
inverse homotopy equivalences and hence induce mutually inverse isomorphisms 
\begin{equation}
\label{eqn-decomposition-of-orbifold-hh-into-the-sum-of-twisted-hhs}
\begin{tikzcd}
\hochhom_\bullet(\A \rtimes G)
\ar[shift left = 1ex]{r}{\nu}
&
\Bigl(
\bigoplus_{ g\in G}
\hochhom_{\bullet}(\A;g)
\Bigr)_{G}.
\ar[shift left = 1ex]{l}{\xi}
\end{tikzcd}
\end{equation}
\end{prps}
\begin{proof}
The composition $\nu \xi$ is the identity map of 
$\Bigl( \bigoplus_{ g\in G}
\hochcx_{\bullet}(\A;g)
\Bigr)_{G}$. The composition $\xi \nu$ is the endomorphism 
of $\hochcx_\bullet(\A \rtimes G)$
which sends any chain 
$$
(\alpha_0, \sigma_0) \otimes (\alpha_1, \sigma_1) \otimes \dots \otimes (\alpha_m, \sigma_m) 
\in \hochcx_m(\A \rtimes G) $$ 
to the chain 
$$
\frac{1}{|G|}\sum\limits_{g\in G}
(g\alpha_0, g\sigma_0\dots \sigma_mg^{-1}) \otimes
(g(\sigma_1\dots\sigma_m)^{-1}(\alpha_1), \id)
\dots \otimes (g\sigma_m^{-1}(\alpha_m), \id) 
\in \hochcx_m(\A \rtimes G). $$ 
Denote $\beta_i^j=(\sigma_i\ldots\sigma_j)^{-1}\alpha_i$ for $0\leq i\leq j \leq m$.
The homotopy between $\xi \nu$ and \eqref{eqn-coinvariant-embedding-map-symmetriser}
on $\hochcx_\bullet(\A \rtimes G)$
is given by the degree $-1$ endomorphism which sends any chain 
$$ 
(\alpha_0, \sigma_0) \otimes (\alpha_1, \sigma_1) \otimes \dots \otimes (\alpha_m, \sigma_m) \in \hochcx_m(\A \rtimes G) $$ 
to the chain 
\begin{multline*}
\frac{1}{|G|}\sum\limits_{g\in G} g.\Biggl(
\sum\limits_{p=0}^{m-1}
\sum\limits_{r, (s_1,\ldots, s_r)\vdash m-p}
\sum\limits_{J\subset \{p+1,\ldots, m+1\}}
(-1)^\eqref{eqn-sign-main-homotopy-xinu-id}
(\beta_{m-p-r+1}^m\ldots\beta_m^m \alpha_0,\sigma_0) \otimes
\\
 \otimes (\alpha_1,\sigma_1)\otimes \ldots \otimes (\alpha_p,\sigma_p)
\otimes X_{p+1}(p,\underline{s},J) \otimes \ldots \otimes X_{m+1}(p,\underline{s},J)\Biggr)+
\\
+\frac{1}{|G|}\sum\limits_{g\in G}\sum\limits_{i=0}^m
(-1)^{i-1}
(g(\beta_1^m\ldots\beta_m^m), g\sigma_0)\otimes (\id,\sigma_1) \otimes \ldots \otimes (\id,\sigma_i)
\otimes (\id, g^{-1}) 
\\
\otimes (\id, g\sigma_{i+1}g^{-1})\otimes \ldots \otimes (\id, g\sigma_mg^{-1}),
\end{multline*}
where
$J$ is a subset of $\{p+1,\ldots, m+1\}$ such that for all $i\in \{p+1,\ldots, m+1\}$
we have $\sum_{j\in J,\ j<i} s_j  \geq |\{j\notin J, p<j\leq i\}|$; in particular, $p+1\in J$ always.
Furthermore, let $k(i)=|\{j\in J,\ j<i\}$, then let
$$
X_i(p,s_1,\ldots, s_r,J)=(\id, \sigma_{p+s_1+\ldots s_{k(i)}+1}\ldots\sigma_{p+s_1+\ldots s_{k(i)+1}})
$$
if $i\in J$, and
$$
X_i(p,s_1,\ldots, s_r, J) = (\beta_{i-k(i)}^{p+s_1+\ldots + s_{k(i)}-i+k(i)}, \id)
$$
if $i\notin J$.
In other words, after $(\alpha_p, \sigma_p)$ we disjoin every $(\alpha_i, \sigma_i)$ into
$(\alpha_i, \id)$ and $(\id, \sigma_i)$, commute some $(\id, \sigma_i)$ towards the left
according to the relations in $\A\rtimes S_n$,
multiply some of them together, and rotate some $(\beta_q^m,\id)$ from the right end
into the first term in the chain, all of it so that $m+2$ total terms remain.
We think of $(s_1,\ldots, s_r)$ as the ordered partition of the set $\{\sigma_{p+1},\ldots, \sigma_m\}$ and
the set $J$ as the set of indices where the products of $(\id, \sigma_i)$ end up.
Then $\sigma_0, \ldots, \sigma_{p+s_1+\ldots+s_{k(i)}}$ are the ones that are to the left of the term number $i$.

The sign in the second summation is given by
\begin{equation}
\label{eqn-sign-main-homotopy-xinu-id}
\left(\sum\limits_{i=0}^{m-p-r} \deg\alpha_i\right)\cdot 
\left(\sum\limits_{i=m-p-r+1}^m \deg\alpha_i\right)
+(p+m-1)r
+\sum\limits_{j\in J} j-p-k(j).
\end{equation}
\end{proof}

\subsection{From $\hochcx_\bullet(\A^{\otimes
n};\sigma_{\underline{\lambda}})_{C(\sigma_{\underline{\lambda}})}$ 
to $\Sym^{\underline{r}(\underline{\lambda})}\hochcx_\bullet(\A)$}
\label{section-from-underline-n-twisted-hh-to-sym-underline-n-hh}

For each conjugacy class $\underline{\lambda} \vdash n$ of $S_n$, choose
the representative 
$$ \sigma_{\underline{\lambda}} := \bigl(1\dots n_1\bigr) 
\bigl((n_1+1) \dots (n_1 + n_2) \bigr) \dots \bigl( (n_1 + \dots +
n_{r(\underline{\lambda}-1)} + 1) \dots n \bigr) \in S_n $$
where $n_1 \leq n_2 \leq \dots \leq n_{r(\underline{\lambda})}$ is the 
unique ordering of the parts of $\underline{\lambda}$ in non-decreasing
order. For example if $\underline{\lambda} = \left\{ 
1 1 1 2 3 3 3 3 4 4 \right\}$, then 
$$ \sigma_{\underline{\lambda}} = (1) (2) (3) (45) (678) (9\;10\;11) (12\;
13\; 14) (15\; 16\; 17) (18\; 19\; 20\; 21) (22\; 23\; 24\; 25). $$

This (or any other) choice defines an isomorphism 
\begin{equation}
\label{eqn-choice-of-conjugacy-class-representative-iso}
\Bigl(
\bigoplus_{ \sigma \in S_n}
\hochcx_{\bullet}(\A; \sigma)
\Bigr)_{S_n}
\xrightarrow{\sim}
\bigoplus_{\underline{\lambda} \vdash n}
\hochcx_{\bullet}(\A^{\otimes
n};\sigma_{\underline{\lambda}})_{C(\sigma_{\underline{\lambda}})},
\end{equation}
where $C(-)$ denotes the centraliser of a group element. This
isomorphism is given by the projection to the corresponding direct
summands. 

We then define:

\begin{defn}
\label{defn-maps-f-underline-n-and-g-underline-n}
Let $\A$ be a small DG category. For any $n \geq 0$ and 
any $\underline{\lambda} \vdash n$, define the maps
\begin{small}
 
\begin{equation}
\begin{tikzcd}
\hochcx_\bullet(\A^{\otimes
n};\sigma_{\underline{\lambda}})_{C(\sigma_{\underline{\lambda}})}
\ar[shift left = 1ex]{r}{f_{\underline{\lambda}}}
&
\Sym^{\underline{r}(\underline{\lambda})} \hochcx_\bullet(\A)
\ar[shift left = 1ex]{l}{g_{\underline{\lambda}}}
\end{tikzcd}
\end{equation}
to be given by the composition
\begin{equation*}
\begin{tikzcd}
\hochcx_\bullet(\A^{\otimes n};\sigma_{\underline{\lambda}})_{C(\sigma_{\underline{\lambda}})}
\ar[shift left = 1ex]{r}{\awmap}
&
\ar[shift left = 1ex]{l}{\ezmap}
\Bigl(\bigotimes_{i = 1}^{r(\underline{\lambda})} \hochcx_\bullet(\A^{\otimes
n_i};t_{n_i})_{t_{n_i}}\Bigr)_{S_{\underline{r}(\underline{\lambda})}}
\ar[shift left = 1ex]{r}{\otimes f_{n_i}}
&
\ar[shift left = 1ex]{l}{\otimes g_{n_i}}
\Bigl(\bigotimes_{i = 1}^{r(\underline{\lambda})} \hochcx_\bullet(\A) \Bigr)_{S_{\underline{r}(\underline{\lambda})}}
\ar[equals]{d}
\\
&
&
\Sym^{\underline{r}(\underline{\lambda})} \hochcx_\bullet(\A),
\end{tikzcd}
\end{equation*}
\end{small}
where 
\begin{itemize}
\item $\ezmap$ and $\awmap$ are the Eilenberg-Zilber and the
Alexander-Whitney maps which lift to the level of
Hochschild complexes the K{\"u}nneth isomorphism and its inverse, 
cf.~\S\ref{subsection-kunneth-isomorphism}.
We first apply these without the coinvariants,
using $\A^{\otimes{n}} =
\bigotimes_{i=1}^{r(\underline{\lambda})} \A^{n_i}$ and 
$\sigma_{\underline{\lambda}} = t_{n_1} \times \dots \times t_{n_{r(\underline{\lambda})}} 
\in S_{n_1} \times \dots \times S_{n_{r(\underline{\lambda})}} < S_n$. 
To see that they descend to the coinvariant spaces, we observe
that $\ezmap$ and $\awmap$ intertwine the action of each $t_{n_i}$
on the RHS with the action on the LHS of 
$t_{n_i} \in S_{n_i} < S_{n_1} \times \dots \times
S_{n_{r(\underline{\lambda})}} < S_n$. Furthermore, they
intertwine the action of $S_{\underline{r}(\underline{\lambda})}$
on the LHS and with the action on the RHS of 
$S_{\underline{r}(\underline{\lambda})} < S_n$, the subgroup which
permutes the parts of the same size in $n_1, \dots,
n_{\underline{r}(\underline{\lambda})}$. Finally,  
$C(t_{n_1} \times \dots \times t_{n_{r(\underline{\lambda})}})$
is generated by $\left<t_{n_1}\right> \times \dots 
\left<t_{n_{r(\underline{\lambda})}}\right>$ 
and $S_{\underline{r}(\underline{\lambda})}$. 
\item The maps $f_{n_i}$ and $g_{n_i}$ are the mutually inverse
homotopy equivalences constructed in \S\ref{section-long-cycle-case},
cf.~\S\ref{section-map-f} and \S\ref{section-map-g}. 

\item The second vertical equality is due to our notation
for the symmetric power indexed by an ordered partition,
cf.~\S\ref{section-symmetric-powers}.

\end{itemize}
\end{defn}

\subsection{Noncommutative orbifold decomposition}

We now put together all the maps we constructed
so far to prove the main theorem of this paper:

\begin{theorem}[Noncommutative orbifold decomposition]
\label{theorem-noncommutative-baranovsky-decomposition}
Let $\A$ be a small DG category and $n \geq 0$. The following
compositions 
are mutually inverse homotopy equivalences of twisted complexes over $\modk$:
\begin{small}
\begin{equation}
\label{eqn-noncommutative-baranovsky-decomposition-complexes}
\begin{tikzcd}
\hochcx_\bullet(\sym^n A)
\ar[shift left = 1ex]{r}{\nu}
&
\Bigl(
\bigoplus_{ \sigma \in S_n}
\hochcx_{\bullet}(\A^{\otimes n};\sigma)
\Bigr)_{S_n} 
\ar[shift left = 1ex]{l}{\xi}
\ar[shift left = 1ex]{d}{\eqref{eqn-choice-of-conjugacy-class-representative-iso}}
& 
\\
&
\bigoplus_{\underline{\lambda} \vdash n}
\hochcx_{\bullet}(\A^{\otimes n};\sigma_{\underline{\lambda}})_{C(\sigma_{\underline{\lambda}})}
\ar[shift left = 1ex]{u}{\eqref{eqn-choice-of-conjugacy-class-representative-iso}^{-1}}
\ar[shift left = 1ex]{r}{\sum f_{\underline{\lambda}}}
&
\ar[shift left = 1ex]{l}{\sum g_{\underline{\lambda}}}
\bigoplus_{\underline{\lambda} \vdash n}
\Sym^{\underline{r}(\underline{\lambda})} \hochcx_{\bullet}(\A),
\end{tikzcd}
\end{equation}
\end{small}
where $\nu$ and $\xi$ are the maps constructed in 
\S\ref{section-from-hh-of-quotient-stack-to-twisted-hhs}
and $f_{\underline{\lambda}}$ and $g_{\underline{\lambda}}$ -- in 
\S\ref{section-from-underline-n-twisted-hh-to-sym-underline-n-hh}. 

In particular, denote the rightward through composition by $\zeta_n$ and 
the leftward by $\eta_n$. These induce mutually inverse isomorphisms: 
\begin{equation}
\label{eqn-noncommutative-baranovsky-decomposition-hh}
\begin{tikzcd}
\hochhom_\bullet(\sym^n \A)
\ar[shift left = 1.25ex]{r}{\zeta_n}[']{\simeq}
&
\ar[shift left = 1.25ex]{l}{\eta_n}
\bigoplus_{\underline{\lambda} \vdash n}
\Sym^{\underline{r}(\underline{\lambda})} \hochhom_{\bullet}(\A).
\end{tikzcd}
\end{equation}
\end{theorem}
\begin{proof}
By Prop.~\ref{prps-nu-and-xi-are-mutually-inverse-homotopy-equivalences},
the maps $\nu$ and $\xi$ are mutually inverse homotopy equivalences.
By construction, the 
map $\eqref{eqn-choice-of-conjugacy-class-representative-iso}$ is an
isomorphism. Finally, by their construction in 
Defn.~\ref{defn-maps-f-underline-n-and-g-underline-n}, 
the maps $f_{\underline{\lambda}}$ and $g_{\underline{\lambda}}$ are 
compositions of the Eilenberg-Zilber and the
Alexander-Whitney maps $\ezmap$ and $\awmap$ with the maps 
$\otimes f_{n_i}$ and $\otimes g_{n_i}$. The former two are
mutually inverse homotopy equivalences lifting the K{\"u}nneth
isomorphism and its inverse, cf.~\S\ref{subsection-kunneth-isomorphism}.
The latter two are mutually inverse homotopy equivalences by Theorems
\ref{theorem-phi-is-the-homotopy-of-gf-minus-id} and 
\ref{theorem-psi-is-the-homotopy-of-fg-minus-id}. 

Finally, over a field of characteristic $0$ taking the coinvariants is an
exact functor. Thus passing to the cohomology in
\eqref{eqn-noncommutative-baranovsky-decomposition-complexes}
we obtain
\eqref{eqn-noncommutative-baranovsky-decomposition-hh}. 
\end{proof}

\subsection{Symmetric algebra structure}

In this section, we observe that our noncommutative orbifold decomposition 
\eqref{eqn-noncommutative-baranovsky-decomposition-complexes} defines
an explicit isomorphism between the total Hochschild homology 
$\bigoplus_{n \geq 0} \hochhom_\bullet(\sym^n \A)$ and the
symmetric algebra $S^* \left( \hochhom_\bullet(\A) \otimes t
\kk[t] \right)$. We then use this in 
\S\ref{section-hopf-algebra-and-lambda-ring-structures} to translate 
some well-known structures which exist on the symmetric algebra to 
the total Hochschild homology. 

We first introduce some notation. Recall that for any $E \in
\pretriagmns(\modk)$ and any ordered partition $(n_1, \dots, n_m)$
of $n$ we write 
$$\Sym^{n_1, \dots, n_m} E := (E^{\otimes n})_{S_{n_1, \dots, n_m}} \simeq 
\Sym^{n_1} E \otimes \Sym^{n_2} E \otimes \dots \otimes \Sym^{n_m} E.$$
View $t \kk[t]$ as a $\kk$-vector space and hence as an object in $\modk$. 
It has a basis $\left\{ t, t^2, t^3, \dots, \right\}$. 
Given any $E \in \pretriagmns(\modk)$, we have 
$$ E \otimes t\kk[t] \simeq \bigoplus_{i \geq 1} E \otimes t^i. $$
Write, for brevity, $E t^i$ for $E \otimes t^i$. We have
\begin{align}
S^*(E \otimes t \kk[t]) := 
& \bigoplus_{m \geq 0} \Sym^{m}(E \otimes t \kk[t])
\simeq \bigoplus_{m \geq 0} \Sym^{m}(\bigoplus_{i \geq 1} E t^i) \simeq 
\\
\nonumber
\simeq 
& 
\bigoplus_{m \geq 0} \;\;
\bigoplus_{
\begin{smallmatrix}
m_1, m_2, \dots \in \mathbb{Z}_{\geq 0}, \\
\sum m_i = m
\end{smallmatrix}
}
\Sym^{m_1,m_2,m_3,\dots} E\;\; (t^1)^{m_1}(t^2)^{m_2}\dots 
\end{align}
As we view $t \kk[t]$ only as a vector space, in 
the symmetric algebra $(t^i)^j \neq t^{ij}$. 

We next regroup the summands using the fact the each 
ordered partition $m_1, m_2, \dots$ of $m$ is
$\underline{r}(\underline{\lambda})$ of the unique unordered partition 
$\underline{\lambda} \vdash n$ where $n = \sum  m_i i$ and 
$\underline{\lambda}$ has $m_1$ parts of size 1, $m_2$ parts of size $2$, etc. 
For example, in the sum above the ordered partition $(4,0,1,2,0,\dots)$ 
of $7$ indexes the monomial 
$$ (t^1)^4 t^3 (t^4)^2 = t^1 t^1 t^1 t^1 t^3 t^4 t^4. $$
We can think of the same monomial as being indexed by the
unordered partition $\left\{1111344\right\}$ of $15$. 
We thus define for any $n \geq 0$ and any $\underline{\lambda} \vdash n$
$$ t^{\underline{\lambda}} := t^{n_1}t^{n_2} \dots t^{n_{r(\underline{\lambda})}}, $$
and regroup the above as:
\begin{equation}
\bigoplus_{n \geq 0} \;\;
\bigoplus_{
\underline{\lambda}\vdash n 
}
\left(\Sym^{\underline{r}(\underline{\lambda})} E \right)t^{\underline{\lambda}}. 
\end{equation}

We introduce the following notation. For any $n \geq 0$ and
$\underline{\lambda} \vdash n$, any element of 
$\Sym^{\underline{r}(\underline{\lambda})} E$ can be written as 
the sum of basic elements 
$$ \underline{s}_1 \otimes \dots \otimes \underline{s}_n
\quad \quad \underline{s}_i \in \Sym^{\underline{r}(\underline{\lambda})} E,
$$
and we further write
$$ \underline{s}_i = s_{i1} \vee \dots \vee s_{ir_i(\underline{\lambda})} \quad \quad s_{ij} \in E. $$
Here we distinguish between the skew-commutative product $\vee$ within
each $\Sym^{r_i(\underline{\lambda})} E$ and ordinary tensor product $\otimes$ in 
$\Sym^{r_1(\underline{\lambda})} E \otimes \dots \otimes 
\Sym^{r_n(\underline{\lambda})} E$. 

For example, when $n = 15$ and $\underline{\lambda} = (1111344)$, we
have
$$ t^{\underline{\lambda}} = t^1 t^1 t^1 t^1 t^3 t^4 t^4, $$
$$ \Sym^{\underline{r}(\underline{\lambda})} = 
\Sym^4 E \otimes E \otimes \Sym^2 E, $$
and any element of $\Sym^{\underline{r}(\underline{\lambda})} E$ 
can be written as a sum of the elements of form
$$ \underline{s} = (s_{11} \vee s_{12} \vee s_{13} \vee s_{14})
\otimes (s_{31}) \otimes (s_{31} \vee s_{32}). $$

We thus have an isomorphism 
\begin{equation}
\label{eqn-symmetric-algebra-decomposition}
S^*(E \otimes t \kk[t]) \simeq 
\bigoplus_{n \geq 0} \;\;
\bigoplus_{
\underline{\lambda} \vdash n 
}
\left(\Sym^{\underline{r}(\underline{\lambda})} E\right) t^{\underline{\lambda}}. 
\end{equation}
It sends any
$$ s_1 t^{n_1} \vee s_2 t^{n_2} \vee \dots \vee s_{m} t^{n_m} \in
S^*(E \otimes t \kk[t]) $$
with $n_1 \leq n_2 \leq \dots \leq n_m$ to 
$$ (s_1 \vee \dots \vee s_{m_1}) \otimes
(s_{m_1+1} \vee \dots \vee s_{m_1+m_2}) \otimes \dots \otimes
(s_{m_1 + \dots + m_{m-1} + 1} \vee \dots \vee s_n) t^{\underline{\lambda}} $$
where $n = \sum n_i$, $\underline{\lambda} = \left\{n_1, \dots, n_m\right\}$
and $m_i = r_i(\underline{\lambda})$. 
The inverse isomorphism sends any 
$$ \underline{s}_1 \otimes \dots \otimes \underline{s}_n
t^{\underline{\lambda}} \in \left(\Sym^{\underline{r}(\underline{\lambda})}
E\right) t^{\underline{\lambda}},
$$
to 
$$ (s_{11} t^1) \vee (s_{12} t^1) \vee \dots \vee (s_{n
r_n(\underline{\lambda})} t^n). $$

We thus have the folowing description of the total Hochschild homology of $\sym^n \A$: 

\begin{theorem}
\label{theorem-total-hh-of-sym-stacks-iso-to-sym-algebra-of-hh-otimes-tkt}
Let $\A$ be a small DG category. Define the maps 
\begin{small}
\label{eqn-symmetric-algebra-homotopy-equivalence-complexes}
\begin{equation}
\begin{tikzcd}
\bigoplus_{n \geq 0} \hochcx_\bullet(\sym^n A)
\ar[shift left = 1ex]{r}{\zeta}
&
S^*(\hochcx_\bullet(\A) \otimes t \kk[t])
\ar[shift left = 1ex]{l}{\eta}
\end{tikzcd}
\end{equation}
\end{small}
to be the compositions 
\begin{small}
\label{eqn-symmetric-algebra-homotopy-equivalence-complexes}
\begin{equation*}
\begin{tikzcd}
\bigoplus_{n \geq 0} \hochcx_\bullet(\sym^n A)
\ar[shift left = 1ex]{r}{\sum_n \zeta_n}
&
\ar[shift left = 1ex]{l}{\sum_n \eta_n}
\bigoplus_{n \geq 0} 
\bigoplus_{
\underline{\lambda} \vdash n 
}
\Sym^{\underline{r}(\underline{\lambda})} \hochcx_\bullet(\A)\; t^{\underline{\lambda}}
\ar[shift left = 1ex]{r}{\eqref{eqn-symmetric-algebra-decomposition}^{-1}}
&
\ar[shift left = 1ex]{l}{\eqref{eqn-symmetric-algebra-decomposition}}
S^*(\hochcx_\bullet(\A) \otimes t \kk[t])
\end{tikzcd}
\end{equation*}
\end{small}

Then $\zeta$ and $\eta$ are mutually inverse homotopy equivalences. In
particular, they induce mutually inverse isomorphims:
\begin{equation}
\label{eqn-symmetric-algebra-homotopy-equivalence-hh}
\begin{tikzcd}
\bigoplus_{n \geq 0} \hochhom_\bullet(\sym^n A)
\ar[shift left = 1.24ex]{r}{\zeta}[']{\simeq}
&
S^*(\hochhom_\bullet(\A) \otimes t \kk[t]). 
\ar[shift left = 1.25ex]{l}{\eta}
\end{tikzcd}
\end{equation}
\end{theorem}
\begin{proof}
The map \eqref{eqn-symmetric-algebra-decomposition} is an isomorphism, 
and by Theorem \ref{theorem-noncommutative-baranovsky-decomposition}
the maps $\sum_n \zeta_n$ and $\sum_n \eta_n$ are mutually inverse
homotopy equivalences.

\end{proof}

\section{Hopf algebra and $\lambda$-ring structures}
\label{section-hopf-algebra-and-lambda-ring-structures}

In this section we use explicit isomorphisms $\eta$ and $\zeta$
between  $\bigoplus_{n \geq 0} \hochhom_\bullet(\sym^n A)$
and $S^*(\hochhom_\bullet(\A) \otimes t \kk[t])$ established in $\S4$ 
to compare -- and define -- some structures on these spaces.  

\subsection{Hopf algebra structure}

For any $\kk$-vector space $E$ its tensor algebra $T^*(E)$ is
naturally a Hopf algebra, and the symmetric algebra $S^*(E)$
is its Hopf subalgebra, cf. \cite[\S4.0]{Sweedler-HopfAlgebras}.
For $S^*(E)$ the Hopf algebra structure agrees with the universal 
enveloping algebra of the Lie algebra $E$ with the zero bracket.

We recall the details to introduce the notation:

\begin{defn}
\label{defn-hopf-algebra-structure-on-the-symmetric-algebra}
Let $\A$ be a small DG category and let $E \in \pretriagmns(\A)$. 
The symmetric algebra $S^*(E)$ is a Hopf algebra with the following
operations:

\begin{itemize}
\item Multiplication 
$\mu\colon S^*(E) \otimes S^*(E) \rightarrow S^*(E)$ given by 
$$ (s_1 \vee \dots \vee s_n) \otimes (t_1 \vee \dots \vee t_m) 
\rightarrow s_1 \vee \dots \vee s_n \vee t_1 \vee \dots \vee t_m, $$

\item Comultiplication
$\Delta\colon S^*(E) \rightarrow S^*(E) \otimes S^*(E) $ given by 
$$ (s_1 \vee \dots \vee s_n) 
\rightarrow \sum_{
\left\{ i_1, \dots, i_m \right\} 
\subseteq \left\{ 1, \dots, n \right\} }
(-1)^{p}
(s_{i_1} \vee \dots \vee s_{i_m}) \otimes (s_{j_1} \vee \dots \vee
s_{j_{n-m}}), $$
where $\left\{ j_1, \dots, j_{n-m} \right\} = \left\{1,\dots, n\right\}
\setminus \left\{ i_1, \dots, i_m \right\}$ and the sign twist $(-1)^p$ 
is the product of $(-1)^{\deg(s_{i_k})\deg(s_{j_{l}})}$ for 
each $i_k > j_l$,

\item Unit $u\colon \kk \rightarrow S^*(E)$ and counit 
$\epsilon\colon S^*(E)\rightarrow \kk$ given by the inclusion 
of and the projection onto $S^0(E) = \kk$,

\item Antipode map $H\colon S^*(E) \rightarrow S^*(E)$ given 
on $S^n(E)$ by $(-1)^n \id$. 
\end{itemize}
\end{defn}

In particular, this defines a Hopf algebra structure on the symmetric
algebra $S^*(\hochhom_\bullet(\A) \otimes t \kk[t])$. We next compute
the corresponding Hopf algebra structure on 
$\bigoplus_{n \geq 0} \hochhom_\bullet(\sym^n A)$:

\begin{theorem}
\label{theorem-hopf-algebra-structure-on-the-total-hh-of-symn-a}
Let $\A$ be a small DG category. The isomorphisms $\zeta$ and $\eta$
of Theorem 
\ref{theorem-total-hh-of-sym-stacks-iso-to-sym-algebra-of-hh-otimes-tkt} 
\begin{equation*}
\begin{tikzcd}
\bigoplus_{n \geq 0} \hochhom_\bullet(\sym^n \A)
\ar[shift left = 1.24ex]{r}{\zeta}[']{\simeq}
&
S^*(\hochhom_\bullet(\A) \otimes t \kk[t]). 
\ar[shift left = 1.25ex]{l}{\eta}
\end{tikzcd}
\end{equation*}
identify the Hopf algebra operations $\mu$, $\Delta$, $u$, $\epsilon$,
and $H$ on $S^*(\hochhom_\bullet(\A) \otimes t \kk[t])$ defined
in Definition \ref{defn-hopf-algebra-structure-on-the-symmetric-algebra}
with the following operations on 
$\bigoplus_{n \geq 0} \hochhom_\bullet(\sym^n \A)$:
\begin{itemize}
\item The map  
$$\mu'\colon
\bigl(\bigoplus_{n \geq 0} \hochhom_\bullet(\sym^n \A)\bigr)
\otimes 
\bigl(\bigoplus_{n \geq 0} \hochhom_\bullet(\sym^n \A)\bigr)
\rightarrow 
\bigoplus_{n \geq 0} \hochhom_\bullet(\sym^n \A)$$ 
is given by the sum over all $n,m \geq 0$ of the maps
\begin{equation}
\label{eqn-multiplication-in-hopf-algebra-structure-on-the-total-hh-of-symn}
\hochcx_\bullet(\sym^n \A) \otimes \hochcx_\bullet(\sym^m \A)
\xrightarrow{\ezmap}
\hochcx_\bullet(\sym^n \A \otimes \sym^m \A)
\xrightarrow{\Ind_{S_n \times S_m}^{S_{n+m}}}
\hochcx_\bullet(\sym^{n+m} \A). 
\end{equation}
\item The map 
$$\Delta'\colon
\bigoplus_{n \geq 0} \hochhom_\bullet(\sym^n \A)
\rightarrow 
\bigl(\bigoplus_{n \geq 0} \hochhom_\bullet(\sym^n \A)\bigr)
\otimes 
\bigl(\bigoplus_{n \geq 0} \hochhom_\bullet(\sym^n \A)\bigr)
$$ 
given by the sum over all $n \geq 0$ and all subsets
$I \subseteq \left\{ 1,\dots, n \right\}$ of the maps
\begin{small}
\begin{equation*}
\begin{tikzcd}
\hochcx_\bullet(\sym^n \A) 
\ar{r}{\Res_{S_{I} \times S_{\bar{I}}}^{S_{n}}}
&
\hochcx_\bullet\left(\A^{\otimes n} \rtimes (S_I \times S_{\bar{I}})\right)
\ar{d}{\simeq}
&
\\
&
\hochcx_\bullet(\sym^{|I|} \A \otimes \sym^{n-|I|} \A)
\ar{r}{\awmap}
&
\hochcx_\bullet(\sym^{|I|} \A) \otimes \hochcx_\bullet(\sym^{n-|I|} \A), 
\end{tikzcd}
\end{equation*}
\end{small}
where $\bar{I} =  \left\{ 1, \dots, n \right\} \setminus I$, 
where $S_I, S_{\bar{I}} < S_n$ 
are the subgroups that only permute 
the elements of $I$ and of $\bar{I}$, and where $|I|$ is the size of $I$. 
\item The maps $u'\colon \kk \rightarrow 
\bigoplus_{n \geq 0} \hochhom_\bullet(\sym^n \A)$
and $\epsilon'\colon 
\bigoplus_{n \geq 0} \hochhom_\bullet(\sym^n \A) \rightarrow \kk$ 
given by the inclusion of and the projection onto 
$\hochhom_\bullet(\sym^0 \A) \simeq \kk$, 
\item The map $H'\colon 
\bigoplus_{n \geq 0} \hochhom_\bullet(\sym^n \A) 
\rightarrow 
\bigoplus_{n \geq 0} \hochhom_\bullet(\sym^n \A)$
given on $\hochhom_\bullet(\sym^n \A)$ by $(-1)^n \id$. 
\end{itemize}
In particular, $\mu'$, $\Delta'$, $u'$, $\epsilon'$ and $H'$ define a
structure of a commutative Hopf algebra on 
$\bigoplus_{n \geq 0} \hochhom_\bullet(\sym^n \A)$. 
\end{theorem}
\begin{proof}
The assertions about the unit map $u$, the counit map $\epsilon$, 
and the antipode map $H$ being identified with the maps $u'$,
$\epsilon'$, and the map $H'$ are trivial. 

For the multiplication $\mu$, in view of the isomorphism 
\eqref{eqn-symmetric-algebra-decomposition} it suffices to prove that
for any $m,n \geq 0$, for any $\underline{\nu} \vdash m$ and 
$\underline{\lambda} \vdash n$, and for any $x \in \Sym^{\underline{\nu}}
\hochhom_\bullet(\A)$ and $y \in \Sym^{\underline{\lambda}} \hochhom_\bullet(\A)$
we have 
\begin{equation}
\label{eqn-verification-of-multiplication-structures-coinciding}
\zeta(\mu(x t^{\underline{\nu}} \otimes y t^{\underline{\lambda}})) = 
\mu'(\zeta(x t^{\underline{\nu}}) \otimes \zeta (y t^{\underline{\lambda}}))). 
\end{equation}

One of the composants of $\zeta$, 
the map $\xi$ which was defined in \eqref{eqn-definition-of-the-map-xi}
to be 
$$
\Bigl(
\bigoplus_{ g\in G}
\hochcx_{\bullet}(\A;g)
\Bigr)_{G} 
\xrightarrow{\eqref{eqn-coinvariant-embedding-map-symmetriser}}
\bigoplus_{ g\in G}
\hochcx_{\bullet}(\A;g)
\xrightarrow{\sum_{g \in G} \xi_g}
\hochcx_\bullet(\A \rtimes G)
$$
can be rewritten as
$$
\Bigl(
\bigoplus_{ g\in G}
\hochcx_{\bullet}(\A;g)
\Bigr)_{G} 
\xrightarrow{\sum_{g \in G} \xi_g}
\hochcx_\bullet(\A \rtimes G)_{G}
\xrightarrow{\eqref{eqn-coinvariant-embedding-map-symmetriser}}
\hochcx_\bullet(\A \rtimes G). 
$$
In view of this, 
\eqref{eqn-verification-of-multiplication-structures-coinciding}
is equivalent to the commutation of the following diagrams:
\begin{tiny}
\begin{equation}
\label{eqn-verif-mult-struct-commutation-top-diagram}
\begin{tikzcd}
\left(
\hochhom_\bullet(\A)^{\otimes r(\underline{\nu})}
\right)_{S_{\underline{r}(\underline{\nu})}}
 \otimes 
\left(
\hochhom_\bullet(\A)^{\otimes r(\underline{\lambda})}
\right)_{S_{\underline{r}(\underline{\lambda})}}
\ar{d}{(\otimes g_{m_i}) \otimes (\otimes g_{n_j})}
\ar[->>]{r}
&
\left(
\hochhom_\bullet(\A)^{\otimes r(\underline{\nu} + \underline{\lambda})}
\right)_{S_{\underline{r}(\underline{\nu} + \underline{\lambda})}}
\ar{d}{(\otimes g_{m_i}) \otimes (\otimes g_{n_j})}
\\
\left(
\bigotimes_{\underline{\nu}}
\hochhom_\bullet(\A^{\otimes\bullet};t_{\bullet})_{t_{\bullet}}
\right)_{S_{\underline{r}(\underline{\nu})}}
\otimes 
\left(
\bigotimes_{\underline{\lambda}}
\hochhom_\bullet(\A^{\otimes\bullet};t_{\bullet})_{t_{\bullet}}
\right)_{S_{\underline{r}(\underline{\lambda})}}
\ar[->>]{r}
\ar{d}{\ezmap \otimes \ezmap}
& 
\left(
\bigotimes_{\underline{\nu} + \underline{\lambda}}
\hochhom_\bullet(\A^{\otimes\bullet};t_{\bullet})_{t_{\bullet}}
\right)_{S_{\underline{r}(\underline{\nu} + \underline{\lambda})}}
\ar{d}{\ezmap}
\\
\hochhom_\bullet(\A^{\otimes m}; \sigma_{\underline{\nu}})_{C(\sigma_{\underline{\nu}})}
\otimes 
\hochhom_\bullet(\A^{\otimes n}; \sigma_{\underline{\lambda}})_{C(\sigma_{\underline{\lambda}})}
\ar{r}{\ezmap}
&
\hochhom_\bullet(\A^{\otimes m + n};
\sigma_{\underline{\nu}}\times\sigma_{\underline{\lambda}})_{C(\sigma_{\underline{\nu}} \times \sigma_{\underline{\lambda}})}
\end{tikzcd}
\end{equation}
\begin{equation}
\label{eqn-verif-mult-struct-commutation-bottom-left-diagram}
\begin{tikzcd}
\hochhom_\bullet(\A^{\otimes m}; \sigma_{\underline{\nu}})_{C(\sigma_{\underline{\nu}})}
\otimes 
\hochhom_\bullet(\A^{\otimes n}; \sigma_{\underline{\lambda}})_{C(\sigma_{\underline{\lambda}})}
\ar{r}{\ezmap}
\ar{d}{\xi_{\sigma_{\underline{\nu}}}\otimes\xi_{\sigma_{\underline{\lambda}}}}
& 
\hochhom_\bullet(\A^{\otimes m + n};
\sigma_{\underline{\nu}}\times\sigma_{\underline{\lambda}})_{C(\sigma_{\underline{\nu}})
\times C(\sigma_{\underline{\lambda}})}
\ar{d}{\xi_{\sigma_{\underline{\nu}} \times \sigma_{\underline{\lambda}}}}
\\
\hochhom_\bullet(\sym^m \A)_{S_m}
\otimes 
\hochhom_\bullet(\sym^n \A)_{S_n}
\ar{r}{\ezmap}
\ar{d}{ \eqref{eqn-coinvariant-embedding-map-symmetriser} \otimes \eqref{eqn-coinvariant-embedding-map-symmetriser} }
&
\hochhom_\bullet(\sym^m \A \otimes \sym^n \A)_{S_m \times S_n}
\ar{d}{\eqref{eqn-coinvariant-embedding-map-symmetriser}}
\\
\hochhom_\bullet(\sym^m \A)
\otimes 
\hochhom_\bullet(\sym^n \A)
\ar{r}{\ezmap}
&
\hochhom_\bullet(\sym^m \A \otimes \sym^n \A)
\end{tikzcd}
\end{equation}
\begin{equation}
\label{eqn-verif-mult-struct-commutation-bottom-right-diagram}
\begin{tikzcd}
\hochhom_\bullet(\A^{\otimes m + n};
\sigma_{\underline{\nu}}\times\sigma_{\underline{\lambda}})_{C(\sigma_{\underline{\nu}})
\times C(\sigma_{\underline{\lambda}})}
\ar[->>]{r}
\ar{d}{\xi_{\sigma_{\underline{\nu}} \times \sigma_{\underline{\lambda}}}}
&
\hochhom_\bullet(\A^{\otimes m + n};
\sigma_{\underline{\nu}}\times\sigma_{\underline{\lambda}})_{C(\sigma_{\underline{\nu}} \times \sigma_{\underline{\lambda}})}
\ar{d}{\xi_{\sigma_{\underline{\nu}} \times \sigma_{\underline{\lambda}}}}
\\ 
\hochhom_\bullet(\sym^m \A \otimes \sym^n \A)_{S_m \times S_n}
\ar{r}{\Ind_{S_m \times S_n}^{S_{m+n}}}
\ar{d}{\eqref{eqn-coinvariant-embedding-map-symmetriser}}
&
\hochhom_\bullet(\sym^{m+n} \A)_{S_{m+n}}
\ar{d}{\eqref{eqn-coinvariant-embedding-map-symmetriser}}
\\
\hochhom_\bullet(\sym^m \A \otimes \sym^n \A)
\ar{r}{\Ind_{S_m \times S_n}^{S_{m+n}}}
&
\hochhom_\bullet(\sym^{m+n} \A)
\end{tikzcd}
\end{equation}
\end{tiny}
where:
\begin{itemize}
\item $\bigotimes_{\underline{\nu}}
\hochcx_\bullet(\A^{\otimes\bullet};t_{\bullet})_{t_{\bullet}}$ 
denotes $\bigotimes_{i=1}^{r(\underline{\nu})} 
\hochcx_\bullet(\A^{\otimes m_i};t_{m_i})_{t_{m_i}}$ 
where we use the unique ordering 
$m_1 \leq \dots \leq m_{r(\underline{\nu})}$ 
of the parts of $\underline{\nu}$ in the non-decreasing order. 
\item Two-headed arrows denote the quotient maps due to taking
the coinvariants by the action of a bigger group.
\item The arrows labeled $(\otimes g_{m_i}) \otimes (\otimes g_{n_j})$, 
$\ezmap$, $\xi_\bullet$ and $\xi_\bullet \otimes \xi_\bullet$
mean the images of the corresponding maps after taking 
the coinvariants as indicated. 
\end{itemize}

These diagrams fit into one big commutative diagram as:
\begin{center}
\begin{tabular}{|c|c|}
\hline
\multicolumn{2}{|c|}{\eqref{eqn-verif-mult-struct-commutation-top-diagram}}
\\
\hline
\eqref{eqn-verif-mult-struct-commutation-bottom-left-diagram}
&
\eqref{eqn-verif-mult-struct-commutation-bottom-right-diagram}
\\
\hline
\end{tabular}
\end{center}
In this big diagram the bottom-left half of the perimeter composes
to give the LHS of 
\eqref{eqn-verification-of-multiplication-structures-coinciding}, 
and the top-right half composes to give the RHS of
\eqref{eqn-verification-of-multiplication-structures-coinciding}.

The top square in \eqref{eqn-verif-mult-struct-commutation-top-diagram} 
commutes tautologically: the right vertical
arrow is by definition the unique arrow making this
square commute. The bottom square in 
\eqref{eqn-verif-mult-struct-commutation-top-diagram}
commutes by the associativity of the Eilenberg-Zilber map.   

The top square in
\eqref{eqn-verif-mult-struct-commutation-bottom-left-diagram}
commutes since the maps $\xi_\bullet$ are compatible with 
the Eilenberg-Zilber maps
\cite[Lemma
5.17]{GyengeLogvinenko-TheHeisenbergAlgebraOfAVectorSpaceAndHochschildHomology}.
The bottom square in 
\eqref{eqn-verif-mult-struct-commutation-bottom-left-diagram}
commutes since the symmetrizing maps
\eqref{eqn-coinvariant-embedding-map-symmetriser} are compatible
with the Eilenberg-Zilber maps: the former permute the columns
of the Hochschild chains of tensor powers of $\A$, while the latter
shuffle the rows while preserving the column positions. 

The top square in 
\eqref{eqn-verif-mult-struct-commutation-bottom-right-diagram}
commutes because the induction functor 
$$
\Ind_{S_m \times S_n}^{S_{m+n}}\colon 
\hperf\left(\A^{n+m} \rtimes (S_m \times S_n)\right)
\rightarrow 
\hperf\left(\A^{n+m} \rtimes S_{m+n} \right)
$$
restricts on the representables to the fully faithful embedding
$$
\A^{n+m} \rtimes (S_m \times S_n)
\hookrightarrow 
\A^{n+m} \rtimes S_{m+n}
$$ 
given by inclusion $S_m \times S_n < S_{m+n}$. The bottom square in 
\eqref{eqn-verif-mult-struct-commutation-bottom-right-diagram}
commutes because for any small DG category $\A$ with strong
action of a finite group $G$, 
the action of $G$ on $\A \rtimes G$ becomes trivial on 
the level of Hochschild homology 
\cite[Lemma
5.15]{GyengeLogvinenko-TheHeisenbergAlgebraOfAVectorSpaceAndHochschildHomology}.
Thus $\hochhom_\bullet(\A \rtimes G)_G = \hochhom_\bullet(\A \rtimes G)$
and so $\eqref{eqn-coinvariant-embedding-map-symmetriser}$ is the identity map.    

We thus conclude that the multiplication $\mu$ is identified by
the isomorphisms $\zeta$ with $\mu'$. A similar analysis reduces 
the statement that the comultiplication $\Delta$ is identified with 
$\Delta'$ to the commutation of the diagram in 
\cite[Lemma
5.18]{GyengeLogvinenko-TheHeisenbergAlgebraOfAVectorSpaceAndHochschildHomology}. 
\end{proof}

\subsection{$K$-theoretic parallels and $\lambda$-ring structure}
\label{section-k-theoretic-parallels-and-lambda-ring-structure}

Let $\A$ be a small DG category. Let $\A$ be equivalent,  
as an enhanced triangulated category, to the
bounded derived category $D(X)$ of coherent sheaves on a smooth 
projective variety $X$. That is -- there is
an exact equivalence $D_c(\A) \simeq D(X)$ induced by
a Morita equivalence of $\A$ with the standard DG enhancement of $X$
\cite[\S4.4]{GyengeKoppensteinerLogvinenko-TheHeisenbergCategoryOfACategory}
\cite{Toen-TheHomotopyTheoryOfDGCategoriesAndDerivedMoritaTheory} 
\cite{LuntsOrlov-UniquenessOfEnhancementForTriangulatedCategories}. 
By the global version of HKR theorem 
\cite[Cor 2.6]{Swan-HochschildCohomologyOfQuasiprojectiveSchemes},
\cite[Theorem 4.6.2]{Kontsevich-DeformationQuantizationOfPoissonManifolds}, 
\cite[Theorem 4.9]{Caldararu-TheMukaiPairingIITheHochschildKostantRosenbergIsomorphism},
we have $HH_\bullet(\A) \simeq \bigoplus_{p \geq 0}H^{\bullet - p}(X, \Omega^p)$. When $\kk =
\mathbb{C}$, by Poincar{\`e} Lemma we further have $HH_\bullet(\A) \simeq 
H^\bullet(X,\mathbb{C})$ and then by the Chern character isomorphism 
\cite[\S2.4]{AtiyahHirzebruch-VectorBundlesAndHomogeneousSpaces} 
we have an isomorphism
$HH_\bullet(\A) \simeq K^\bullet(X) \otimes \mathbb{C}$
of $\mathbb{Z}/2$-graded vector spaces,
where $K^\bullet(X)$ denotes the topological $K$-theory of $X$.

For any finite group $G$ acting on $X$, in \cite[Theorem
2]{AtiyahSegal-OnEquivariantEulerCharacteristics} Atiyah and Segal
gave a decomposition of the complexified equivariant topological 
$K$-theory in terms of the fixed point loci of $G$-action: 
\begin{equation}
\label{eqn-decomposition-of-equivariant-topological-k-theory}
K_G^\bullet(X) \otimes \mathbb{C} \simeq \bigoplus_{[g] \in G}
K^\bullet(X_g)_{C(g)} \otimes \mathbb{C},
\end{equation}
where the sum is taken over the representatives of the conjugacy
classes of $G$, $X_g$ denotes the fixed point locus of $g \in G$, 
and $C(g)$ denotes the centraliser of $g \in G$. The equivariant
$K$-theory $K^\bullet_G(X) \otimes \mathbb{C}$ can be viewed as the $K$-theory
of the global quotient stack $[X/G]$, and the decomposition 
\eqref{eqn-decomposition-of-equivariant-topological-k-theory} can be
viewed as an earlier $K$-theoretic analogue of the decomposition 
of the orbifold Hochschild homology (and, more generally,
cyclic homology) of $X$ by the second author
\cite[Theorem 1.1]{Baranovsky-OrbifoldCohomologyAsPeriodicCyclicHomology}
that this paper generalises to the noncommutative case.

Segal \cite[Prop.~1.4]{Segal-EquivariantKTheoryAndSymmetricProducts}
and Wang \cite[Theorem
3]{Wang-EquivariantKTheoryWreathProductsAndHeisenbergAlgebra} used the
decomposition \eqref{eqn-decomposition-of-equivariant-topological-k-theory}
to identify the total symmetric quotient stack $K$-theory of $X$ with 
a symmetric algebra:
\begin{equation}
\label{eqn-k-theoretic-symmetric-ring-structure-on-total-symm-k-theory}
\bigoplus_{n \geq 0} K^\bullet_{S_n}(X^n) \otimes \mathbb{C}
\simeq 
S^* \left(K^\bullet(X) \otimes_\mathbb{C} t\mathbb{C}[t]\right). 
\end{equation}
This is a $K$-theoretic analogue of our Theorem 
\ref{theorem-total-hh-of-sym-stacks-iso-to-sym-algebra-of-hh-otimes-tkt}
for Hochschild homology. Segal and Wang used it to give 
$\bigoplus_{n \geq 0} K^\bullet_{S_n}(X^n) \otimes \mathbb{C}$
the structure of the Fock space for the Heisenberg algebra of $K^\bullet(X)$. 
A Hochschild homology analogue of this Heisenberg algebra
action for our noncommutative setting is established in
\cite{GyengeLogvinenko-TheHeisenbergAlgebraOfAVectorSpaceAndHochschildHomology}.  

Wang also defined Hopf algebra and $\lambda$-ring structures on  
$\bigoplus_{n \geq 0} K^\bullet_{S_n}(X^n) \otimes \mathbb{C}$. The Hopf
algebra structure 
\cite[\S2]{Wang-EquivariantKTheoryWreathProductsAndHeisenbergAlgebra}
is the image under the isomorphism 
\eqref{eqn-k-theoretic-symmetric-ring-structure-on-total-symm-k-theory}
of the standard Hopf algebra structure of a symmetric algebra. 
It is a $K$-theoretic analogue of
our Theorem \ref{theorem-hopf-algebra-structure-on-the-total-hh-of-symn-a}
for the Hochschild homology. 
The $\lambda$-ring structure 
\cite[\S3]{Wang-EquivariantKTheoryWreathProductsAndHeisenbergAlgebra}
was constructed using isomorphism 
\eqref{eqn-k-theoretic-symmetric-ring-structure-on-total-symm-k-theory}
and was shown to be isomorphic to the free $\lambda$-ring generated by
$K^\bullet(X) \otimes \mathbb{C}$. 

In this section, we investigate whether there is an analogous
$\lambda$-ring structure on $\bigoplus_{n \geq 0} \hochhom(\sym^n \A)$
in our noncommutative setting. 

A $\lambda$-ring is a notion introduced by Grothendieck 
\cite[\S{I}.1-2]{Grothendieck-ClassesDeFaisceauxEtTheoremeDeRiemannRoch}
axiomatizing the structure defined on the Grothendieck ring $K_0(X)$
of vector bundles on $X$ by the operation of taking the exterior powers. 
So a $\lambda$-ring is a unital, commutative
ring $R$ equipped with a family of set-theoretic operations
$\lambda^i\colon R \rightarrow R$ for $i \geq 0$ satisfying
certain conditions 
\cite[\S{I}.1-2]{Grothendieck-ClassesDeFaisceauxEtTheoremeDeRiemannRoch}
\cite[Déf.~2.4]{Berthelot-GeneralitesSurLesLambdaAnneaux}. 
We also refer the reader to
\cite{Knutson-LambdaRingsAndTheRepresentationTheoryOfTheSymmetricGroup}, 
\cite{AtiyahTall-GroupRepresentationsLambdaRingsAndTheJHomomorphism}, 
and more modern \cite{Yau-LambdaRings}. A $\lambda$-ring structure can also be 
defined in terms of the \em Adams operations \rm $\psi^i\colon R
\rightarrow R$, and over a field of characteristic $0$
these can be any ring homomorphisms satisfying $\psi^1 = \id$ and $\psi^{m} \circ \psi^{n} = \psi^{mn}$,
see 
\cite[p.~49]{Knutson-LambdaRingsAndTheRepresentationTheoryOfTheSymmetricGroup}
\cite[Prop.~1.2]{Wilkerson-LambdaRingsBinomialDomainsAndVectorBundlesOverCPInfty}.

Given any $E \in \modk$, the symmetric DG algebra $S^*(E \otimes t\kk[t])$ 
has a natural $\lambda$-ring structure given by 
\begin{equation}
\label{eqn-free-lambda-ring-structure-on-symmetric-algebra-of-E-tkt}
\psi^{m}(v t^n) := v t^{mn} \quad \quad \quad \forall\; v \in E, n \geq 0. 
\end{equation}
With this structure, $S^*(E \otimes t\kk[t])$ is the \em free
$\lambda$-ring generated by $E$\rm. It satisfies the universal
property that for any other DG $\lambda$-ring $R$ over $\kk$ whose Adams
operations are $\modk$-homomorphisms, the DG $\lambda$-ring
homomorphisms $S^*(E \otimes t\kk[t]) \rightarrow R$ are in bijection
with $\modk$-morphisms $E \rightarrow R$. The bijection
is given by the restriction to $E \otimes t \subset S^*(E \otimes
t\kk[t])$. 

Thus the isomorphisms $\zeta$ and $\eta$ of Theorem 
\ref{theorem-total-hh-of-sym-stacks-iso-to-sym-algebra-of-hh-otimes-tkt} 
\begin{equation*}
\begin{tikzcd}
\bigoplus_{n \geq 0} \hochhom_\bullet(\sym^n \A)
\ar[shift left = 1.24ex]{r}{\zeta}[']{\simeq}
&
S^*(\hochhom_\bullet(\A) \otimes t \kk[t])
\ar[shift left = 1.25ex]{l}{\eta}
\end{tikzcd}
\end{equation*}
give the total Hochschild homology 
$\bigoplus_{n \geq 0} \hochhom_\bullet(\sym^n \A)$
the structure of a graded $\lambda$-ring with  multiplication 
\eqref{eqn-multiplication-in-hopf-algebra-structure-on-the-total-hh-of-symn}
and  Adams operations 
$$ \psi^m\colon \hochhom_\bullet(\sym^n \A) 
\rightarrow \hochhom_\bullet(\sym^{mn} \A). $$
For $n = 1$ they are given by the composition
$$ \hochcx_\bullet(\A) \xrightarrow{g_m}
\hochcx_\bullet(\A^{\otimes m}; t_m)_{t_m} 
\xrightarrow{\eqref{eqn-coinvariant-embedding-map-symmetriser}}
\hochcx_\bullet(\A^{\otimes m}; t_m)
\xrightarrow{\xi_{t_m}}
\hochcx_\bullet(\sym^m \A),$$
and analogously for higher $n$.  

The $K$-theoretic $\lambda$-ring construction in 
\cite[\S3]{Wang-EquivariantKTheoryWreathProductsAndHeisenbergAlgebra}
differs from our as follows. Unlike $HH_\bullet(\A)$, 
the topological $K$-theory $K^\bullet(X) \otimes \mathbb{C}$ has
has a natural $\lambda$-ring structure with the multiplication 
given by the tensor product of vector bundles and
$\lambda$-operations -- by the exterior powers. 
The corresponding Adams operations $\psi^m_{K(X)}\colon K^\bullet(X)
\rightarrow K^\bullet(X)$ are described using the decomposition 
\eqref{eqn-k-theoretic-symmetric-ring-structure-on-total-symm-k-theory}
as follows. 

Let $V$ be a vector bundle on $X$. Equip 
$V^{\boxtimes m} := V \boxtimes V \boxtimes \dots \boxtimes V$
on $X^m$ with the natural action of $S_m$ which acts
by permuting the factors. The assignment $V \mapsto V^{\boxtimes m}$
defines a map $\boxtimes m\colon K^\bullet(X) \rightarrow K^\bullet_{S_n}(X^n)$. 
Define 
\begin{small}
\begin{equation}
\label{eqn-defn-ktheoretic-adams-operations}
\psi_{K(X)}^m\colon K^\bullet(X) \otimes \mathbb{C} \xrightarrow{\boxtimes m} K^\bullet_{S_m}(X^m)
\otimes \mathbb{C} \simeq \sum_{\underline{\lambda} \vdash m} 
\Sym^{\underline{r}(\underline{\lambda})} K^\bullet(X) \otimes \mathbb{C} 
t^{\underline{\lambda}} \twoheadrightarrow K^\bullet(X) \otimes \mathbb{C} t^m,
\end{equation}
\end{small}
where the last map is the projection to the summand given by the
long cycle partition $(m)$ of $m$. Each 
$\psi_{K(X)}^m: K^\bullet(X) \otimes \mathbb{C} \rightarrow K^\bullet(X)
\otimes \mathbb{C}$ is an automorphism which is identified by the
Chern character with the automorphism of $H^*(X,\mathbb{C})$ which
acts by the multiplication by $m^{k}$ on $H^{2k}(X,\mathbb{C})$
and $H^{2k-1}(X,\mathbb{C})$.

The $\lambda$-ring structure constructed  
in \cite[\S3]{Wang-EquivariantKTheoryWreathProductsAndHeisenbergAlgebra}
on $\bigoplus_{n \geq 0} K^\bullet_{S_n}(X^n) \otimes \mathbb{C}$ is 
induced via the isomorphism 
$\bigoplus_{n \geq 0} K^\bullet_{S_n}(X^n) \otimes \mathbb{C}$
from the $\lambda$-ring structure on the symmetric algebra
$S^* \left(K^\bullet(X) \otimes_\mathbb{C} t\mathbb{C}[t]\right)$ 
whose Adams operations are defined by
$$ 
\psi^{m}(v t^n) := \psi^m_{K(X)}(v) t^{mn} \quad \quad \quad \forall\; 
v \in K^\bullet(X), n \geq 0. $$
Comparing this with the canonical Adams operations 
\eqref{eqn-free-lambda-ring-structure-on-symmetric-algebra-of-E-tkt} 
defining the free $\lambda$-ring generated by a graded
vector space, we see that the $\lambda$-ring structure constructed
on $S^* \left(K^\bullet(X) \otimes_\mathbb{C} t\mathbb{C}[t]\right)$ in
\cite[\S3]{Wang-EquivariantKTheoryWreathProductsAndHeisenbergAlgebra}
is not the canonical free one. However, since $\psi^m_{K(X)}$ are
automorphisms, it is isomorphic to the canonical one via the
$\lambda$-ring automorphism 
of $S^* \left(K^\bullet(X) \otimes_\mathbb{C} t\mathbb{C}[t]\right)$ 
which sends $K^\bullet(X) t^n$ to $K^\bullet(X) t^n$ via $\psi^n_{K(X)}$. 

As $\hochhom_\bullet(\A)$ has no natural $\lambda$-ring structure, 
we used the canonical Adams operations 
\eqref{eqn-free-lambda-ring-structure-on-symmetric-algebra-of-E-tkt} 
to define $\lambda$-ring structure on 
$S^*(\hochhom_\bullet(\A) \otimes t \kk[t])$, and hence on 
$\bigoplus_{n \geq 0} \hochhom_\bullet(\sym^n \A)$. Thus, 
in the commutative case, 
the HKR and Chern character isomorphism 
$\hochhom_\bullet(\A) \simeq K^\bullet(X) \otimes \mathbb{C}$ doesn't 
identify our $\lambda$-ring structure on 
$S^*(\hochhom_\bullet(\A) \otimes t \kk[t])$ with that
of 
\cite[\S3]{Wang-EquivariantKTheoryWreathProductsAndHeisenbergAlgebra}
on $S^* \left(K^\bullet(X) \otimes_\mathbb{C} t\mathbb{C}[t]\right)$. 
But could it be that the induced $\lambda$-ring structures on 
$\bigoplus_{n \geq 0} \hochhom_\bullet(\sym^n \A)$ and 
$\bigoplus_{n \geq 0} K^\bullet_{S_n}(X^n) \otimes \mathbb{C}$
are identified by the isomorphisms 
$\hochhom_\bullet(\sym^n \A) \simeq K^\bullet_{S_n}(X^n) \otimes \mathbb{C}$
cooked up using the orbifold HKR isomorphism of 
\cite[Theorem 1.15]{ArinkinCaldararuHablicsek-FormalityOfDerivedIntersectionsAndTheOrbifoldHKRIsomorphism}?

It is not an oversight that our $\lambda$-ring structure 
is not a noncommutative analogue of the one in
\cite[\S3]{Wang-EquivariantKTheoryWreathProductsAndHeisenbergAlgebra}. 
If it were, we would have constructed a noncommutative analogue on
$\hochhom_\bullet(\A)$ of the Adams operations $\psi^m_{K(X)}$ on 
$K^\bullet(X) \otimes \mathbb{C}$. Since the eigenspaces of $\psi^m_{K(X)}$
recover on $K^\bullet(X) \otimes \mathbb{C}$ 
the cohomological grading of $H^*(X,\mathbb{C})$, 
we would have constructed a noncommutative Hodge decomposition 
$H^{i,j}$ on $\hochhom_\bullet(\A)$. 

Curiously, though, we now have noncommutative analogues
for the classes in $\hochhom_0(\A)$ of the objects of $\A$
of both steps in the definition \eqref{eqn-defn-ktheoretic-adams-operations}.
The assignment $V \mapsto V^{\boxtimes m}$ 
corresponds to sending $a \in \A$ to the
$\hochhom_0$ class of the symmetric idempotent $\sum_{\sigma \in S_m}
\sigma$ of the object $a^{\otimes m}$ in $\sym^m \A$. 
This is precisely the Euler class of the $\sym^m \A$-module 
corresponding to the representable $\A^{\otimes m}$-module $a^{\otimes m}$
equipped with the $S_n$-equivariant structure which permutes its factors. 

On the other hand, the projection $K^\bullet_{S_m}(X^m) \otimes \mathbb{C}
\rightarrow K^\bullet(X) \otimes \mathbb{C}$ corresponds in our case to the composition 
\begin{align}
\label{eqn-projection-from-sym-m-to-A} 
\hochcx_\bullet(\sym^m \A) \xrightarrow{\nu}
\hochcx_\bullet(\A^{\otimes m};t_m)_{t_m} \xrightarrow{f_m}
\hochcx_\bullet(\A). 
\end{align}
This sends the symmetric idempotent 
$\sum_{\sigma \in S_m} \sigma$ of the object $a^{\otimes m}$ 
first to $\id_{a^{\otimes m}}$ in 
$\hochcx_\bullet(\A^{\otimes m};t_m)_{t_m}$, and then to 
$\id_a \in \hochcx_\bullet(\A)$. 

Thus replicating in the noncommutative case the steps
of the $K$-theoretic Adams operation $\psi_m$ produces
a trivial operation on the $\hochhom_0$ classes of the objects of $\A$. 
This agrees with our definition of the $\lambda$-structure on 
$S^*(\hochhom_\bullet(\A) \otimes t \kk[t])$ above. 
It is interesting whether the noncommutative analogue of 
the assignment $V \mapsto V^{\boxtimes m}$ can be extended to 
the whole of $\hochhom_\bullet(\A)$ and whether its composition 
with the projection \eqref{eqn-projection-from-sym-m-to-A} 
would be the identity map. 

\bibliography{references}

\providecommand{\bysame}{\leavevmode\hbox to3em{\hrulefill}\thinspace}
\providecommand{\MR}{\relax\ifhmode\unskip\space\fi MR }
\providecommand{\MRhref}[2]{%
  \href{http://www.ams.org/mathscinet-getitem?mr=#1}{#2}
}
\providecommand{\href}[2]{#2}
\begin{thebibliography}{KKP07}

\bibitem[ACH19]{ArinkinCaldararuHablicsek-FormalityOfDerivedIntersectionsAndTheOrbifoldHKRIsomorphism}
Dima Arinkin, Andrei C{\u a}ld{\u a}raru, and Marton Hablicsek, \emph{Formality
  of derived intersections and the orbifold {HKR} isomorphism}, J. Algebra
  \textbf{540} (2019), 100--120, arXiv:1412.5233.

\bibitem[AH61]{AtiyahHirzebruch-VectorBundlesAndHomogeneousSpaces}
Michael Atiyah and Friedrich Hirzebruch, \emph{Vector bundles and homogeneous
  spaces}, Differential {G}eometry (Carl~B. Allendoerfer, ed.), Proc. Symp.
  Pure Math., vol.~3, American Mathematical Society, 1961, pp.~7--38.

\bibitem[AL21]{AnnoLogvinenko-BarCategoryOfModulesAndHomotopyAdjunctionForTensorFunctors}
Rina Anno and Timothy Logvinenko, \emph{Bar category of modules and homotopy
  adjunction for tensor functors}, Int. Math. Res. Not. \textbf{2021} (2021),
  no.~2, 1353--1462, arXiv:1612.09530.

\bibitem[AL24]{AnnoLogvinenko-UnboundedTwistedComplexes}
\bysame, \emph{Unbounded twisted complexes}, J. Algebra \textbf{647} (2024),
  794--822, arXiv:2303.11796.

\bibitem[AS89]{AtiyahSegal-OnEquivariantEulerCharacteristics}
Michael Atiyah and Graeme Segal, \emph{On equivariant {E}uler characteristics},
  J. Geom. Phys. \textbf{6} (1989), no.~4, 671--677.

\bibitem[AT69]{AtiyahTall-GroupRepresentationsLambdaRingsAndTheJHomomorphism}
Michael Atiyah and David~Orme Tall, \emph{Group representations,
  $\lambda$-rings and the {J}-homomorphism}, Topology \textbf{8} (1969), no.~3,
  253--297.

\bibitem[Bar03]{Baranovsky-OrbifoldCohomologyAsPeriodicCyclicHomology}
Vladimir Baranovsky, \emph{Orbifold cohomology as periodic cyclic homology},
  Int. J. Math. \textbf{14} (2003), no.~08, 791--812.

\bibitem[Ber71]{Berthelot-GeneralitesSurLesLambdaAnneaux}
Pierre Berthelot, \emph{G{\'e}n{\'e}ralit{\'e}s sur les $\lambda$-{A}nneaux},
  Th{\'e}orie des {I}ntersections et {T}h{\'e}or{\`e}me de {R}iemann-{R}och
  (SGA 6), Lecture Notes in Math., no. 225, Springer-Verlag, 1971,
  pp.~297--364.

\bibitem[BFK23]{BelmansFuKrug-HochschildCohomologyOfHilbertSchemesOfPointsOnSurfaces}
Pieter Belmans, Lie Fu, and Andreas Krug, \emph{Hochschild cohomology of
  {H}ilbert schemes of points on surfaces}, arXiv:2309.06244, 2023.

\bibitem[Bry87]{Brylinski-CyclicHomologyAndEquivariantTheories}
Jean-Luc Brylinski, \emph{Cyclic homology and equivariant theories}, Annales de
  l'institut Fourier \textbf{37} (1987), no.~4, 15--28.

\bibitem[Că05]{Caldararu-TheMukaiPairingIITheHochschildKostantRosenbergIsomorphism}
Andrei Căldăraru, \emph{The {M}ukai pairing {II}: the
  {H}ochschild--{K}ostant--{R}osenberg isomorphism}, Advances in Mathematics
  \textbf{194} (2005), no.~1, 34--66.

\bibitem[Efi20]{Efimov-HomotopyFinitenessOfSomeDGCategoriesFromAlgebraicGeometry}
Alexander~I. Efimov, \emph{Homotopy finiteness of some {DG} categories from
  algebraic geometry}, J. Eur. Math. Soc. \textbf{22} (2020), no.~9,
  2897--2942, pre-print arXiv:1308.0135.

\bibitem[FT87]{FeiginTsygan-CyclicHomologyOfAlgebrasWithQuadraticRelationsUniversalEnvelopingAlgebrasAndGroupAlgebras}
Boris Feigin and Boris Tsygan, \emph{Cyclic homology of algebras with quadratic
  relations, universal enveloping algebras and group algebras}, K-Theory,
  Arithmetic and Geometry: Seminar, Moscow University, 1984--1986 (Yuri~I.
  Manin, ed.), Lecture Notes in Math., vol. 1289, Springer-Verlag, Berlin,
  Heidelberg, 1987, pp.~210--239.

\bibitem[GJ93]{GetzlerJones-TheCyclicHomologyOfCrossedProductAlgebras}
Ezra Getzler and John~D.S. Jones, \emph{The cyclic homology of crossed product
  algebras}, J. Reine Angew. Math. \textbf{445} (1993), 161--174.

\bibitem[GK14]{GanterKapranov-SymmetricAndExteriorPowersOfCategories}
Nora Ganter and Mikhail Kapranov, \emph{Symmetric and exterior powers of
  categories}, Transform. Groups \textbf{19} (2014), no.~1, 57--103,
  arXiv:1110.4753.

\bibitem[GKL21]{GyengeKoppensteinerLogvinenko-TheHeisenbergCategoryOfACategory}
{\'A}d{\'a}m Gyenge, Clemens Koppensteiner, and Timothy Logvinenko, \emph{The
  {H}eisenberg category of a category}, to appear in Mem. Am. Math. Soc.,
  (2021).

\bibitem[GL25]{GyengeLogvinenko-TheHeisenbergAlgebraOfAVectorSpaceAndHochschildHomology}
{\'A}d{\'a}m Gyenge and Timothy Logvinenko, \emph{The {H}eisenberg algebra of a
  vector space and {H}ochschild homology}, arXiv:2511.03649, (2025).

\bibitem[Gro71]{Grothendieck-ClassesDeFaisceauxEtTheoremeDeRiemannRoch}
Alexander Grothendieck, \emph{Classes de faisceaux et theoreme de
  {R}iemann-{R}och}, Th{\'e}orie des {I}ntersections et {T}h{\'e}or{\`e}me de
  {R}iemann-{R}och (SGA 6), Lecture Notes in Math., no. 225, Springer-Verlag,
  1971, pp.~20--77.

\bibitem[Gro95]{Grojnowski-InstantonsAndAffineAlgebrasITheHilbertSchemeAndVertexOperators}
Ian Grojnowski, \emph{Instantons and affine algebras {I}: {T}he {H}ilbert
  scheme and {V}ertex operators}, Math. Res. Lett. \textbf{3} (1995), no.~2,
  275--291, arXiv:alg-geom/9506020.

\bibitem[Kal07]{Kaledin-HomologicalMethodsInNoncommutativeGeometry}
Dmitri Kaledin, \emph{Homological methods in non-commutative geometry},
  \url{http://imperium.lenin.ru/~kaledin/tokyo/}, 2007.

\bibitem[Kel94]{Keller-DerivingDGCategories}
Bernhard Keller, \emph{Deriving {DG} categories}, Ann. Sci. {\'E}cole Norm.
  Sup. \textbf{27} (1994), no.~1, 63--102.

\bibitem[Kel06]{Keller-OnDifferentialGradedCategories}
\bysame, \emph{On differential graded categories}, Proceedings of the
  International Congress of Mathematicians, Madrid, Spain, 2006, vol.~II, Eur.
  Math. Soc., Z{\"u}rich, 2006, arXiv:math/0601185, pp.~151--190.

\bibitem[Kel20]{Keller-HochschildCohomologyAndDerivedCategories}
B.~Keller, \emph{Hochschild (co)homology and derived categories,}, expanded
  notes of lectures given at the Isfahan school and conference on
  representations of algebras in April 2019, available from author's homepage,
  (2020).

\bibitem[KKP07]{KatzarkovKontsevichPantev-HodgeTheoreticAspectsOfMirrorSymmetry}
Ludmil Katzarkov, Maxim Kontsevich, and Tony Pantev, \emph{Hodge theoretic
  aspects of mirror symmetry}, From Hodge Theory to Integrability and TQFT:
  tt*-geometry (Ron~Y. Donagi and Katrin Wendland, eds.), Proceedngs of
  Symposia in Pure Mathematics, vol.~78, Amer. Math. Soc., Providence, RI,
  2007, pp.~87--174.

\bibitem[Knu73]{Knutson-LambdaRingsAndTheRepresentationTheoryOfTheSymmetricGroup}
Donald Knutson, \emph{{Lambda-Rings and the Representation Theory of the
  Symmetric Group}}, Lecture Notes in Math., no. 308, Springer-Verlag, 1973.

\bibitem[Kon03]{Kontsevich-DeformationQuantizationOfPoissonManifolds}
Maxim Kontsevich, \emph{Deformation quantization of {P}oisson manifolds}, Lett.
  Math. Phys. \textbf{66} (2003), no.~3, 157--216.

\bibitem[KS09]{KontsevichSoibelman-NotesOnAInftyAlgebrasAInftyCategoriesAndNoncommutativeGeometry}
Maxim Kontsevich and Yan Soibelman, \emph{Notes on ${A}_\infty$-algebras,
  ${A}_\infty$-categories and {N}on-{C}ommutative {G}eometry}, Homological
  Mirror Symmetry, Lecture Notes in Physics, vol. 757, Springer Berlin
  Heidelberg, 2009, pp.~1--67.

\bibitem[LO10]{LuntsOrlov-UniquenessOfEnhancementForTriangulatedCategories}
Valery~A. Lunts and Dmitri~O. Orlov, \emph{Uniqueness of enhancement for
  triangulated categories}, J. Amer. Math. Soc. \textbf{23} (2010), 853--908,
  arXiv:0908.4187.

\bibitem[Lod97]{Loday-CyclicHomology}
Jean-Louis Loday, \emph{Cyclic homology}, Grundlehren der mathematischen
  {W}issenschaften, no. 301, Springer, 1997.

\bibitem[LS25]{FuPortaSibillaScherotzke-HKRIsomorphismForDerivedDM}
N.Sibilla L.Fu, M.Porta and S.Scherotzke, \emph{Hochschild-kostant-rosenberg
  isomorphism for derived deligne-mumford stacks}, preprint arxiv:2509.00501,
  (2025).

\bibitem[ML63]{MacLane-Homology}
Saunders Mac~Lane, \emph{Homology}, Die Grundlehren der mathematischen
  Wissenschaften, vol. 114, Springer-Verlag, Berlin, Heidelberg, New York,
  1963.

\bibitem[Nak97]{Nakajima-HeisenbergAlgebraAndHilbertSchemesOfPointsOnProjectiveSurfaces}
Hiraku Nakajima, \emph{Heisenberg algebra and {H}ilbert schemes of points on
  projective surfaces}, Annals of Mathematics \textbf{145} (1997), no.~2,
  379--388, arXiv:alg-geom/9507012.

\bibitem[Nor25a]{Nordstrom-HochschildHomologyOfSymmetricPowersOfDGcategories}
Ville Nordstr{\"o}m, \emph{A decomposition theorem for the hochschild homology
  of symmetric powers of a dg category}, arXiv:2511.03269, 2025.

\bibitem[Nor25b]{Nordstrom-FiniteGroupActionsOnDGCategoriesAndHochschildHomology}
\bysame, \emph{Finite group actions on {DG} categories and {H}ochschild
  homology}, Can. Math. Bull. (2025), 1--21.

\bibitem[Orl16]{Orlov-SmoothAndProperNoncommutativeSchemesAndGluingofDGcategories}
Dmitri Orlov, \emph{Smooth and proper noncommutative schemes and gluing of {DG}
  categories}, Adv. Math. \textbf{302} (2016), 59--105, arXiv:1402.7364.

\bibitem[Seg96]{Segal-EquivariantKTheoryAndSymmetricProducts}
Graeme Segal, \emph{Equivariant k-theory and symmetric products}, pre-print,
  1996.

\bibitem[Swa96]{Swan-HochschildCohomologyOfQuasiprojectiveSchemes}
Richard~G. Swan, \emph{Hochschild cohomology of quasiprojective schemes}, J.
  Pure Appl. Algebra \textbf{110} (1996), 57--80.

\bibitem[Swe69]{Sweedler-HopfAlgebras}
M.~Sweedler, \emph{Hopf algebras}, W.A. Benjamin Inc., 1969.

\bibitem[To{\"e}07]{Toen-TheHomotopyTheoryOfDGCategoriesAndDerivedMoritaTheory}
Bertrand To{\"e}n, \emph{The homotopy theory of {\em dg}-categories and derived
  {M}orita theory}, Invent. Math. \textbf{167} (2007), no.~3, 615--667,
  arXiv:math/0408337.

\bibitem[Wan00]{Wang-EquivariantKTheoryWreathProductsAndHeisenbergAlgebra}
Weiqiang Wang, \emph{{Equivariant $K$-theory, wreath products, and Heisenberg
  algebra}}, Duke Mathematical Journal \textbf{103} (2000), no.~1, 1--23.

\bibitem[Wil82]{Wilkerson-LambdaRingsBinomialDomainsAndVectorBundlesOverCPInfty}
Clarence Wilkerson, \emph{Lambda-rings, binomial domains, and vector bundles
  over $\mathbb{C}\mathbb{P}(\infty)$}, Communications in Algebra \textbf{10}
  (1982), no.~3, 311--328.

\bibitem[Yau10]{Yau-LambdaRings}
Donald Yau, \emph{Lambda-{R}ings}, World Scientific, 2010.

\end{thebibliography}
\bibliographystyle{amsalpha}
\end{document}